\documentclass[12pt]{amsbook}

\usepackage{pdfsync}
\usepackage{graphicx}
\usepackage{float}
\numberwithin{figure}{chapter}
\usepackage{graphicx,amssymb,amsfonts,latexsym,amsmath,amsthm,times, amscd, amsxtra, euscript, MnSymbol}

\usepackage{enumerate}
\newcommand{\A}{\mathbb{A}}
\newcommand{\B}{\mathbb{B}}

\newcommand{\J}{\mathbb{J}}
\newcommand{\N}{\mathbb{N}}

\newcommand{\Q}{\mathbb{Q}}
\newcommand{\K}{\mathbb{K}}

\newtheorem{definition}{{\bf Definition}}[chapter]
\newtheorem{theorem}{{\bf Theorem}}[chapter]
\newtheorem{theoreme}{{\bf Th\'eor\`eme}}[chapter]
\newtheorem{corollary}{{\bf Corollary}}[chapter]
\newtheorem{corollaire}{{\bf Corollaire}}[chapter]
\newtheorem{proposition}{\noindent {\bf Proposition}}[chapter]
\newtheorem{fact}{\noindent {\bf Fact}}[chapter]

\newtheorem{lemma}{\noindent {\bf Lemma}}[chapter]

\newtheorem{claim}{\noindent {\bf Claim}}[chapter]
\newtheorem{question}{\noindent {\bf Question}}[chapter]
\newtheorem{questions}[question]{Questions}
\newtheorem{example}{\noindent {\bf Example}}[chapter]

\newtheorem{notation}{\noindent {\bf Notation}}[chapter]
\newtheorem{problem}{Problem}[chapter]

\newtheorem{remark}{\noindent {\bf Remark}}[chapter]

\def\endproof{\hfill {\kern 6pt\penalty 500
   \raise -0pt\hbox{\vrule \vbox to5pt {\hrule width 5pt
  \vfill\hrule}\vrule}}}
\makeindex
\begin{document}

\title{Chains conditions in algebraic lattices}

\author{Ilham Chakir}
\vspace{10cm}

\address {D\'epartement de Math\'ematiques, Universit\'e Hassan $1^{er}$, Facult\'e des Sciences et Techniques, Settat, Maroc}

\email{chakirilham@yahoo.fr}

\maketitle

Ce  travail  repr\'esente la  "Th\`ese d'\'Etat" de l'auteur soutenue  le 11 Mai 2009 devant l'Universit\'e Mohamed V, Facult\'e des Sciences, Rabat (Maroc). 
Le jury \'etait  compos\'e des Professeurs  A. Alami Idrissi, Fac.  des Sciences, Rabat (Maroc), Pr\'esident;  M. Kabil,  F.S.T Mohammedia  (Maroc), Membre;  D. Misane, Fac. des Sciences, Rabat (Maroc), Membre et  Rapporteur;  
M. Pouzet, Universit\'e Claude-Bernard Lyon 1  (France), Membre et Rapporteur; 
I. Rosenberg, Universit\'e de Montr\'eal, (Canada),  Rapporteur; M. Sobrani, F.S.T F\`es-Sa\"{\i}s (Maroc), Membre;  N. Zaguia, Universit\'e d'Ottawa  (Canada), Rapporteur.

\vspace{3cm}

Le texte comprend deux r\'esum\'es -l'un en anglais, l'autre en fran\c{c}ais; une introduction en fran\c{c}ais, suivie du corps de la th\`ese form\'ee  de quatre chapitres en anglais.   

\newpage

\begin{center} {\huge Remerciements}
\end{center}
\vspace{1.5cm}
	Je tiens ‡ remercier tout d'abord le Professeur ALAMI IDRISSI Ali  de m'avoir fait l'honneur de pr\'esider le jury.

Mes remerciements les plus profonds et toute ma gratitude vont au Professeur Maurice POUZET. Aupr\`es de lui, j'ai appris et approfondi mes connaissances en math\'ematiques et au del\`a. Je lui sais gr\'e  \'egalement  de tous les efforts qu'il a consenti dans mon travail. Son soutien, son accueil chaleureux  et surtout son amiti\'e m'ont profond\'ement touch\'ee. Sa participation  au  jury me comble de plaisir.

Je suis reconnaissante au Professeur Driss MISANE de s'\^etre int\'eress\'e \`a mon travail en r\'edigeant un rapport. Son aide et sa disponibilit\'e m'ont \'et\'e utiles et sa participation au jury me fait honneur.

Je remercie chaleureusement les Professeurs Ivo ROSENBERG  et  Nejib ZAGUIA qui se sont int\'eress\'es \`a ce travail et ont eu l'amabilit\'e de r\'ediger un rapport de pr\'e-soutenance.

J'exprime mes vifs remerciements au Professeur Mohamed SOBRANI qui m'a fait l'honneur de participer au jury.

Je suis redevable au Professeur Mustapha KABIL dont l'aide et les conseils m'ont \'et\'e pr\'ecieux. Je le remercie \'egalement pour l'honneur qu'il me fait de participer au jury.

Je remercie Ègalement l'ensemble du Laboratoire de Probabilit\'es, Combinatoire et Statistique  \`a Lyon au sein duquel j'ai finalis\'e ce travail et dont l'aide logistique m'a \'et\'e d'un grand soutien.

Je remercie aussi les membres du D\'epartement de Math\'ematiques de la Facult\'e des Sciences et Techniques de Settat, sp\'ecialement le Professeur El Moustapha EL KHOUZAI.

Enfin, je ne trouve pas les mots juste pour exprimer ma gratitude \`a Monsieur Mamoun  DRIBI,  directeur de la soci\'et\'e GHESS, qui m'a aid\'ee, conseill\'ee et encourag\'ee, a mis \`a ma disposition bureau, mat\'eriel et un cadre agr\'eable de travail. Mlle Souad ALIOUA dont le sourire, l'accueil et la serviabilit\'e ont accompagn\'e ce travail.

\tableofcontents

\bibliography{articles,books,thesis}

\chapter*{Abstract}
This work studies the relationship between the chains 
\index{chain} of an  algebraic lattice  and the order structure of the join-semilattice \index{join-semilattice}
 of its compact elements\index{compact}.  The results  are presented into four chapters, each corresponding to a paper  written in collaboration with 
Maurice Pouzet. \begin{enumerate}
\item A characterization of well-founded algebraic lattices,  19p. \\arXiv:0812.2300.
\item The length of chains in algebraic lattices, Les annales ROAD du LAID3, special issue 2008, pp 379-390 (proceedings of ISOR'08, Algiers, Algeria, Nov 2-6, 2008). 
\item The length of chains in algebraic modular lattices, Order, 24(2007) 224-247.
 \item Infinite independent sets  in distributive lattices, Algebra Universalis, 53(2005) 211-225.
\end{enumerate}

Our first studies on this theme have appeared in our doctoral thesis \cite {chak} presented in Lyon in 1992. A part of our results is included in Chapter  3 and  4 of the present work. 

Here are our main results (the necessary definitions can be found in the last section of this volume).

We show that for \emph{every order type  $\alpha$ there is a list  $\mathbb{B}_{\alpha}$ of join-semilattices, with cardinality at most  
$2^{\mid\alpha\mid}$,  such that an algebraic lattice $L$ 
contains a chain of order type 
 $I(\alpha)$ if and only if the join-subsemilattice $K(L)$ of its compact elements  contains a join-semilattice isomorphic to a  member of $\mathbb{B}_{\alpha}$.}
(Theorem \ref {thmfirst},  Chapter \ref{chap:wellfoundedbis}).

We conjecture that when  $\alpha$ is countably infinite, there is a finite list. 
The following result supports this conjecture: 
Let  $[\omega]^{<\omega}$ be  the set of  finite subsets of  $\omega$ ordered by inclusion. Then, \emph{among the join-subsemilattices of  $[\omega]^{<\omega}$ belonging to $\B_{\alpha}$, one embeds  in all others as a join-semilattice} (cf. Theorem \ref{thm:mainalgebra}, Chapter \ref{chap:wellfoundedbis}).

We also show that
\emph{an algebraic lattice  $L$  is well-founded if and only if the join-semilattice $K(L)$ of compact elements $L$ is well-founded 
and contains no  join-semilattice isomorphic to
$\underline \Omega(\omega^*)$ or to  $[\omega]^{<\omega}$}
(Theorem \ref {thm4}, Chapter \ref{chap:wellfounded}).

We describe the countable  indivisible order types $\alpha$ such that: \emph{for every modular algebraic lattice $L$,   $L$ contains no chain of  order type $\alpha$ if and only if  the join-semilattice of its compact elements  contains neither $\alpha$ nor a join-semilattice isomorphic to  $[\omega]^{<\omega}$}
(Theorem \ref {thm6},  Chapter \ref{chap:algebraic}). The chains $\omega^*$ et $\eta$ are among these order types $\alpha$.

We identify  two meet-semilattices  $\Gamma$ et $\Delta$.  We show  (Theorem \ref{thm8}, Chapter \ref{chap:TD}) that: \emph{for every  distributive lattice $T$, the following properties are equivalent: \begin{enumerate} [{(i)}] \item $T$ contains a join-subsemilattice isomorphic to  $[\omega]^{<\omega}$;
\item  The lattice
$[\omega]^{<\omega}$ is a quotient of a sublattice of  $T$;
\item  $T$ contains a sublattice isomorphic to 
$I_{<\omega}(\Gamma)$ ou \`a $I_{<\omega}(\Delta)$.\end{enumerate}}

This abstract is followed by  its  french version and by an introduction in french. The four chapters are in english. \\


\chapter*{R\'esum\'e}
Cette th\`ese porte sur le rapport entre la longueur des
cha\^{\i}nes \index{cha\^{\i}ne} d'un treillis alg\'ebrique et la
structure d'ordre du sup-treillis \index{sup-treillis} de ses
\'el\'ements compacts. \index{compact} Les r\'esultats obtenus ont
donn\'e lieu \`a quatre articles \'ecrits en collaboration avec
Maurice Pouzet  et constituant chacun un
chapitre de la th\`ese:
\begin{enumerate}
\item A characterization of well-founded algebraic lattices, 19p. \\ arXiv:0812.2300.
\item The length of chains in algebraic lattices, Les annales ROAD du LAID3, special issue 2008, pp 379-390 (proceedings of ISOR'08, Algiers, Algeria, Nov 2-6, 2008). 
\item The length of chains in algebraic modular lattices, Order, 24(2007) 224-247.
 \item Infinite independent sets  in distributive lattices, Algebra Universalis, 53(2005) 211-225.
\end{enumerate}

Nos premi\`eres recherches sur ce th\`eme sont apparues dans notre th\`ese de doctorat \cite {chak} Lyon 1992. Une partie d'entre elles est incluse dans les chapitres 3 et 4 du pr\'esent travail. Voici nos principaux r\'esultats (on pourra trouver en appendice les d\'efinitions n\'ecessaires).

 Nous montrons que: \emph{si $\alpha$ est un type d'ordre,  il existe une liste de
sup-treillis, soit $\mathbb{B}_{\alpha}$, de taille au plus
$2^{\mid\alpha\mid}$,  telle que quel que soit le  treillis alg\'ebrique $L$,
$L$  contient une cha\^{\i}ne de
type $I(\alpha)$ si et seulement si le sous sup-treillis $K(L)$ de ses \'el\'ements compacts contient un sous sup-treillis
isomorphe \`a un membre de $\mathbb{B}_{\alpha}$.}
(Theorem \ref {thmfirst},  Chapter \ref{chap:wellfoundedbis}).

Nous conjecturons que si $\alpha$ est d\'enombrable il existe une liste finie.

Consid\'erant l'ensemble $[\omega]^{<\omega}$ form\'e des parties finies de $\omega$ et ordonn\'e par inclusion,  nous montrons \`a l'appui de cette conjecture que: \emph{parmi les sous sup-treillis de $[\omega]^{<\omega}$ appartenant \`a $\B_{\alpha}$, l'un d'eux se plonge comme sous sup-treillis dans tous les autres} (cf. Theorem \ref{thm:mainalgebra}, Chapter \ref{chap:wellfoundedbis}).

Comme r\'esultat positif, nous montrons que:
\emph{un treillis alg\'ebrique $L$  est bien fond\'e si est seulement si
le sup-treillis $K(L)$ des \'el\'ements compacts de $L$ est bien
fond\'e et ne contient pas de sous sup-treillis  isomorphe \`a
$\underline \Omega(\omega^*)$ ou \`a $[\omega]^{<\omega}$}
(Theorem \ref {thm4}, Chapter \ref{chap:wellfounded}).

Nous d\'ecrivons les types d'ordre indivisibles d\'enombrables $\alpha$ tels que: \emph{ quel que soit le treillis  alg\'ebrique modulaire $L$,   $L$ ne contient pas de cha\^{\i}ne de type $\alpha$ si et seulement si le sup-treillis de ses \'el\'ements compacts  ne contient ni $\alpha$ ni de  sous sup-treillis  isomorphe \`a $[\omega]^{<\omega}$}
(Theorem \ref {thm6},  Chapter \ref{chap:algebraic}). Parmi eux figurent
$\omega^*$ et $\eta$.

Nous identifions deux inf-treillis $\Gamma$ et $\Delta$.  Nous  montrons (Theorem \ref{thm8}, Chapter \ref{chap:TD}) que: \emph{pour un treillis distributif $T$,  les propri\'et\'es suivantes
sont \'equivalentes: \begin{enumerate} [{(i)}] \item $T$ contient un sous sup-treillis isomorphe \`a $[\omega]^{<\omega}$;
\item  Le treillis
$[\omega]^{<\omega}$ est quotient d'un sous-treillis de $T$;
\item  $T$ contient un sous-treillis isomorphe \`a
$I_{<\omega}(\Gamma)$ ou \`a $I_{<\omega}(\Delta)$.\end{enumerate}}

\chapter*{Introduction}
La notion d'op\'erateur  de fermeture est centrale en math\'ematiques. Elle est \'egalement  un outil de mod\'elisation dans plusieurs sciences appliqu\'ees comme l'informatique et  les sciences sociales (eg. logique et  mod\'elisation du calcul, bases de donn\'ees relationnelles, fouille de donn\'ees...). Dans beaucoup d'exemples concrets, les op\'erateurs de fermeture satisfont une propri\'et\'e d'engendrement fini, et on dit que ceux-ci sont alg\'ebriques.  \`A un op\'erateur de fermeture est associ\'e un treillis complet, le treillis de ses ferm\'es. Et  ce treillis est alg\'ebrique si et seulement si la fermeture est alg\'ebrique. Du fait de l'importance de ces op\'erateurs de fermeture, les treillis alg\'ebriques \index{treillis alg\'ebrique} sont des
objets privil\'egi\'es de la  th\'eorie des treillis. \index{treillis}

Un exemple de base est l'ensemble, ordonn\'e par inclusion, des sections initiales d'un ensemble ordonn\'e (ou  pr\'eordonn\'e).  C'est celui qui motive notre recherche. Son importance vient, pour nous, de la th\'eorie des relations et de la logique. Les structures relationnelles \'etant pr\'eordonn\'ees par abritement, les sections initiales correspondent aux classes de mod\`eles finis des th\'eories universelles tandis que les id\'eaux sont les \^ages de structures relationnelles introduits par Fra{\"\i}ss{\'e} \cite{fraissetr}. Et toute une approche de l'\'etude des structures relationnelles finies peut \^etre exprim\'ee en termes d'ordre et tout particuli\`erement en termes de ce treillis des sections initiales   \cite{pouzetcondch, pouzettr, pou-sob, sobrani}.

Notre travail porte sur la fa\c{c}on dont les cha\^{\i}nes d'un treillis alg\'ebrique  $L$ se refl\`etent dans la structure du sup-treillis  $K(L)$ constitu\'e des \'el\'ements compacts de ce treillis. L'objectif \'etant de d\'eterminer les sup-treillis que $K(L)$ doit contenir pour assurer que  $L$ contienne une cha\^{\i}ne  d'un type donn\'e $\alpha$.

Pour la compr\'ehension de ce qui suit, rappelons qu'un  treillis $L$ est {\it alg\'ebrique} s'il est
complet et si tout \'el\'ement est  supr\'emum d'\'el\'ements
compacts. L'ensemble $K(L)$ des \'el\'ements compacts de $L$ est
un {\it sup-treillis} (c'est \`a dire un ensemble ordonn\'e
\index{ensemble ordonn\'e} dans lequel  toute paire d'\'el\'ements
a un supr\'emum) admettant un plus petit \'el\'ement  et, de plus, l'ensemble
$J(K(L))$ des id\'eaux \index{id\'eal} de $K(L)$ est un treillis
complet \index{treillis complet}isomorphe \`a $L$.
R\'eciproquement, si $P$ est un sup-treillis ayant un plus petit \'el\'ement,
l'ensemble $J(P)$ des id\'eaux de $P$ est un treillis alg\'ebrique
dont l'ensemble des \'el\'ements compacts est isomorphe \`a $P$.
Ainsi, pour revenir \`a l'exemple qui motive notre recherche, si $L$ est l'ensemble des sections initiales d'un ensemble ordonn\'e $Q$, alors $K(L)$ est \'egal \`a l'ensemble $I_{<\omega}(Q)$ des sections initiales de $Q$ qui sont  finiment engendr\'ees. Et $I(Q)$ est isomorphe \`a  $J(I_{<\omega}(Q))$.

Exprim\'e en termes de sup-treillis et d'id\'eaux, notre travail  porte  donc sur le rapport entre la
longueur des cha\^{\i}nes d'id\'eaux d'un sup-treillis admettant un plus petit \'el\'ement et la structure d'ordre de ce sup-treillis.

Il s'inspire d'un rapport entre longueur des cha\^{\i}nes d'id\'eaux d'un ensemble ordonn\'e $P$  quelconque et structure d'ordre de cet ensemble ordonn\'e constat\'e par Pouzet et  Zaguia, 1985 \cite{pz}.

\'Etant donn\'ee une cha\^{\i}ne de type $\alpha$, d\'esignons par $I(\alpha)$ le type d'ordre de l'ensemble des sections initiales de cette cha\^{\i}ne. Disons que $\alpha$  est \emph{ind\'ecomposable}\index{ind\'ecomposable} si la cha\^{\i}ne  s'abrite dans chacune de ses sections  finales non-vides.  \index{d\'enombrable}

 Pouzet et Zaguia obtiennent  le r\'esultat suivant (cf. \cite {pz}  Theorem 4 pp. 162 et Theorem \ref{thm2},  Chapter \ref{chap:wellfoundedbis}).

\begin{theoreme}\label{thm:posetintro}
Soit $\alpha$ un type d'ordre \index{type d'ordre}  ind\'ecomposable d\'enombrable.

Il existe une liste finie $A_{1}^{\alpha}, \ldots,
A_{n_\alpha}^{\alpha}$ d'ensembles ordonn\'es, tels que pour tout
ensemble ordonn\'e $P$, l'ensemble $J(P)$ des id\'eaux de $P$
ne contient pas de  cha\^{\i}ne de type $I(\alpha)$ si et seulement si
$P$ ne contient aucun  sous-ensemble isomorphe \`a un des $A_
{1}^{\alpha}, \ldots, A_{n_\alpha}^{\alpha}$.
\end{theoreme}
Les  $A_{1}^{\alpha}, \ldots, A_{n_\alpha}^{\alpha}$ sont les "obstructions" \`a l'existence  d'une cha\^{\i}ne d'id\'eaux de type $I(\alpha)$. Une  obstruction typique s'obtient par  \emph{sierpinskisation monotone}: on se donne une  bijection $\varphi$ des entiers sur une cha\^{\i}ne, dont le type  $\omega\alpha$ est la somme de $\alpha$ copies de la cha\^{\i}ne $\omega$ des entiers, telle que $\varphi^{-1}$ soit croissante sur chaque $\omega.\{\beta\}$ pour $\beta\in \alpha$. On  ordonne $\N$ en posant $x\leq y$ si $x\leq y$ dans l'ordre naturel et  $\varphi(x)\leq \varphi(y)$ dans l'ordre de  $\omega\alpha$.  On constate que  l'ensemble  ordonn\'e obtenu, augment\'e d'un  plus petit \'el\'ement, s'il n'en a pas, contient une cha\^{\i}ne d'id\'eaux de type  $I(\alpha)$. En outre, tous les ensembles ordonn\'es obtenus par ce proc\'ed\'e se plongent les uns dans les autres; de ce fait,  on les d\'esigne par le m\^eme symbole $\underline {\Omega}(\alpha)$ (cf. Lemma 3.4.3 pp. 167 \cite{pz}). Si $\alpha$ n'est ni $\omega$, ni $\omega^*$, le dual de $\omega$ et ni $\eta$ (le type d'ordre des rationnels), les autres obstructions s'obtiennent au moyen de sommes lexicographiques d'obstructions correspondant \`a des  cha\^{\i}nes  de type strictement inf\'erieur \`a $\alpha$.

Il est naturel de se demander ce que devient le r\'esultat ci-dessus si, au lieu d'ensembles
ordonn\'es,  on consid\`ere des sup-treillis.  Cette  question est au coeur de notre travail.

Afin de raccourcir la suite de l'expos\'e, formalisons un peu:

D\'esignons par  $\mathbb{E}$ la classe des ensembles ordonn\'es. \'Etant donn\'es $P,P'\in \mathbb{E}$ disons que $P$ s'\emph{abrite} dans $P'$ et notons $P\leq P'$ si $P$ est isomorphe \`a une partie de $P'$. Cette relation est un pr\'eordre.  Deux ensembles ordonn\'es $P$ et $P'$ tels que $P\leq P'$ et $P'\leq P$ sont dits \emph{\'equimorphes}\index{\'equimorphe}.
Pour un type d'ordre $\alpha$, d\'esignons   par
$\mathbb{E}_{\neg\alpha}$ la classe form\'ee des \'el\'ements $P\in \mathbb{E}$
tels que $J(P)$ n'abrite pas de cha\^{\i}ne de type $I(\alpha)$. Pour une partie $\mathbb {B} $ de $\mathbb{E}$ d\'esignons par
$\uparrow\mathbb{B}$ la classe des $P\in \mathbb{E}$ abritant un \'el\'ement de $\mathbb {B}$ et
posons $Forb(\mathbb{B}):= \mathbb E \setminus \uparrow\mathbb{B}$
(" Forb" pour "forbidden"). Le r\'esultat ci-dessus s'\'ecrit
\begin{equation}\label{eq: forbposet} \mathbb{E}_{\neg\alpha}=Forb(\{A_{1}^{\alpha}, \ldots,
A_{n_\alpha}^{\alpha}\}).
\end{equation}
Rempla\c{c}ons la classe
$\mathbb E$ par  la classe $\mathbb{J}$ des sup-treillis $P$ admettant un plus petit \'el\'ement. Pour un type d'ordre $\alpha$, d\'esignons par
$\mathbb{J}_{\alpha}$ la classe des $P$ dans  $\mathbb{J}$ tels
que $J(P)$ abrite une cha\^{\i}ne de type $I(\alpha)$ et posons $\mathbb{J}_{\neg\alpha}:= \J\setminus \mathbb{J}_{\alpha}$. Pour une partie $\mathbb{B}$ de $\mathbb J$,
d\'esignons par    $\uparrow\mathbb{B}$  la  classe des $P$ dans
$\mathbb{J}$ contenant un sous sup-treillis (et non un
sous-ensemble  ordonn\'e) isomorphe \`a un membre de $\mathbb{B}$
et  posons $Forb_{\J}(\mathbb{B}):= \mathbb J \setminus
\uparrow\mathbb{B}$.

Nous pouvons pr\'eciser notre question ainsi:

\begin{questions}\label{quest: basic}

\begin{enumerate}
\item Pour un type d'ordre $\alpha$,  quelles sont les parties  $\mathbb{B}$ de
$\mathbb{\J}$ les plus simples possible telles que:
\begin{equation}\label{eq:forblattice}
\mathbb
{J}_{\neg\alpha}=Forb_{\J}(\mathbb{B}).
\end{equation}
\item Est ce que pour toute   cha\^{\i}ne d\'enombrable $\alpha$ on peut trouver $\mathbb{B}$ fini tel que $\mathbb {J}_{\neg\alpha}=Forb_{\J}(\mathbb{B})$?
\end{enumerate}
\end{questions}
Notre travail est consacr\'e   \`a ces questions.  Nous n'y r\'epondons que partiellement.

Pour r\'epondre \`a la deuxi\`eme question,  il est tentant d'utiliser le Th\'eor\`eme \ref{thm:posetintro}.
En effet, si l'ensemble $J(P)$
des id\'eaux d'un sup-treillis $P$ contient une cha\^{\i}ne de
type $I(\alpha)$ alors $P$
 doit contenir,  comme ensemble ordonn\'e,  un des $A_{i}^{\alpha}$. Et par
cons\'equent, $P$ doit contenir, comme sous sup-treillis, le
sup-treillis engendr\'e par $A_{i}^{\alpha}$ dans $P$.  Mais nous ne savons pas d\'ecrire les sup-treillis que l'on peut
engendrer au moyen des $A_{i}^{\alpha}$. Notre approche consiste \`a adapter, quand cela para\^{\i}t possible, la preuve du Th\'eor\`eme \ref{thm:posetintro}.
Nous pr\'esentons ci-dessous les r\'esultats obtenus.

\section{R\'esultats g\'en\'eraux}
Notons imm\'ediatement que si
$\alpha$ est le type d'une cha\^{\i}ne finie on peut prendre $\mathbb{B}=\{1+\alpha \}$. Ceci est encore vrai si $\alpha$ est la  cha\^{\i}ne $\omega$ des entiers. Un cas plus int\'eressant est celui de la  cha\^{\i}ne $\omega^*$
des entiers n\'egatifs. Les ensembles ordonn\'es n'abritant pas la cha\^{\i}ne $\omega^*$ sont dits \emph{bien fond\'es};  ensembles ordonn\'es dont toute partie non vide contient un \'el\'ement minimal, ils servent \`a mod\'eliser le raisonnement par induction.

Notons $\Omega(\omega^*)$ l'ensemble $[
\omega]^2$ des parties \`a deux \'el\'ements de $\omega$, qu'on
identifie aux paires  $(i,j)$, $i<j<\omega$, muni de l'ordre
suivant: $(i,j)\leq (i',j')$ si et seulement si $i'\leq i$ et
$j\leq j'$. Soit $\underline\Omega(\omega^*):=\Omega(\omega^*)
\cup \{\emptyset \}$ obtenu en ajoutant un plus petit \'el\'ement.
Notons $[\omega]^{<\omega}$ l'ensemble, ordonn\'e par inclusion,
des parties finies de $\omega$.

\begin{figure}
\begin{center}
\includegraphics[width=2.5in]{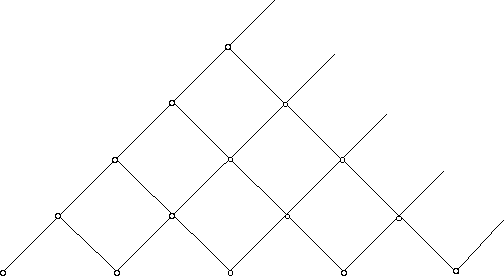}
\caption{$\Omega(\omega^*)$}
\end{center}
\label{fig:omega}
\end{figure}

Les ensembles $\underline\Omega(\omega^*)$ et $[\omega]^{<\omega}$
sont des treillis bien-fond\'es, tandis que  les treillis
alg\'ebriques $J(\underline \Omega(\omega^*))$  et
$J([\omega]^{<\omega})$
 ne le sont pas
(par exemple  $J([\omega]^{<\omega})$ est isomorphe \`a $\mathfrak
{P}(\omega)$, l'ensemble des parties de $\omega$). Ces deux
ensembles sont incontournables,
en effet:

\begin{equation}\label{eq:forbidomegastar}\J_{\neg \omega^*}=Forb_{\J}(\{1+\omega^*, \underline
{\Omega}(\omega^*), [\omega]^{<\omega}\}).\end{equation}

C'est le premier r\'esultat significatif de ce travail. Il se reformule de fa\c{c}on plus accessible comme suit:

 \begin{theoreme}\label{b-fond1}(cf. Theorem \ref {thm4},  Chapter \ref{chap:wellfounded})
Un treillis alg\'ebrique $L$  est bien fond\'e si est seulement si
le sup-treillis $K(L)$ des \'el\'ements compacts de $L$ est bien
fond\'e et ne contient pas de sous sup-treillis  isomorphe \`a
$\underline \Omega(\omega^*)$ ou \`a $[\omega]^{<\omega}$.
\end{theoreme}

Notons que la  liste donn\'ee par le
Th\'eor\`eme \ref{thm:posetintro} dans le cas $\alpha= \omega^*$ a un terme de moins. On a en effet (cf.
\cite {pz} Theorem 1. pp.160):

\begin{equation}\mathbb E_{\neg \omega^*}=Forb\{\omega^*, \underline \Omega(\omega^*)\}
\end{equation}

Ceci tient au fait que $\underline\Omega(\omega^*)$ est isomorphe
comme ensemble ordonn\'e \`a  un sous-ensemble de
$[\omega]^{<\omega}$, mais non pas comme sup-treillis.

Apr\`es les cha\^{\i}ne finies, les cha\^{\i}nes $\omega$ et $\omega ^*$,  une cha\^{\i}ne typique est  la cha\^{\i}ne $\eta$ des rationnels.   Les ensembles ordonn\'es n'abritant pas la cha\^{\i}ne  des rationnels
sont dits \emph{dispers\'es}.

Notons $\underline\Omega(\eta)$ l'ensemble des couples  d'entiers $(n, \frac{i}{2^n})$ tels que $0\leq i<2^n$ ordonn\'e de sorte que:  $(n, \frac{i}{2^n})\leq (m, \frac{j}{2^m})$ si $n\leq m$ et $\frac{i}{2^n}\leq \frac{j}{2^m}$. Ainsi  $\underline \Omega(\eta)$ est un sous-ensemble du produit de la cha\^{\i}ne $\omega$ et de la cha\^{\i}ne $D$ des nombres dyadiques de l'intervalle $[0, 1[$. C'est en fait un sous sup-treillis de ce produit.

\begin{figure}[htbp]
\centering
\includegraphics[width=3in]{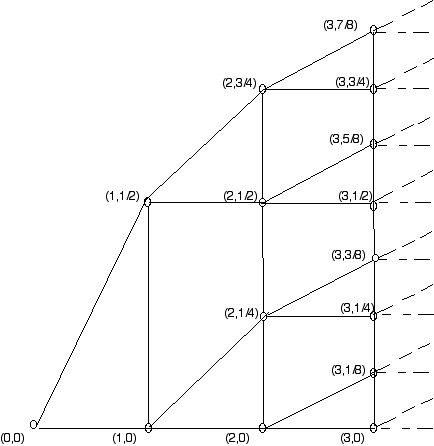}
\caption{$ \underline \Omega(\eta)$}
\label{Omega}
\end{figure}

Dans \cite {pz} (Theorem 2. pp.161) il est prouv\'e que:

\begin{equation}\mathbb E_{\neg \eta}=Forb\{\eta, \underline \Omega(\eta)\}
\end{equation}

Nous ne savons pas r\'epondre  \`a la question suivante:
\begin{question} Est-ce que $\J_{\neg\eta}= Forb_{\J}(\{1+\eta, [\omega]^{<\omega}, \underline \Omega(\eta)\} )$?
\end{question}

%
%

 Si $\alpha$ est un type d'ordre arbitraire, nous obtenons le r\'esultat suivant (analogue \`a celui d\'ej\`a obtenu en \cite{pz}, cf.Theorem 3 pp.162).

\begin{theoreme}\label{liste} (cf. Theorem \ref {thmfirst}, Chapter \ref{chap:wellfoundedbis})
Soit $\alpha$ un type d'ordre. Il existe une liste de
sup-treillis, soit $\mathbb{B}_{\alpha}$, de taille au plus
$2^{\mid\alpha\mid}$,  telle que pour tout sup-treillis $P$,
l'ensemble $J(P)$ des id\'eaux de $P$ contient une cha\^{\i}ne de
type $I(\alpha)$ si et seulement si $P$ contient un sup-treillis
isomorphe \`a un membre de $\mathbb{B}_{\alpha}$.
\end{theoreme}

C'est tr\`es loin d'une r\'eponse positive \`a la seconde de nos questions, \`a savoir l'existence  d'une liste finie lorsque  $\alpha$ est d\'enombrable.

Si cette question  a une r\'eponse  positive alors il y a une seule liste de cardinal minimum,  pourvu que des sup-treillis $P$, $P'$ qui se plongent  l'un
dans l'autre comme sup-treillis soient identifi\'es.  En effet,  pr\'eordonnons $\J$ en posant $P\leq P'$ si $P$  se plonge dans $P'$ comme sous  sup-treillis. Dans  l'ordre quotient, $\mathbb
J_{\alpha}$ est une section finale. Et si $\mathbb B$ est une liste finie de cardinal minimum alors dans ce quotient c'est
 l'ensemble des \'el\'ements minimaux de $\mathbb
J_{\alpha}$.

 En fait notre seconde question \'equivaut \`a:

 \begin{question}
 \begin{enumerate}
 \item  Est-ce que $\mathbb{J}_{\alpha}$
admet un nombre fini d'\'el\'ements minimaux?
\item  Est-ce que tout
\'el\'ement de $\mathbb {J}_{\alpha}$ majore un \'el\'ement
minimal?
\end{enumerate}
\end{question}

Par exemple $1+\alpha\in \mathbb{J}_{\alpha}$.
En effet  $J(1+\alpha)=1+J(\alpha)=I(\alpha)$. De plus $1+\alpha$ est
minimal dans $\mathbb{J}_{\alpha}$. En effet, soit $Q\in
\mathbb{J}$ tel que $Q\leq 1+\alpha$, on a $Q=1+\beta$. Comme
$J(Q)=I(\beta)$, si $Q\in \mathbb{J}_{\alpha}$, $I(\alpha)\leq
I(\beta)$. Ce qui implique $\alpha\leq\beta$ et $1+\alpha\leq Q$.

Lorsque $\alpha$ est d\'enombrable, le treillis
$[\omega]^{<\omega}$ joue un r\^{o}le central. Nous avons:\begin{equation}[\omega]^{<\omega}\in
\mathbb{J}_{\alpha}.
\end{equation}

 En effet $J([\omega]^{<\omega})$ est isomorphe \`a
$\mathfrak{P}(\omega)$ lequel contient une cha\^ine de type $I(\alpha)$ pour tout $\alpha$ d\'enombrable.

Nous montrons que vis a vis du pr\'eordre sur $\J$,  l'ensemble des sous sup-treillis de $[\omega]^{<\omega}$ qui appartiennent \`a $\J_{\alpha}$ a un plus petit \'e\'l\'ement:
\begin{theoreme} \label{thm:mainalgebraintro}(cf. Theorem \ref{thm:mainalgebra}, Chapter \ref{chap:wellfoundedbis}) Pour tout $\alpha$ d\'enombrable, $\J_{\alpha}$ contient un  sup-treillis $Q_{\alpha}$ qui se plonge comme sous sup-treillis dans tout sous sup-treillis de  $[\omega]^{<\omega}$ appartenant \`a $\J_{\alpha}$.
Ce treillis $Q_{\alpha}$ \'etant \'egal \`a:
\begin{enumerate}[${\bullet}$]
\item $1+\alpha$ si $\alpha$ est fini.
\item $I_{<\omega}(S_{\alpha})$ o\`u $S_{\alpha}$ est une sierpinskisation de $\alpha$ et de $\omega$ si $\alpha$ est infini  d\'enombrable.
\end{enumerate}
\end{theoreme}
Ce r\'esultat est cons\'equence des deux th\'eor\`emes suivants:
\begin{theoreme}\label{minimal} (cf. Theorem \ref{ordinal},
Chapter \ref{chap:wellfoundedbis}) Soit $\alpha$ un type d'ordre d\'enombrable. Le
sup-treillis $[\omega]^{<\omega}$ est minimal dans
$\mathbb{J}_{\alpha}$ si et seulement si $\alpha$ n'est pas un
ordinal.
\end{theoreme}

Soit  $\alpha$ un ordinal.   Soit $S_{\alpha}:=\alpha$ si $\alpha<\omega$. Si  $\alpha=\omega\alpha'+n$ avec  $\alpha'\not =0$ et $n<\omega$,  soit   $S_{\alpha}:=\Omega(\alpha')\oplus n$  la somme directe  de $\Omega(\alpha')$ et de la cha\^{\i}ne $n$,  o\`u $\Omega(\alpha')$ est une sierpinskisation de $\omega\alpha'$ et $\omega$, via une bijection $\varphi: \omega\alpha'\rightarrow \omega$ telle que $\varphi^{-1}$ soit croissante sur chaque $\omega.\{\beta\}$.

\begin{theoreme} \label{thm:qalphaintro} (cf. Theorem \ref {thm:qalpha}, Chapter \ref{chap:wellfoundedbis}). Si  $\alpha$ est un ordinal, alors  $Q_{\alpha}:= I_{<\omega}(S_{\alpha})$ est le plus petit sous sup-treillis $P$ de  $[\omega]^{<\omega}$  qui appartient \`a  $\J_{\alpha}$ .
\end{theoreme}

Les ensembles ordonn\'es  $\underline{\Omega}(\omega^*)$ et $\underline \Omega (\eta)$ sont des sup-treillis. Pour chaque cha\^{\i}ne d\'enombrable $\alpha$, nous consid\'erons des instances particuli\`eres de $\Omega(\alpha)$ qui sont des sous sup-treillis  de $\omega\times\alpha$. Nous  d\'esignons encore par $\Omega(\alpha)$ l'un quelconque d'entre eux, comme nous d\'esignons par
 $\underline\Omega(\alpha)$ le sup-treillis obtenu en ajoutant un plus petit \'el\'ement \`a $\Omega(\alpha)$ s'il n'en  a pas d\'ej\`a un.

A un type d'ordre d\'enombrable $\alpha$ nous associons un sup-treillis $P_{\alpha}$ d\'efini comme suit.

Ou bien $\alpha$  ne s'abrite pas dans une de ses sections finales strictes. Dans ce cas $\alpha$ s'\'ecrit  $\alpha= n+\alpha'$ avec $n<\omega$ et $\alpha'$ sans premier \'el\'ement. Et nous posons $P_{\alpha}:=n+\underline \Omega(\alpha')$. Dans le cas contraire, si $\alpha$ est \'equimorphe \index{\'equimorphe} \`a $\omega+ \alpha'$  nous posons $P_{\alpha}:= \underline \Omega(1+\alpha')$, sinon, nous posons $P_{\alpha}= \underline \Omega(\alpha)$.

L' importance du type de sierpinskisation ci-dessus vient du   r\'esultat suivant:
 \begin{theoreme}\label{thm:serpinskalphaintro}(cf. Theorem \ref{thm:serpinskalpha},
Chapter \ref{chap:wellfoundedbis})
Si $\alpha$ est un type d'ordre infini d\'enombrable,  $P_{\alpha}$ est minimal dans   $\J_{\alpha}$.
 \end{theoreme}

%
%
Des  th\'eor\`emes \ref{minimal} et \ref {thm:serpinskalphaintro}  d\'ecoule que:

\emph{ Si $\alpha$ n'est pas un ordinal, $P_{\alpha}$ et  $[\omega]^{<\omega}$ sont  deux \'el\'ements incomparables de $\J_{\alpha}$}.

Tandis que d'apr\`es les  th\'eor\`emes \ref{thm:qalphaintro} et \ref{thm:serpinskalphaintro}:

  \emph {Si  $\alpha$ est un ordinal infini d\'enombrable, $P_{\alpha}$ et $Q_{\alpha}$  sont des obstructions minimales.}

En fait, si $\alpha\leq \omega+\omega=\omega2$, elles coincident.
En effet, si $\alpha=\omega+n$ avec  $n<\omega$, alors $S_{\alpha}= \Omega (1)\oplus  n$. Comme  $ \Omega (1)$ est isomorphe \`a  $\omega$,  $Q_{\alpha}$ est isomorphe au produit direct $\omega\times (n+1)$ qui est lui m\^eme isomorphe \`a  $\Omega(n+1)=P_{\alpha}$. Si $\alpha=\omega+\omega= \omega2$, $S_{\alpha}= \Omega (2)$. Cet ensemble ordonn\'e est isomorphe au produit direct $\omega\times 2$. Et l'ensemble  $Q_{\alpha}$ est isomorphe \`a  $[\omega]^2$, la partie du produit direct  $\omega \times  \omega$ constitu\'ee des couples $(i,j)$ tels que $i<j$. En retour $[\omega]^2$ est isomorphe   \`a $\Omega(\omega)=P_{\alpha}$.

Comme nous le verrons (Chapitre \ref{chap:wellfoundedbis}, Corollary \ref{incomp}) au dela de  $\omega2$ ces deux obstructions sont incomparables.






 Ce travail sugg\`ere deux autres questions.
\begin{questions}
\begin{enumerate}
\item Si $\alpha$ est un ordinal infini, est  ce que  les obstructions minimales  sont
$\alpha$, $P_{\alpha}$, $Q_\alpha$ et des sommes
lexicographiques d'obstructions correspondant \`a des ordinaux
plus petits?
\item Si $\alpha$ est une cha\^{\i}ne dispers\'ee non bien ordonn\'ee, est ce que les obstructions minimales  sont
$\alpha$, $P_{\alpha}$, $[\omega]^{<\omega}$ et des sommes
lexicographiques d'obstructions correspondant \`a cha\^{\i}nes dispers\'ees
plus petites?
\end{enumerate}
\end{questions}
Nous ne pouvons donner que quelques exemples de type d'ordres pour lesquels la r\'eponse est positive.

\section{Le cas des treillis modulaires}
N'ayant pas de r\'eponse au probl\`eme pos\'e, nous le
restreignons. Au lieu de consid\'erer la classe des treillis
alg\'ebriques nous consid\'erons des sous-classes particuli\`eres.
Une classe int\'eressante est celle des treillis alg\'ebriques
modulaires.
Au lieu de rechercher les obstructions
correspondant  \`a un type d\'enombrable $\alpha$ nous recherchons
les types d'ordre $\alpha$ pour lesquels il n'y a que deux
obstructions, \`a savoir $1+\alpha$ et $[\omega]^{<\omega}$. Les
r\'esultats que nous obtenons sont valables pour une classe de
treillis plus large que les treillis modulaires alg\'ebriques, la
classe $\mathbb L$ d\'efinie comme suit. Notons $\mathbb A $ la
classe des treillis alg\'ebriques. Pour un type d'ordre $\alpha$,
soit $L_{\alpha}:= 1+(1\oplus \alpha)+1$ le treillis form\'e par
la somme directe de la cha\^{\i}ne \`a un \'el\'ement $1$ et de la
cha\^{\i}ne $\alpha$, avec un plus petit et un plus grand
\'el\'ements ajout\'es. Notons $\mathbb L$  la collection des
$L\in\mathbb A$ tel que $L$ ne contient aucun sous-treillis \index{sublattice} isomorphe \`a
$L_{\omega+1}$ ou \`a $L_{\omega^*}$. Puisque  $L_{2}$ est isomorphe \`a $M_{5}$, le treillis non modulaire \`a cinq \'el\'ements, tout treillis alg\'ebrique modulaire appartient \`a   $\mathbb L$.
\begin{figure}[htbp]
\begin{center}
\includegraphics[width=3in]{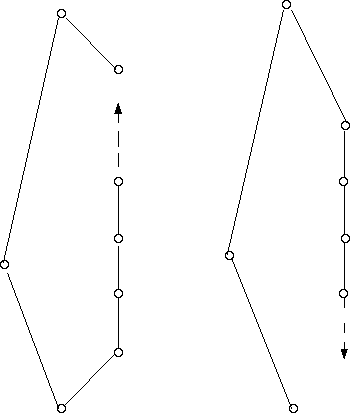}
\caption{$L_{\omega+1}$,     $L_{\omega^*}$} \label{L(omega)}
\end{center}\end{figure}

\begin{figure}[htbp]
\begin{center}\includegraphics[width=0.8in]{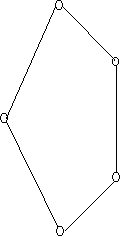}
\caption{$M_{5}$} \label{M5}
\end{center}\end{figure}

Pour un type d'ordre $\alpha$, nous notons $\mathbb A_{\neg
\alpha}$ (resp. $\mathbb L_{\neg \alpha}$) la collection des
$L\in\mathbb A$ (resp. $L\in\mathbb L$) tel que $I(\alpha)\not
\leq L$. Consid\'erons la classe $\mathbb{K}$ des types d'ordre
$\alpha$ tel que $L \in  \mathbb L_{\neg \alpha}$ d\`es que $K(L)$
ne contient ni cha\^{\i}ne de type $1+\alpha$ ni sous ensemble
isomorphe \`a $[\omega]^{<\omega}$. Si $\alpha$ est d\'enombrable
alors ces deux conditions sont n\'ecessaires pour interdire $I(\alpha)$ dans $L$;
si, de plus, $ \alpha$ est ind\'ecomposable, cela revient \`a
interdire $\alpha$ dans $L$.
\begin{theoreme}\label{thm6} (cf. Theorem \ref {thm6},  Chapter \ref{chap:algebraic})
La classe $\mathbb{K}$ satisfait les propri\'et\'es suivantes:\\
$(p_{1})$ $0\in\mathbb{K}$, $1\in \mathbb{K}$;\\ $(p_{2})$ Si
$\alpha+1\in\mathbb{K}$ et $\beta\in\mathbb{K}$ alors
$\alpha+1+\beta\in\mathbb{K}$;\\  $(p_{3})$ Si
$\alpha_{n}+1\in\mathbb{K}$
 pour tout $n<\omega$ alors la $\omega$-somme
$\gamma:=\alpha_{0}+1+\alpha_{1}+1+\ldots+\alpha_{n}+1+\ldots$ est
dans $\mathbb{K}$;\\ $(p_{4})$ Si $\alpha_{n}\in\mathbb{K}$ et
$\{m: \alpha_{n}\leq\alpha_{m}\}$ est infini pour tout $n<\omega$
alors la $\omega^*$-somme $\delta:=
\ldots+\alpha_{n+1}+\alpha_{n}+\ldots+\alpha_{1}+\alpha_{0}$ est
dans $\mathbb{K}$.\\ $(p_{5})$ Si $\alpha$ est un type d'ordre
d\'enombrable dispers\'e, alors $\alpha\in \mathbb{K}$ si et
seulement si $\alpha=\alpha_{0}+\alpha_{1}+...+\alpha_{n}$ avec
$\alpha_{i}\in \mathbb{K}$ pour $i\leq n$, $\alpha_{i}$
strictement ind\'ecomposable \`a gauche pour $i<n$ et $\alpha_{n}$
ind\'ecomposable;\\
 $(p_{6})$  $\eta\in \mathbb{K}$.\\
\end{theoreme}

\begin{corollaire}\label {cor: scatteredmodularintro}(cf. Corollary \ref{cor: scatteredmodular}, Chapter \ref {chap:algebraic})Un treillis alg\'ebrique modulaire est bien fond\'e, respectivement dispers\'e, si et seulement si le sup-treillis de ses \'el\'ements compacts  est bien fond\'e, resp. dispers\'e, et ne contient pas de sous sup-treillis isomorphe \`a $[\omega]^{<\omega}$.
\end{corollaire}

On peut dire un peu plus. Soit $\mathbb{P}$ la plus petite classe de types d'ordre,
satisfaisant les propri\'et\'es $(p_{1})$ \`a $(p_{4})$ dessus. Par
exemple, $\omega, \omega^{*}, \omega(\omega^{\alpha})^{*},
\omega^{*}\omega, \omega^{*}\omega\omega^{*}$ sont dans
$\mathbb{P}$. On a  $\mathbb{P}\subseteq \mathbb{K}\setminus
\{\eta \}$.

Nous ignorons si tout type d'ordre d\'enombrable
$\alpha\in \mathbb{K}\setminus \{\eta \}$ est \'equimorphe \`a un
$\alpha' \in \mathbb{P}$. La classe des types d'ordre $\alpha$ qui
sont \'equimorphes \`a un $\alpha' \in \mathbb{P}$ satisfait aussi
$(p_{5})$ donc, pour r\'epondre \`a notre question, nous pouvons supposer que $\alpha$ est ind\'ecomposable. Nous n'avons r\'eussi
\`a r\'epondre \`a cette question que dans le cas o\`u $\alpha$
 est \emph{indivisible}, c'est \`a dire s'abrite dans au moins  une des deux parties de n'importe qu'elle partition de $\alpha$ en deux parties.    Nous \'eliminons les types d'ordres indivisibles qui
ne sont pas dans $\mathbb{P}$ en construisant des treillis
alg\'ebriques de la forme $J(T)$ o\`u $T$ est un treillis
distributif appropri\'e. Ces treillis distributifs sont obtenus
par des sierpinskizations. Nous obtenons:

\begin{theoreme} (cf. Theorem \ref{thm7},  Chapter\ref{chap:algebraic})
Un  type d'ordre d\'enombrable indivisible $\alpha$ est \'equimorphe \`a un type d'ordre d\'enombrable  $\alpha' \in \mathbb {P} \cup \{\eta\}$ si et seulement si quelque soit le treillis  alg\'ebrique modulaire $L$ les propri\'e\'es suivantes sont \'equivalentes:
\begin{enumerate}[{(i)}]\item    $L$ ne contient pas de cha\^{\i}ne de type $I(\alpha)$.

\item  Le sup-treillis $K(L)$ des \'el\'ements compacts  ne contient ni $1+\alpha$ ni de  sous sup-lattice isomorphe \`a $[\omega]^{<\omega}$.
\end{enumerate}
\end{theoreme}



\section{Treillis distributifs et  sup-treillis $[\omega]^{<\omega}$}
Le fait qu'un sup-treillis $P$ contienne un sous-sup-treillis
isomorphe \`a $[\omega]^{<\omega}$  \'equivaut \`a l'existence
d'un ensemble ind\'ependant infini. Rappelons que dans un
sup-treillis $P$, un sous-ensemble $X$ est {\it ind\'ependant} si
$x\not \leq \bigvee F$ pour tout $x\in X$ et toute partie finie
non vide $F$ de $X\setminus \{x\}$. Les conditions qui assurent
l'existence d'un ensemble ind\'ependant infini ou les
cons\'equences de leurs inexistence  ont \'et\'e d\'ej\`a
consid\'er\'es (voir les recherches sur le "free-subset problem"
de Hajnal \cite {shel} ou "on the cofinality of posets"
\cite{galv, mp2}).

Le r\'esultat  suivant, traduit l'existence d'un ensemble
ind\'ependant de taille $\kappa$ pour un cardinal $\kappa$.

\begin{theoreme} \cite {chapou}  \cite {lmp3} \label  {tm01} (cf. Theorem \ref{tm1},  Chapter \ref{chap:TD}) Soit $\kappa$ un cardinal ;
pour un sup-treillis $P$ les propri\'et\'es suivantes sont
\'equivalentes:\\ $(i)$ $P$ contient un ensemble ind\'ependant de
taille $\kappa$;\\ $(ii)$ $P$ contient un sous sup-treillis
isomorphe \`a $[\kappa]^{<\omega}$;\\ $(iii)$ $P$ contient un
sous-ensemble isomorphe \`a $[\kappa]^{<\omega}$;\\ $(iv)$ $J(P)$
contient un sous-ensemble isomorphe \`a $\mathfrak P (\kappa)$;\\
$(v)$ $\mathfrak P (\kappa)$ se plonge dans $J(P)$ par une
application qui pr\'eserve les sup arbitraires.\end{theoreme}

L'existence d'un ensemble ind\'ependant infini, dans le cas d'un
treillis distributif, se traduit par l'existence de deux
sup-treillis particuliers $I_{<\omega}(\Gamma)$ et
$I_{<\omega}(\Delta)$ , o\`u $\Gamma$ et $\Delta$ sont deux
inf-treillis sp\'eciaux.

 Soit
$\Delta:= \{ (i,j): i<j\leq\omega \}$ muni de l'ordre $(i,j)\leq
(i',j')$ si et seulement si $j\leq i'$ ou $i=i'$ et $j\leq j'$.
Soit $\Gamma:=\{(i,j)\in \Delta : j=i+1$ ou  $j=\omega \}$ muni de
l'ordre induit. Les ensembles ordonn\'es $\Delta$ et $\Gamma$ sont
des inf-treillis bien-fond\'es dont les \'el\'ements maximaux
forment une anticha\^{\i}ne infinie. Pour un ensemble ordonn\'e
$Q$, notons $I_{<\omega}(Q)$ l'ensemble des sections initiales
finiment engendr\'es de $Q$. Les ensembles $I_{<\omega}(\Delta)$
et $I_{<\omega}(\Gamma )$, sont des treillis  distributifs
bien-fond\'es contenant un sous ensemble isomorphe \`a
$[\omega]^{<\omega}$.
\begin{figure}
\begin{center}
\includegraphics[width=2.5in]{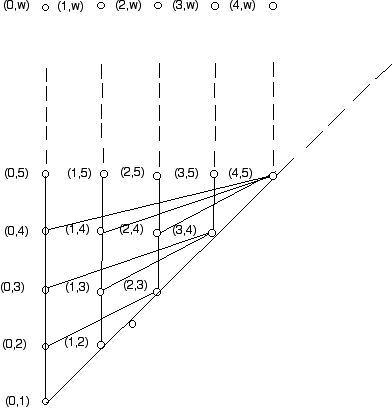}
\caption{$\Delta$}
\label{fig:delta}
\end{center}
\end{figure}
\begin{figure}
\begin{center}
\includegraphics[width=2.5in]{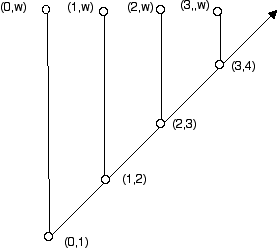}
\caption{$\Gamma$}
\label{fig:gamma}
\end{center}
\end{figure}

Notons que, pour un treillis distributif, contenir
$[\omega]^{<\omega}$, comme sous sup-treillis  ou comme
sous-ensemble ordonn\'e est \'equivalent.
\begin{theoreme}\label{thm08} (cf. Theorem \ref{thm8}, Chapter \ref{chap:TD})
Soit $T$ un treillis distributif. Les propri\'et\'es suivantes
sont \'equivalentes:\\ $(i)$ $T$ contient un sous-ensemble
isomorphe \`a $[\omega]^{<\omega}$;\\
$(ii)$ Le treillis
$[\omega]^{<\omega}$ est quotient d'un sous-treillis de $T$;\\
$(iii)$ $T$ contient un sous-treillis isomorphe \`a
$I_{<\omega}(\Gamma)$ ou \`a $I_{<\omega}(\Delta)$.
\end{theoreme}

Un des arguments de la preuve du Th\'eor\`eme \ref{thm08} est le
Th\'eor\`eme de Ramsey  \cite{rams} appliqu\'e comme dans
\cite{duff}. Les ensembles ordonn\'es $\Delta$ (avec un plus grand
\'el\'ement ajout\'e) et $\Gamma$ ont \'et\'e consid\'er\'es
avant, dans \cite{lmp1} et \cite{lmp2}. L'ensemble ordonn\'e
$\delta$ obtenu \`a partir de $\Delta$ en supprimant les
\'el\'ements maximaux a \'et\'e aussi consid\'er\'e par E.
Corominas en 1970 comme une variante de l'exemple construit par R.
Rado \cite{rado}.
\chapter[Well-founded algebraic lattices]{A characterization of well-founded algebraic lattices}\label{chap:wellfounded}

We characterize well-founded  algebraic lattices by means of
forbidden subsemilattices of the join-semilattice made of their
compact elements.  More specifically, we show that  an algebraic
lattice $L$ is well-founded if and only if $K(L)$, the
join-semilattice of compact elements of $L$, is well-founded and
contains neither $[\omega]^{<\omega}$, nor
$\underline\Omega(\omega^*)$
 as a join-subsemilattice.   As an immediate corollary, we get that an  algebraic modular lattice $L$ is well-founded if and only if
$K(L)$ is well-founded and contains no infinite independent set.
If $K(L)$ is a join-subsemilattice of $I_{<\omega}(Q)$, the set of
finitely generated initial segments of a well-founded poset $Q$,
then $L$ is well-founded if and only if $K(L)$ is
well-quasi-ordered.

\section{Introduction and synopsis of results}
Algebraic lattices \index{algebraic lattice} and join-semilattices
\index{join-semilattice} (with a 0) are two aspects of the same
thing, as expressed in the following basic result.
\begin{theorem} \cite{Hofmann}, \cite{grat} The collection $J(P)$ of ideals \index{ideal} of a join-semilattice $P$, once ordered by inclusion,
is an algebraic lattice
and the subposet $K(J(P))$ of its compact elements \index{compact
element} is isomorphic to $P$. Conversely, the subposet $K(L)$ of
compact elements of an algebraic lattice $L$ is a join-semilattice
with a $0$ and $J(K(L))$ is isomorphic to $L$.\end{theorem}

In this paper, we characterize well-founded algebraic lattices by
means of forbidden join-subsemilattices of the join-semilattice made
of their compact elements. In the sequel  $\omega$  denotes the
chain of non-negative integers, and when this causes no confusion,
the first infinite cardinal \index{cardinal} as well as  the first
infinite ordinal \index{ordinal}. We denote $\omega^*$ the chain of
negative integers.  We  recall that a poset $P$ is {\it
well-founded} \index{well-founded} provided that every non-empty
subset of $P$ has a minimal element. With the Axiom of dependent
choices, this amounts to the fact that $P$ contains no subset
isomorphic to $\omega^*$. Let $\Omega(\omega^*)$ be the set $[
\omega]^2$ of two-element subsets of $\omega$, identified to pairs
$(i,j)$, $i<j<\omega$, ordered so that $(i,j)\leq (i',j')$ if and
only if $i'\leq i$ and $j\leq j'$ w.r.t. the natural order on
$\omega$. Let $\underline\Omega(\omega^*):=\Omega(\omega^*) \cup
\{\emptyset \}$ be obtained by adding a least element. Note that
$\underline\Omega(\omega^*)$ is isomorphic to the set of bounded
intervals of $\omega$ (or $\omega^*$) ordered by inclusion. Moreover
$\underline\Omega(\omega^*)$ is a join-semilattice
($(i,j)\vee(i',j')=(i\wedge i',j\vee j')$). The join-semilattice
$\underline\Omega(\omega^*)$ embeds in $\Omega(\omega^*)$ as a
join-semilattice; the advantage of $\underline\Omega(\omega^*)$
w.r.t. our discussion is to have a zero. Let $\kappa$ be a cardinal
number, e.g. $\kappa:= \omega$; denote $[\kappa]^{<\omega}$  (resp.
$\mathfrak {P}(\kappa)$ ) the set, ordered by inclusion, consisting
of finite (resp. arbitrary ) subsets of $\kappa$. The posets
$\underline\Omega(\omega^*)$ and $[\kappa ]^{<\omega}$ are
well-founded lattices, whereas  the algebraic lattices $J(\underline
\Omega(\omega^*))$  and $J([\kappa]^{<\omega})$ ($\kappa$ infinite)
are not well-founded (and we may note that $J([\kappa]^{<\omega})$
is isomorphic to $\mathfrak {P}(\kappa)$). As  a poset
$\underline\Omega(\omega^*)$ is isomorphic to  a subset  of
$[\omega]^{<\omega}$, but not as a join-subsemilattice. This is our
first result.

\begin{proposition}\label{w-f}
$\underline\Omega(\omega^*)$ does not embed in
$[\omega]^{<\omega}$ as a join-subsemilattice;
 more generally, if $Q$ is a well-founded poset then $\underline \Omega(\omega^*)$ does
not embed as a join-subsemilattice into $I_{<\omega}(Q)$, the
join-semilattice made of finitely generated initial segments of
$Q$.
\end{proposition}
\begin{figure}
\begin{center}
\includegraphics[width=2.5in]{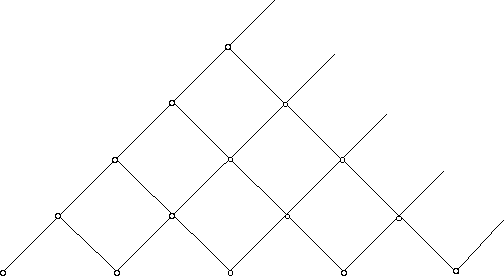}
\caption{$\Omega(\omega^*)$}
\end{center}
\label{fig:omega}
\end{figure}
Our next result expresses that $\underline
\Omega(\omega^*)$ and  $[\omega]^{<\omega}$ are unavoidable
examples of well-founded join-semilattices whose set of ideals is
not well-founded.
\begin{theorem}\label{thm4}
An algebraic lattice $L$  is well-founded if and only if $K(L)$ is
well-founded  and contains no join-subsemilattice  isomorphic to
$\underline \Omega(\omega^*)$ or to $[\omega]^{<\omega}$.\\
\end{theorem}

The fact that a join-semilattice $P$ contains a
join-subsemilattice isomorphic to $[\omega]^{<\omega}$  amounts to
the existence of an infinite independent set. Let us recall that a
subset $X$ of a join-semilattice $P$ is {\it independent}
\index{independent} if $x\not \leq \bigvee F$ for every $x\in X$
and every non-empty finite subset $F$ of $X\setminus \{x\}$.
Conditions which may insure the existence of an infinite
independent set  or consequences of the inexistence
 of such sets  have been considered within the framework of the structure of closure systems \index{closure system}
 (cf. the research on the  "free-subset problem" of Hajnal
\cite {shel} or on the cofinality of posets \cite{galv, mp2}). A
basic result is the following. \\
\begin{theorem} \cite {chapou}  \cite {lmp3} \label  {tm2.1} Let $\kappa$ be a cardinal number; for a join-semilattice $P$
the following properties are equivalent:\\ $(i)$ $P$ contains an
independent set of size $\kappa$;\\ $(ii)$ $P$ contains  a
join-subsemilattice isomorphic to $[\kappa]^{<\omega}$;\\ $(iii)$
$P$ contains  a subposet isomorphic to $[\kappa]^{<\omega}$;\\
$(iv)$ $J(P)$ contains a subposet isomorphic to $\mathfrak P
(\kappa)$;\\ $(v)$ $\mathfrak P (\kappa)$ embeds in $J(P)$ via a
map preserving arbitrary joins.
\end{theorem}
Let $L(\alpha):= 1+(1\oplus J(\alpha))+1$ be the lattice made of
the direct sum \index{direct sum} of the one-element chain $1$ and
the chain $J(\alpha)$, ($\alpha$  finite or equal to $\omega^*$),
with top and bottom added.
\begin{figure}
\begin{center}
\includegraphics[width=1in]{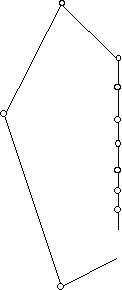}
\caption{$L(\omega^*)$}
\end{center}
\label{lomega}
\end{figure}

Clearly $J(\underline\Omega(\omega^*))$ contains a sublattice
isomorphic to  $L(\omega^*)$. Since a modular lattice
\index{modular lattice} contains no sublattice isomorphic  to
$L(2)$, we get as a corollary of Theorem \ref {thm4}:

\begin{theorem}\label{cor4}
An algebraic modular lattice $L$  is well-founded
\index{well-founded} if and only if $K(L)$ is well-founded  and
contains no infinite independent set.
\end{theorem}

 Another consequence is this:
\begin{theorem} \label{thmwf}
 For a join-semilattice $P$, the following properties are equivalent:
 \begin{enumerate}[(i)]
 \item $P$ is well-founded with no infinite antichain \index{antichain};
 \item $P$ contains no infinite independent set and embeds as a join-semilattice into a join-semilattice
 of the  form $I_{<\omega} (Q)$ where $Q$ is some well-founded poset.
\end{enumerate}
\end{theorem}

Posets which are well-founded and have no  infinite antichain are
said {\it well-partially-ordered}\index{well-partially-ordered} or
{\it well-quasi-ordered}\index{well-quasi-ordered}, wqo for short.
They play an important role in several areas (see \cite
{fraissetr}). If $P$ is a wqo join-semilattice then $J(P)$, the
poset of ideals  of  $P$,  is well-founded and one may assign to
every $J\in J(P)$ an ordinal, its  {\it height} \index{height},
denoted by $h(J, J(P))$. This ordinal is defined  by induction,
setting $h(J,J(P)):= Sup(\{h(J',J(P))+1: J'\in J(P), J'\subset
J\})$ and $h(J', J(P)):= 0$ if $J'$ is minimal in $J(P)$. The
ordinal $h(J(P)):= h(P, J(P))+1$ is the {\it height} of $J(P)$. If
$P:= I_{<\omega}(Q)$, with $Q$ wqo, then $J(P)$ contains a chain
of order type $h(J(P))$. This is an equivalent form of the famous
result of de Jongh and Parikh \cite{dejongh-parikh} asserting that
among the linear extensions \index{linear extension} of a wqo, one
has a maximum order type.

\begin{problem} Let $P$ be a wqo join-semilattice; does $J(P)$ contain a chain of order type $h(J(P))$?\end{problem}

An immediate corollary of Theorem \ref {thmwf} is:

\begin{corollary} \label{poset} A join-semilattice $P$ of $[\omega]^{<\omega}$ contains either  $[\omega]^{<\omega}$ as a join-semilattice or is wqo.
\end{corollary}
 Let us compare  join-subsemilattices of $[\omega]^{<\omega}$.
 Set $P\leq P'$ for two  such join-subsemilattices  if $P$ embeds in $P'$ as a join-semilattice.
 This gives a quasi-order  and,  according to  Corollary \ref{poset},
 the poset corresponding to this quasi-order has a largest element (namely $[\omega]^{<\omega}$),
 and  all other members come from   wqo join-semilattices.
 Basic examples of join-subsemilattices of $[\omega]^{<\omega}$ are the  $I_{<\omega} (Q)$'s
 where $Q$ is a countable poset such that no element is above infinitely many
 elements. These posets $Q$ are exactly those which are embeddable
 in the poset $[\omega]^{<\omega}$ ordered by inclusion. An
 interesting subclass is made of posets of the form $Q=(\N, \leq)$
 where the order $\leq$ is the intersection of the natural order $\mathfrak{N}$ on $\N$
 and of a linear order\index{linear order} $\mathfrak{L}$ on $\N$,
 (that is $x \leq y$ if $x \leq y$ w.r.t. $\mathfrak{N}$ and $x \leq y$ w.r.t. $\mathfrak{L}$).
 If $\alpha$ is the type of the linear order,
 a poset of this form  is a {\it sierpinskisation} \index{sierpinskisation} of $\alpha$.
 The corresponding join-semilattices are wqo provided that the posets  $Q$ have no infinite antichain;
 in the particular case of  a sierpinskisation of $\alpha$ this amounts to the fact that $\alpha$ is
 well-ordered \index{well-order}.

 As shown in \cite {pz}, sierpinskisations
given by a bijective map $\psi:\omega\alpha\rightarrow \omega$
which is  order-preserving \index{order-preserving} on each
component $\omega \cdot \{i\}$ of $\omega\alpha$ are all
embeddable in each other, and for this reason  denoted by the same
symbol $\Omega(\alpha)$. Among the representatives of $\Omega
(\alpha)$, some are join-semilattices, and among them, join-subsemilattices
of the direct product \index{direct product} $\omega\times \alpha$
(this is notably the case of the poset $\Omega(\omega ^{*})$ we
previously defined). We extend the first part of Proposition
\ref{w-f}, showing that except for $\alpha\leq\omega$, the
representatives of $\Omega(\alpha)$ which are join-semilattices
never embed in $[\omega]^{<\omega}$ as join-semilattices, whereas
they embed  as posets (see Corollary \ref{sierpinski} and Example
\ref{ex:ordinal}). From this result, it follows that the posets
$\Omega (\alpha)$ and $I_{<\omega}(\Omega (\alpha))$  do not embed
in each other as join-semilattices.

 These two posets  provide  examples of a  join-semilattice
$P$ such that $P$ contains no chain of type $\alpha$ while $J(P)$
contains a chain of type $J(\alpha)$. However,   if $\alpha$ is
not well ordered  then $I_{<\omega}(\Omega (\alpha))$ and
$[\omega]^{<\omega}$ embed in each other as join-semilattices.

  \begin{problem}
 Let $\alpha$ be a countable  ordinal.
Is there a minimum member among the join-subsemilattices $P$ of
$[\omega]^{<\omega}$ such that $J(P)$ contains a chain of type
$\alpha+1$? Is it true that this minimum is
$I_{<\omega}(\Omega(\alpha))$  if $\alpha$ is indecomposable?
\end{problem}

\section{Definitions and basic results}\label {subsection2.2}
Our definitions and notations are standard and agree with \cite
{grat} except on minor points that we will mention. We adopt the
same terminology as in \cite {chapou}. We recall only few things.
Let $P$ be a poset. A subset $I$ of $P$ is an {\it initial
segment} \index{initial segment} of $P$ if $x\in P$, $y\in I$ and
$x\leq y$ imply $x\in I$. If
 $A$ is a subset of $P$, then $\downarrow A=\{x\in P: x\leq y$ for some $y\in A\}$ denotes the least initial segment containing
$A$. If $I=\downarrow A$ we say that $I$ is {\it generated} by $A$
or $A$ is {\it cofinal} in \index{cofinal} $I$. If $A=\{a\}$ then
$I$ is a {\it principal initial segment} \index{principal initial
segment} and we write $\downarrow a$ instead of $\downarrow
\{a\}$. We denote $down(P)$ the set of principal initial segments
of $P$. A {\it final segment} \index{final segment} of $P$ is any
initial segment of $P^*$, the dual of $P$. We denote by $\uparrow
A$ the final segment generated by $A$. If $A=\{a\}$ we write
$\uparrow a$ instead of $\uparrow \{a\}$. A subset $I$ of $P$ is
{\it directed}\index{directed}
 if every finite subset of $I$
has an upper bound in $I$ (that is $I$ is non-empty and every pair
of elements of $I$ has an upper bound). An {\it ideal}\index{ideal}
is a non-empty directed initial segment of $P$ (in some other texts,
the empty set is an ideal). We denote $I(P)$ (respectively, $I_
{<\omega }(P)$, $J(P)$) the set of initial segments (respectively,
finitely generated initial segments, ideals of $P$) ordered by
inclusion and we set
 $ J_ *(P):=J(P)\cup\{\emptyset\}$, $I_ 0 (P):=I_ {<\omega}
(P)\setminus\{\emptyset\}$. Others authors use {\it down set} for
initial segment. Note that $down(P)$ has not to be confused with
$I(P)$. If $P$ is a join-semilattice with a $0$, an element $x\in
P$ is {\it join-irreducible} \index{join-irreducible} if it is
distinct from $0$, and if $x=a\vee b$ implies $x=a$ or $x=b$ (this
is a slight variation from \cite{grat}). We denote $\J_{irr}(P)$
the set of join-irreducibles of $P$. An element $a$ in a lattice
$L$ is {\it compact}\index{compact} if for every $A\subset L$,
$a\leq \bigvee A$ implies $a\leq \bigvee A'$ for some finite
subset $A'$ of $A$. The lattice $L$ is {\it compactly
generated}\index{ compactly generated} if every element is a
supremum of compact elements. A lattice is {\it
algebraic}\index{algebraic} if it is complete and compactly
generated.
\\

We note that $I_{ <\omega}(P)$ is the set of compact elements of
$I(P)$, hence $J(I_{ <\omega}(P))\cong I(P)$. Moreover $I_
{<\omega}(P)$ is  a lattice, and in fact a distributive lattice,
if and only if $P$ is {\it $\downarrow$-closed}
\index{$\downarrow$-closed}, that is, the intersection of two
principal initial segments of $P$ is a finite union, possibly
empty, of principal initial segments. We also note that $J(P)$  is
the set of join-irreducible elements of $I(P)$; moreover,
$I_{<\omega}(J(P))\cong I(P)$ whenever $P$ has no infinite
antichain. \\ Notably for the proof of Theorem
\ref{finitesubsets}, we will need the following  results.

\begin{theorem} \label{wellfounded} Let $P$ be a poset.\\
$a)$  $I_{ <\omega}(P)$ is well-founded if and only if $P$ is
well-founded (Birkhoff  1937, see \cite {birk});\\ $b)$ $I_{
<\omega}(P)$ is wqo iff $P$ is wqo iff $I(P)$ well-founded (
Higman 1952 \cite {higm}); \\ $c)$ if $P$ is a well-founded
join-semilattice with a 0, then every member of $P$ is a finite
join of join-irreducible elements of $P$ (Birkhoff, 1937, see
\cite {birk});\\ $d)$ A join-semilattice $P$ with a zero is wqo if
and only if every member of $P$ is a finite join of
join-irreducible elements of $P$ and the set $\J_{irr}(P)$ of
these join-irreducible elements is wqo (follows from $b)$ and
$c)$).
\end{theorem}
 A poset $P$ is {\it scattered} \index{scattered} if it does not contain  a copy of $\eta$,
 the chain of rational numbers. A topological space $T$ is
{\it scattered} if every non-empty closed set contains some
isolated point. The power set of a set, once  equipped with the
product topology, is a compact space. The set $J(P)$ of ideals of
a join-semilattice $P$ with a $0$ is a closed subspace of
$\mathfrak P (P)$, hence is a compact space too. Consequently, an
algebraic lattice $L$ can be viewed as a poset and a topological
space as well. It is easy  to see that if $L$ is topologically
scattered \index{topologically scattered} then it is order
scattered \index{order scattered}. It is  a more significant fact,
due to M.Mislove \cite {misl}, that the converse holds if $L$ is
distributive.

\section[Separating chains of ideals]
{Separating chains of ideals and proofs of Proposition \ref {w-f}
and Theorem \ref{thm4}}

Let $P$ be a join-semilattice. If $x\in P$ and $J\in J(P)$, then
$\downarrow x$ and $J$ have a join $\downarrow x \bigvee J$ in
$J(P)$ and $\downarrow x \bigvee J=\downarrow\{ x \vee y: y\in J\}$.
Instead of $\downarrow x \bigvee J$ we also use the notation
$\{x\}\bigvee J$. Note that $\{x\}\bigvee J$ is the least member of
$J(P)$ containing $\{x\}\cup J$. We say that a non-empty chain
$\mathcal I$ of ideals of $P$ is {\it separating} \index{separating}
if for every $I\in\mathcal I\setminus\{\cup\mathcal I\}$ and every
$x\in \cup\mathcal I\setminus I$, there is some $J\in\mathcal I$
such that $I\not\subseteq\{x\}\bigvee J$.\\ If $\mathcal I$ is
separating then $\mathcal I$ has a least element implies it is a
singleton set. In $P:=[\omega]^{<\omega}$, the chain $\mathcal
I:=\{I_{n}: n<\omega\}$ where $I_{n}$ consists of the finite subsets
of $\{m: n\leq m\}$ is separating. In $P:=\omega^*$, the chain
$\mathcal I:=\{\downarrow x: x\in P\}$ is non-separating, as well as
all of its infinite subchains. In $P:=\Omega(\omega^*)$ the chain
$\mathcal I:=\{I_{n}: n<\omega\}$ where $I_{n}:=\{(i,j): n\leq
i<j<\omega\}$ has the same property.\\ We may observe that {\it a
join-preserving embedding\index{join-preserving embedding} from a
join-semilattice $P$ into a join-semilattice $Q$ transforms every
separating (resp. non-separating) chain of ideals of $P$ into a
separating (resp. non-separating) chain of ideals of $Q$} (If
$\mathcal I$ is a separating chain of ideals of $P$, then $\mathcal
J=\{f(I): I\in \mathcal I\}$ is a separating chain of ideals of
$Q$). Hence the containment of $[\omega]^{<\omega}$ (resp. of
$\omega^*$ or of $\Omega(\omega^*)$), as a join-subsemilattice,
provides a chain of ideals which is separating (resp.
non-separating, as are all its infinite subchains, as well). We show
in the  next  two lemmas that the converse holds.

\begin{lemma}\label{independent}
A join-semilattice $P$ contains an infinite independent set if and
only if it contains an infinite separating chain of ideals.
\end{lemma}

\begin{proof}
Let $X=\{x_{n}: n<\omega\}$ be an infinite independent set. Let
$I_{n}$ be the ideal generated by $X\setminus\{x_{i}: 0\leq i\leq
n\}$. The chain $\mathcal I=\{I_{n}: n<\omega\}$ is separating. Let
$\mathcal I$ be an infinite separating chain of ideals. Define
inductively an infinite sequence $x_{0}, I_{0}, \ldots, x_{n},I_{n},
\ldots$ such that $I_{0}\in\mathcal I\setminus\{\cup\mathcal I\},
x_{0}\in \cup\mathcal I \setminus I_{0}$ and
 such that:\\
 $a_{n})$
$I_{n}\in\mathcal I$;\\ $b_{n})$ $I_{n}\subset I_{n-1}$;\\
 $c_{n})$
$x_{n}\in I_{n-1}\setminus(\{x_{0}\vee \ldots \vee
x_{n-1}\}\bigvee I_{n})$ for every  $n\geq 1$.\\ The construction
is immediate. Indeed, since $\mathcal I$ is infinite then
$\mathcal I\setminus\{\cup\mathcal I\}\not =\emptyset$. Choose
arbitrary $I_{0}\in\mathcal I\setminus\{\cup\mathcal I\}$ and
$x_{0}\in\cup\mathcal I\setminus I_{0}$. Let $n\geq 1$. Suppose
$x_{k}, I_{k}$ defined and satisfying $a_{k}), b_{k}), c_{k})$ for
all $k\leq n-1$. Set $I:=I_{n-1}$ and $x:=x_{0}\vee \ldots \vee
x_{n-1}$. Since $I\in\mathcal I$ and $x\in\cup\mathcal I\setminus
I$, there is some $J\in\mathcal I$ such that
$I\not\subseteq\{x\}\bigvee J$. Let $z\in I\setminus(\{x\}\bigvee
J)$. Set $x_{n}:=z$, $I_{n}:=J$. The set $X:=\{x_{n}: n<\omega\}$
is independent. Indeed if $x\in X$ then since $x=x_{n}$ for some
$n$, $n<\omega$, condition $c_{n})$ asserts that there is some
ideal containing $X\setminus \{x\}$ and excluding $x$.
\end{proof}\\

\begin{lemma} \label{w*}
A join-semilattice $P$ contains either $\omega^*$ or
$\Omega(\omega^*)$ as a join-subsemilattice if and only if it
contains an $\omega^*$-chain $\mathcal I$ of ideals such that all
infinite subchains are non-separating.
\end{lemma}

\begin{proof}
Let $\mathcal I$ be an $\omega^*$-chain of ideals and let $A$ be its
largest element (that is $A=\cup\mathcal I$). Let $E$ denote the set
$\{x: x\in A$ and $I\subset\downarrow x$ for some $I\in \mathcal
I\}$.\\ {\bf Case (i)}. For every $I\in\mathcal I$, $I\cap
E\not=\emptyset$. We can build an infinite strictly decreasing
sequence $x_{0}, \ldots, x_{n}, \ldots$ of elements of $P$. Indeed,
let us choose $x_{0}\in E\cap(\cup\mathcal I)$ and $I_{0}$ such that
$I_{{0}}\subset\downarrow x_{0}$. Suppose $x_{0}, \ldots, x_{n}$ and
$I_{0}, \ldots, I_{n}$ defined such that $I_{i}\subset\downarrow
x_{i}$ for all $i=0, \ldots, n$. As $E\cap I_{n}\not=\emptyset$ we
can select $x_{n}\in E\cap I_{n}$ and by definition of $E$, we can
select some $I_{n+1}\in \mathcal I$ such
that $I_{n+1}\subset\downarrow x_{n+1}$. Thus $\omega^*\leq P$.\\
{\bf Case (ii)}. There is some $I\in\mathcal I$ such that $I\cap
E=\emptyset$. In particular all members of $\mathcal I$ included in
$I$ are unbounded in $I$. Since all infinite subchains of $\mathcal
I$ are non-separating then, with no loss of generality, we may
suppose that $I=A$ (hence $E=\emptyset$). We set $I_{-1}:=A$ and
define a sequence $x_{0}, I_{0}, \ldots, x_{n}, I_{n}, \ldots$ such
that
 $I_{n}\in\mathcal I$, $x_{n}\in I_{n-1}\setminus I_{n}$ and $I_{n}\subseteq \{x_{n}\}\bigvee I$ for all $I\in\mathcal I$, all $n<\omega$.
Members of this sequence being defined for all $n', n'<n$, observe
that the set $\mathcal I_{n}:=\{I\in\mathcal I: I\subseteq
I_{n-1}\}$ being infinite is non-separating, hence there are $I\in
\mathcal I_{n}$ and $x\in I_{n-1}\setminus I$ such that $I\subseteq
\{x\}\bigvee J$ for all $J\in\mathcal I_{n}$. Set $I_{n}:=I$ and
$x_{n}:=x$. Next, we define a sequence $y_{0}:=x_{0}, \ldots, y_{n},
\ldots$ such that for every $n\geq 1$:\\ $a_{n})$ $x_{n} \leq
y_{n}\in I_{n-1}$; \\$b_{n})$ $y_{n} \not\leq y_{0}\vee y_{n-1}$;\\
$c_{n})$  $y_{j}\leq y_{i}\vee y_{n}$ for every $i\leq j\leq n$.\\
Suppose $y_{0}, \ldots, y_{n-1}$ defined for some $n$, $n\geq 1$.
Since $I_{n-1}$ is unbounded, we may select $z\in I_{n-1}$ such that
$z\not\leq y_{0}\vee\ldots\vee y_{n-1}$. If $n=1$, we set
$y_{1}:=x_{1}\vee z$. Suppose $n\geq 2$. Let $0\leq j\leq n-2$.
Since $y_{j+1}\vee\ldots\vee y_{n-1}\in I_{j}\subseteq \{
x_{j}\}\bigvee I_{n-1}$  we may select $t_{j}\in I_{n-1}$ such that
$y_{j+1}\vee\ldots\vee y_{n-1}\leq x_{j}\vee t_{j}$. Set
$t:=t_{0}\vee\ldots\vee t_{n-2}$ and $y_{n}:=x_{n}\vee z\vee t$.\\
Let $f: \Omega(\omega^*)\to P$ be defined by
$f(i,j):=y_{i}\vee y_{j}$ for all $(i,j)$, $i<j<\omega$.\\
Condition $c_{n})$ insures that $f$ is join-preserving. Indeed, let
$(i,j), (i',j')\in\Omega(\omega^*)$. We have $(i,j)\vee
(i',j')=(i\wedge i',j\vee j')$ hence $f((i,j)\vee (i',j'))=f(i\wedge
i',j\vee j')=y_{i\wedge i'}\vee y_{j\vee j'}$. If $F$ is a finite
subset of $\omega$ with minimum $a$ and maximum $b$ then conditions
$c_{n})$ force $\bigvee\{y_{n}: n\in F\}=y_{a}\vee y_{b}$. If
$F:=\{i,j,i',j'\}$ then, taking account of $i<j$ and $i'<j'$, we
have $f(i,j)\vee f(i',j')=y_{i}\vee y_{j}\vee y_{i'}\vee
y_{j'}=y_{i\wedge i'}\vee y_{j\vee j'}$. Hence
$f((i,j)\vee (i',j'))=f(i,j)\vee f(i',j')$, proving our claim.  \\
Next, $f$ is one-to-one. Let $(i,j), (i',j')\in\Omega(\omega^*)$
such that $f(i,j)=f(i',j')$, that is $y_{i}\vee y_{j}=y_{i'}\vee
y_{j'}$ $(1)$. Suppose $j<j'$. Since $0\leq i<j$, Condition $c_{j})$
implies $y_{i}\leq y_{0}\vee y_{j}$. In the other hand,  since
$0\leq j\leq j'-1$, Condition $c_{j'-1})$ implies $y_{j}\leq
y_{0}\vee y_{j'-1}$. Hence $y_{i}\vee y_{j}\leq y_{0}\vee y_{j'-1}$.
From $(1)$ we get $y_{j'}\leq y_{0}\vee y_{j'-1}$, contradicting
Condition $b_{j'})$. Hence $j'\leq j$. Exchanging the roles of
$j,j'$ gives $j'\leq j$ thus $j=j'$. If $i<i'$ then, Conditions
$a_{i'})$ and $a_{j'})$ assure $y_{i'}\in I_{i'-1}$ and $y_{j'}\in
I_{j'-1}$. Since $I_{j'-1}\subseteq I_{i'-1}$ we have $y_{i'}\vee
y_{j'}\in I_{i'-1}$. In the other hand $x_{i}\not\in I_{i}$ and
$x_{i}\leq y_{i}\vee y_{j}$ thus $y_{i}\vee y_{j}\not\in I_{i}$.
From $I_{i'-1}\subseteq I_{i}$, we have $y_{i}\vee y_{j}\not\in
I_{i'-1}$, hence $y_{i}\vee y_{j}\not = y_{i'}\vee y_{j'}$ and
$i'\leq i$. Similarly we get $i\leq i'$. Consequently $i=i'$.
\end{proof}

\subsection{Proof of Proposition \ref{w-f}}
If $\underline \Omega(\omega^*)$ embeds in $[\omega]^{<\omega}$ then
$[\omega]^{<\omega}$ contains a non-separating $\omega^*$-chain of
ideals. This is impossible: a  non-separating chain of ideals of
$[\omega]^{<\omega}$ has necessarily a least element. Indeed, if the
pair $x,I$ ($x\in [\omega]^{<\omega}$, $I\in  \mathcal I$) witnesses
the fact that the chain $\mathcal I$ is non-separating  then there
are at most $\mid x\mid+1$ ideals belonging to  $\mathcal I$ which
are included in $I$ (note that the set $\{\cup I \setminus \cup J:
J\subseteq I, J\in \mathcal I\}$  is a chain of subsets of $x$). The
proof of the general case requires more care. If $\underline
\Omega(\omega^*)$ embeds in $I_{<\omega}(Q)$ as a join-semilattice
then we may find a sequence $x_{0}, I_{0}, \ldots, x_{n}, I_{n},
\ldots$ such that $I_{n}\subset I_{n-1}\in J(I_{<\omega}(Q))$,
$x_{n}\in I_{n-1} \setminus I_{n}$ and $I_{n} \subseteq
\{x_{n}\}\bigvee I_{m}$ for every $n<\omega$ and every $m<\omega$.
Set $I_{\omega}:=\bigcap\{ I_{n}: n<\omega\}$, $\overline
I_{n}:=\cup I_{n}$ for every $n\leq\omega$, $Q':=Q\setminus
\overline I_{\omega}$ and $y_{n}:=x_{n}\setminus \overline
I_{\omega}$ for every $n<\omega$. We claim that $y_{0}, \ldots,
y_{n}, \ldots$ form a strictly descending sequence in
$I_{<\omega}(Q')$. According to Property $a)$  stated in Theorem
\ref {wellfounded},  $Q'$, thus $Q$, is not well-founded.\\ First,
$y_{n}\in I_{<\omega}(Q')$. Indeed, if $a_{n}\in [Q]^{<\omega}$
generates $x_{n}\in I_{<\omega}(Q)$ then, since $\overline
I_{\omega}\in I(Q)$, $a_{n}\setminus \overline I_{\omega}$ generates
$x_{n}\setminus \overline I_{\omega}\in I(Q')$. Next,
$y_{n+1}\subset y_{n}$. It suffices to prove that the following
inclusions hold: $$x_{n+1}\cup \overline I_{\omega}\subseteq
\overline I_{n}\subset x_{n}\cup \overline I_{\omega}$$ Indeed,
substracting $\overline I_{\omega}$, from the sets  figuring above,
we get:\\ $$y_{n+1}=(x_{n+1}\cup \overline I_{\omega})\setminus
 \overline I_{\omega}\subset (x_{n}\cup \overline I_{\omega})\setminus \overline I_{\omega}=y_{n}$$
The first inclusion is obvious. For the second note that, since
$J(I_{<\omega}(Q))$ is isomorphic to $I(Q)$, complete distributivity
holds, hence  with the hypotheses on the sequence $x_{0}, I_{0},
\ldots, x_{n}, I_{n}, \ldots$  we have $I_{n} \subseteq \bigcap\{\{
x_{n}\}\bigvee I_{m}: m<\omega\}=\{ x_{n}\}\bigvee\bigcap\{ I_{m}:
m<\omega\}=\{ x_{n}\}\bigvee I_{\omega}$, thus $\overline
I_{n}\subset x_{n}\cup \overline I_{\omega}$.
\endproof\\
\begin{remark} One can deduce the fact that $\Omega(\omega^*)$ does not embed as  a join-semilattice in
$[\omega]^{<\omega}$ from the fact that it contains a strictly
descending chain of completely meet-irreducible \index{completely
meet-irreducible} ideals (namely the chain $\mathcal I:=\{I_{n}:
n<\omega\}$ where $I_{n}:=\{(i,j): n\leq i<j<\omega\}$) (see
Proposition \ref {lem03}) but this fact by itself does not prevent
the existence of some well-founded poset $Q$ such that
$\Omega(\omega^*)$ embeds as a join semilattice in
$I_{<\omega}(Q)$. \end{remark}

\subsection{Proof of Theorem \ref{thm4}}
In terms of join-semilattices and  ideals, result becomes this:
let $P$ be a join-semilattice, then $J(P)$ is well-founded if and
only if $P$ is well-founded  and contains no join-subsemilattice
isomorphic to $\Omega(\omega^*)$ or to $[\omega]^{<\omega}$.\\

The proof goes as follows. Suppose that $J(P)$ is not well-founded. If some $\omega^*$-chain in $J(P)$ is separating then, according to
Lemma \ref{independent}, $P$
contains an infinite independent set. From Theorem \ref {tm2.1}, it contains a join-subsemilattice isomorphic to $[\omega]^{<\omega}$.
If no $\omega^*$-chain in $J(P)$ is separating, then all the infinite subchains of an arbitrary $\omega^*$-chain are non-separating.
From Lemma \ref{w*},
 either $\omega^*$ or
$\Omega ( \omega^*)$ embed in $P$ as a join-semilattice. The
converse is obvious. \endproof\\

\section[Join-subsemilattices of $I_{<\omega}(Q)$]{Join-subsemilattices of $I_{<\omega}(Q)$ and proof of Theorem  \ref{thmwf}}
In this section, we consider  join-semilattices  which embed in
join-semilattices of the form $I_{<\omega}(Q)$. These are easy  to
characterize internally (see Proposition \ref {lem04}). This is
also the case if the posets $Q$ are antichains (see Proposition
\ref {lem03}) but does not go so well if the posets  $Q$ are
well-founded (see Lemma \ref{counterexample}).

Let us recall  that if $P$ is a  join-semilattice, an element $x\in
P$ is  {\it join-prime}\index{join-prime} (or prime if there is no
confusion), if it is distinct from the least element $0$, if any,
and if $x\leq a\vee b$ implies $x\leq a$ or $x\leq b$. This amounts
to the fact that $P\setminus \uparrow x$ is an ideal . We denote
$\J_{pri}(P)$, the set of join-prime  members of $P$. We recall that
$\J_{pri}(P)\subseteq \J_{irr}(P)$; the equality holds provided that
$P$ is a distributive lattice. It also holds if $P=
{I}_{<\omega}(Q)$. Indeed:

\begin{fact}For  an arbitrary poset $Q$, we have:
\begin{equation}\label{eqirr}
 {\J}_{irr}( I_{<\omega} (Q))=\J_{pri}( I_{<\omega} (Q))= down (Q)
 \end{equation}
\end{fact}
\begin{fact} For a poset $P$, the following properties are equivalent:
\begin{itemize}
\item $P$ is isomorphic to $ {I}_{<\omega}(Q)$ for some poset $Q$;
\item $P$ is a join-semilattice with a least element  in which  every element is a finite  join of primes.
\end{itemize}
\end{fact}

\begin{proof} Observe that the primes in $I_{<\omega} (Q)$, are the $\downarrow x$, $x\in
Q$. Let $I\in I_{<\omega} (Q)$ and $F\in [Q]^{<\omega}$ generating
$I$, we have $I=\cup \{\downarrow x: x\in F\}$ . Conversely, let
$P$ be a join-semilattice with a $0$. If every element in $P$ is a
finite join of primes, then $P\cong I_{<\omega} (Q)$ where
$Q:=\J_{pri}(P)$.
\end{proof}

Let $L$ be  a complete lattice \index{complete lattice}. For $x\in
L$ , set $x^{+}:=\bigwedge\{y\in L: x<y\}$. We recall that $x\in L$
is {\it completely meet-irreducible} \index{completely
meet-irreducible} if $x=\bigwedge X$ implies $x\in X$, or
-equivalently- $x\neq x^{+}$.  We denote $\triangle (L)$ the set of
completely meet-irreducible members of $L$. We recall the following
Lemma.
\begin{lemma}\label{lem01} Let $P$ be a join-semilattice,  $I\in J(P)$ and $x\in P$.
Then $x\in I^{+}\setminus I$ if and only if  $I$ is a maximal ideal of $P\setminus
\uparrow x$.
\end{lemma}

\begin{proposition}\label{lem04} Let $P$ be a join-semilattice. The
following properties are equivalent:
\begin {enumerate}[(i)]
\item $P$ embeds in
$I_{<\omega}(Q)$,  as a join-semilattice, for some poset $Q$;
\item $P$ embeds in $I_{<\omega}(J(P))$ as a join-semilattice;
\item $P$ embeds in $I_{<\omega}(\triangle(J(P)))$ as a join-semilattice;
\item  For every $x\in P$, $P\setminus \uparrow x$ is a finite
union of ideals.
\end{enumerate}
\end{proposition}

\begin{proof}
$(i)\Rightarrow (iv)$ Let $\varphi$ be an embedding from $P$ in
$P':= I_{<\omega}(Q)$. We may suppose that $P$ has a least element
$0$ and that $\varphi (0)=\emptyset$ (if $P$ has no least element,
add one, say $0$,  and set $\varphi (0):=\emptyset $; if $P$ has a
least element, say $a$,  and $\varphi (a)\not =\emptyset$,  add to
$P$ an element $0$ below $a$ and set $\varphi (0):= \emptyset$).
For $J'\in \mathfrak { P}(P')$,  let $\varphi^{-1}(J'):= \{x\in P:
\varphi (x)\in J'\}$. Since $\varphi$ is order-preserving,
$\varphi^{-1}(J')\in I(P)$ whenever $J'\in I(P')$ ; moreover,
since $\varphi$ is join-preserving, $\varphi^{-1}(J')\in J(P)$
whenever $J'\in J(P')$. Now, let  $x\in P$. We have
$\varphi^{-1}(P'\setminus \varphi(x)):= P\setminus \uparrow x$.
Since $\varphi(x)$ is a finite join of primes, $P'\setminus
\uparrow  \varphi (x)$ is a finite union of ideals. Since their
inverse images are ideals, $P\setminus \uparrow x$ is a finite
union of ideals too.

$(iv)\Rightarrow (iii)$  We use  the well-known method for representing a poset by a family of sets.

\begin{fact}\label{representation}Let $P$ be a poset and $Q\subseteq I(P)$. For $x\in P$ set
$\varphi _{Q}(x):=\{J\in Q: x\not \in J\}$.  Then:
\begin{enumerate}[(a)]
\item  $\varphi _{Q}(x) \in I(Q)$;
\item  $\varphi _{Q}:P\rightarrow I(Q)$ is an order-preserving \index{order-preserving} map;
\item $\varphi _{Q}$  is  an order-embedding\index{order-embedding} if and only if for every $x, y\in P$ such that $x\not \leq y$
there is some $J\in Q$ such that $x\not \in J$ and $y\in J$.
\end{enumerate}
\end{fact}

Applying  this to $Q:=\triangle( J(P))$ we get  immediately that
$\varphi _{Q}$ is join-preserving \index{join-preserving}. Moreover,
$\varphi _{Q}(x)\in I_{<\omega}(Q)$ if and only if $P\setminus
\uparrow x$ is a finite union of ideals. Indeed, we have
$P\setminus \uparrow x= \cup \varphi _{Q}(x)$, proving that
$P\setminus \uparrow x$ is a finite union of ideals provided that
$\varphi _{Q}(x)\in I_{<\omega}(Q)$. Conversely, if $P\setminus
\uparrow x$ is a finite union of ideals, say $I_{0}, \dots, I_{n}$,
then since ideals are prime members of $I(P)$, every ideal included
in $I$ is included in some $I_i$, proving that $\varphi _{Q}(x)\in
I_{<\omega}(Q)$. To conclude, note that if $P$ is a join-semilattice
then $\varphi _{Q}$ is join-preserving.

$(iii)\Rightarrow (ii)$ Trivial.

$(ii)\Rightarrow (i)$ Trivial.

\end{proof}

\begin{corollary}\label{corthm16} If a   join-semilattice $P$ has no infinite antichain, it embeds in
$I_{<\omega}(J(P))$ as a join-subsemilattice.
\end{corollary}
\begin{proof} As is well known, if  a poset has no infinite antichain then every initial segment
is a finite union of ideals (cf \cite{Erdos-Tarski}, see also
\cite{fraissetr} 4.7.3 pp. 125). Thus Proposition \ref{lem04}
applies.
\end{proof}

Another corollary of Proposition  \ref{lem04}  is the following.

\begin{corollary}\label{lem00}  Let $P$ be a join-semilattice. If  for every $x\in P$,  $P\setminus \uparrow x$ is a finite union of ideals and $ \triangle (J(P))$ is well-founded   then   $P$ embeds  as a join-subsemilattice in $I_{<\omega}(Q)$,  for some well-founded poset $Q$.
\end{corollary}

The converse does not hold:

\begin{example}\label{counterexample} There is a bipartite poset $Q$ such that $I_{<\omega}(Q)$ contains a
join-semilattice $P$ for which  $ \triangle (J(P))$ is not
well-founded.
\end{example}

\begin{proof}
Let $\underline 2:= \{0,1\}$ and $Q:=\mathbb{N}\times
\underline{2}$. Order $Q$ in such a way that $(m,i)<(n,j)$ if
$m>n$ in $\N$ and $i<j$ in $\underline 2$.

Let $P$ be the set
of subsets $X$ of $Q$ of the form $X:=F\times \{0\}\cup G\times
\{1\}$ such that $F$ is a  non-empty final segment of $\mathbb{N}$, $G$ is a
non-empty finite subset of $\mathbb{N}$ and
\begin{equation}\label{contrex}
min(F)-1\leq min(G)\leq
min(F)
\end{equation}
where $min(F)$ and $min(G)$  denote the least element of $F$  and
$G$ w.r.t. the natural order on $\mathbb{N}$. For each $n\in \N$,
let $I_{n}:=\{ X\in P:  (n,0) \not \in X\}$.

{\bf Claim }
\begin{enumerate}
\item $Q$ is bipartite and  $P$ is a join-subsemilattice of $I_{<\omega}(Q)$.
\item The $I_{n}$'s form a strictly descending sequence \index{descending sequence} of members of
$ \triangle (J(P))$.
\end{enumerate}

{\bf Proof of the Claim }

1. The poset $Q$ is decomposed into two antichains, namely $\N\times
\{0\}$ and $\N\times \{1\}$ and for this raison is called {\it
bipartite}.
 Next, $P$ is a
subset of $I_{<\omega}(Q)$. Indeed, Let $X\in P$. Let $F, G$ such
that $ X=F\times \{0\}\cup G\times \{1\}$.  Set $G':=
G\times\{1\}$. If $min(G)=min(F)-1$, then $X=\downarrow G'$
whereas if  $min(G)=min(F)$ then $X=\downarrow G' \cup \{(min(F),
0)\}$. In both cases    $X\in I_{<\omega}(Q)$.  Finally, $P$ is a
join-semilattice. Indeed,  let $X,X'\in P$ with $X:=F\times
\{0\}\cup G\times \{1\}$ and $X':=F'\times \{0\}\cup G'\times
\{1\}$. Obviously $X\cup X'=(F\cup F')\times \{0\}\cup (G\cup
G')\times \{1\}$. Since $X, X'\in P$, $F\cup F'$ is a non-empty
final segment of $\mathbb{N}$ and $G\cup G'$ is a non-empty finite
subset of $\mathbb{N}$. We have $min(G\cup
G')=min(\{min(G),min(G')\})\leq min(\{min(F),min(F')\})=min(F\cup
F')$ and similarly $min(F\cup F')-1=min\{min(F),
min(F')\}-1=min\{min(F)-1, min(F')-1\} \leq min
\{min(G),min(G')\}=min(G\cup G')$, proving that inequalities as in
(\ref {contrex}) hold. Thus $X\cup X'\in I_{<\omega}(Q)$.

2.  Due to its definition, $I_{n}$ is an non-empty initial segment
of $P$ which is closed under finite unions, hence $I_{n}\in J(P)$.
Let $X_{n}:= \{(n,1), (m, 0): m\geq n+1\} $ and $Y_{n}:= X_{n}\cup
\{(n,0)\}$. Clearly, $X_{n} \in I_{n}$ and $Y_{n} \in P$. We claim
that $I_{n}^{+}=I_{n}\bigvee \{Y_{n}\}$. Indeed, let $J$ be an
ideal containing strictly $I_{n}$. Let $Y:=\{m\in \mathbb{N}:
m\geq p\}\times\{0\}\cup G\times\{1\} \in J\setminus I_{n}$. Since
$Y\not\in I_{n}$, we have  $p\leq n$ hence  $Y_{n}\subseteq Y\cup
X_{n}\in J$. It follows that $Y_{n}\in J$, thus
$I_{n}^{+}\subseteq J$, proving our claim. Since $I_{n}^{+}\not=
I_{n}$, $I_{n}\in\triangle (J(P))$. Since, trivially,
$I_{n}^{+}\subseteq I_{n-1}$ we have $I_{n}\subset I_{n-1}$,
proving that the $I_n$'s form a strictly descending sequence.
\end{proof}

Let  $E$ be a set and $\mathcal{F}$ be a subset of $\mathfrak
P(E)$, the power set of $E$. For $x\in E$, set $\mathcal {F}_{\neg
x}:=\{F\in \mathcal{F}: x\not\in F\}$ and for $X \subset \mathcal
F$ , set $\overline X:= \bigcup X$. Let $\mathcal{F}^{<\omega}$
(resp. $\mathcal{F}^\cup$) be the collection of finite (resp.
arbitrary) unions of members of $\mathcal{F}$. Ordered by
inclusion, $\mathcal{F}^\cup$ is a  complete lattice
\index{complete lattice}, the least element and the largest
element being  the empty set and $\bigcup\mathcal{F}$,
respectively.

\begin{lemma}\label{cl2} Let $Q$ be a poset,  $\mathcal{F}$ be  a subset of $I_{<\omega}(Q)$ and $P:= \mathcal{F}^{<\omega}$ ordered by inclusion.
\begin{enumerate}[{(a)}]

\item The map $X\rightarrow \overline X$ is an isomorphism  from  $J(P)$ onto  $\mathcal{F}^\cup$ ordered by inclusion.

\item  If $I\in \triangle (J(P))$
then there is
some $x\in Q$ such that $I=P_{\neg x}$.

\item If $\downarrow q$ is finite for every  $q\in Q$ then $\overline{I^+}\setminus\overline{I}$ is finite for every  $I\in J(P)$ and
the set $\varphi_{\triangle}(X):= \{I\in {\triangle}(J(P)) : X\not\in I\}$ is finite for every $ X\in P$.

\end{enumerate}
\end{lemma}

\begin{proof}

\noindent $(a)$ Let $I$ and $J$ be two ideals of $P$. Then $J$
contains $I$ if and only if $\overline{J}$ contains $\overline{I}$.
Indeed, if $I\subseteq J$ then, clearly $\overline{I}\subseteq
\overline{J}$. Conversely, suppose $\overline{I}\subseteq
\overline{J}$. If $X\in I$, then $X\subseteq \overline{I}$, thus
$X\subseteq \overline{J}$. Since $X\in I_{<\omega}(Q)$, and
$X\subseteq \overline{J}$, there are $X_{1}, \ldots, X_{n}\in J$
such that $X\subseteq Y=X_{1}\cup \ldots \cup X_{n}$. Since $J$ is
an ideal $Y\in J$. It follows that $X\in J$.

\noindent $(b)$ Let  $I\in \triangle (J(P))$. From $(a)$,  we have
$\overline{I}\subset\overline{I^+}$. Let $x\in
\overline{I^+}\setminus\overline{I}$.  Clearly  $P_{\neg x}$ is an
ideal containing $I$. Since  $x\not\in \overline{P_{\neg x}}$,
$P_{\neg x}$ is distinct from $I^{+}$. Hence $P_{\neg x}=I$. Note
that the converse of assertion $(b)$ does not  hold in general.

\noindent $(c)$ Let
$I\in \triangle (J(P))$ and $X\in I^{+}\setminus I$.  We have $\{X\}\bigvee
I= I^+$, hence  from$(a)$ $\overline{\{X\} \bigvee
I}= \overline{I^+}$. Since  $\overline{\{X\} \bigvee
I}=X\cup \overline{I}$ we have
$\overline{I^+}\setminus\overline{I}\subseteq X$. From our hypothesis on $P$,  $X$ is
finite, hence $\overline{I^+}\setminus\overline{I}$ is finite. Let $X\in P$.  If  $I\in \varphi_{\triangle}(X)$ then according  to $(b)$
there is some $x\in Q$ such that $I= P_{\neg x}$. Necessarily $x\in X$.
Since $X$ is finite, the number of these $I$'s is finite.
\end{proof}

\begin{proposition}\label{lem03} Let $P$ be a join-semilattice.
The
following properties are equivalent:
\begin{enumerate}[(i)]
\item $P$ embeds in
$[E]^{<\omega}$ as a join-subsemilattice for some set $E$;
\item   for every $x\in P$, $\varphi_{\triangle}(x)$ is finite.
\end{enumerate}
\end{proposition}
\begin{proof}
$(i)\Rightarrow (ii)$ Let $\varphi$ be an embedding from $P$ in
$[E]^{<\omega}$ which preserves joins.  Set $\mathcal F:= \varphi
(P)$. Apply   part $(c)$  of Lemma \ref{cl2} .
 $(ii)\Rightarrow (i)$ Set $E:= \triangle (J(P))$. We have  $\varphi_{\triangle}(x)\in [E]^{<\omega}$.  According to Fact \ref{representation}
 and Lemma \ref{lem01},  the map $\varphi_{\triangle}: P \rightarrow [E]^{<\omega}$ is an embedding preserving joins.\end{proof}

 \begin{corollary}\label{sierpinski} Let $\beta$ be a countable  order type. If a proper initial segment
 contains infinitely many non-principal  initial segments then no sierpinskisation  $P$ of  $\beta$ with $\omega$
 can embed  in $[\omega]^{<\omega}$ as a join-semilattice (whereas it embeds as a poset).
 \end{corollary}
 \begin{proof}
According to Proposition \ref{lem03}  it suffices to prove  that $P$ contains some $x$ for which $\varphi_{\Delta}(x)$ is infinite.

 Let $P$ be a sierpinskisation  of $\beta$  and  $\omega$. It is  obtained as the intersection of two linear orders $L$,   $L'$ on the same set  and having respectively order type  $\beta$  and  $\omega$.   We may suppose that the ground set is  $\N$ and $L'$ the natural order.

{\bf  Claim 1}
 A non-empty subset $I$ is a non-principal ideal  of
$P$ if and only if this is a non-principal initial segment of $L$.

{\bf Proof of  Claim 1} Suppose that  $I$ is  a non-principal
initial segment of $L$. Then, clearly, $I$ is an initial segment
of $P$. Let us check that $I$  is up-directed. Let $x, y\in I$;
since $I$ is non-principal in $L$, the set $A:= I\cap\uparrow_{L}
x \cap \uparrow_{L} y$ of upper-bounds of $x$ and $y$ w.r.t.  $L$
which belong to  $I$ is infinite; since $B:= \downarrow_{L'} x
\cup \downarrow_{L'} y$ is finite, $A\setminus B$ is non-empty. An
arbitrary element $z\in A\setminus B$ is an upper bound of $x,y$
in $I$ w.r.t. the poset  $P$ proving that $I$ is up-directed.
Since $I$ is infinite, $I$ cannot have a largest element in $P$,
hence $I$ is a non-principal ideal of $P$. Conversely, suppose
that $I$ is  a non-principal ideal of $P$.   Let us check that $I$
is an initial segment of $L$. Let $x\leq_{L} y$ with $y\in I$.
Since $I$ non-principal in $P$, $A:=\uparrow_{P} y\cap I$ is
infinite; since $B:= \downarrow_{L'} x \cup\downarrow_{L'} y$ is
finite, $A\setminus B$ is non-empty. An arbitrary element of
$A\setminus B$ is an upper bound of $x$ and $y$ in $I$ w.r.t. $P$.
It  follows that $x\in I$. If $I$ has a largest element w.r.t. $L$
then such an element must be maximal in $I$ w.r.t. $P$, and since
$I$ is an ideal, $I$ is a principal ideal, a contradiction.

{\bf Claim 2} Let  $x\in \N$. If there is a non-principal ideal of
$L$ which does not contain $x$, there is a maximal one, say $I_x$.
If $P$ is a join-semilattice, $I_x\in \Delta(P)$.

{\bf Proof of Claim 2} The first part follows from Zorn's Lemma. The second part follows from Claim 1 and
 Lemma  \ref{lem01}.

 {\bf Claim 3} If an initial segment $I$ of $\beta$ contains infinitely many non-principal initial segments
 then there is an infinite sequence $(x_n)_{n<\omega}$ of elements of $I$ such that the  $I_{x_n}$'s  are all distinct.

 {\bf Proof of Claim 3} With Ramsey's theorem obtain a sequence $(I_n)_{n<\omega}$ of non-principal initial segments which is either strictly
 increasing \index{increasing sequence} or strictly decreasing \index{decreasing sequence}.
 Separate two successive members by some element $x_n$ and apply the first part of Claim 2.

If we pick $x\in \N \setminus I$ then it follows from Claim 3 and the second part of Claim 2 that
$\varphi_{\Delta}(x)$ is infinite.
\end{proof}

\begin{example}\label{ex:ordinal} If $\alpha$ is a countably infinite order type distinct from  $\omega$, $\Omega(\alpha)$ is not embeddable  in $[\omega]^{<\omega}$ as a join-semilattice.
\end{example}
Indeed,  $ \Omega(\alpha)$ is a sierpinskisation of $ \omega\alpha$ and $\omega$.  And  if $\alpha$ is distinct from  $\omega$,
 $\alpha$ contains  some element which majorizes  infinitely many others.  Thus $\beta:= \omega\alpha$ satisfies
 the hypothesis of Corollary \ref{sierpinski}.

 Note that on an other hand, for every ordinal $\alpha\leq \omega$, there are  representatives of  $ \Omega(\alpha)$ which are  embeddable in $[\omega]^{<\omega}$ as  join-semilattices.

\begin{theorem}\label {finitesubsets}
Let $Q$ be a well-founded poset and let $\mathcal{F}\subseteq I_{<\omega}(Q)$. The following properties are equivalent:
\begin{itemize}
\item[$1)$] $\mathcal{F}$ has no infinite antichain;
\item[$2)$] $\mathcal{F}^{<\omega}$ is wqo;
\item[$3)$] $J(\mathcal{F}^{<\omega})$ is topologically scattered;
\item[$4)$] $\mathcal{F}^\cup$ is order-scattered;
\item[$5)$] $\mathfrak{P}(\omega)$ does not embed in $\mathcal{F}^\cup$;
\item[$6)$] $\lbrack \omega\rbrack^{<\omega}$  does not embed in $\mathcal{F}^{<\omega}$;
\item[$7)$] $\mathcal{F}^\cup$ is well-founded.
\end{itemize}
\end{theorem}

\begin{proof} We prove the following chain of implications:
$$1)\Longrightarrow 2)\Longrightarrow 3)\Longrightarrow
4)\Longrightarrow 5)\Longrightarrow 6)\Longrightarrow 7)
\Longrightarrow 1)$$ $1)\Longrightarrow 2)$. Since $Q$ is
well-founded then, as mentioned in $a)$ of Theorem
\ref{wellfounded}, $I_{<\omega}(Q)$ is well-founded.  It follows
first that $\mathcal{F}^{<\omega}$ is well-founded, hence from
Property $c)$ of Theorem  \ref{wellfounded},  every member of
$\mathcal{F}^{<\omega}$ is a finite join of join-irreducibles. Next,
as a subset of $\mathcal{F}^{<\omega}$, $\mathcal {F}$ is
well-founded, hence wqo according to  our hypothesis. The set of
join-irreducible members of $\mathcal{F}^{<\omega}$ is wqo as a
subset of $\mathcal {F}$.
 From Property $d)$  of Theorem  \ref{wellfounded}, $\mathcal{F}^{<\omega}$ is wqo \\
$2)\Longrightarrow 3)$. If $\mathcal{F}^{<\omega}$ is wqo then
$I(\mathcal{F}^{<\omega})$ is well-founded (cf. Property  ($b)$ of
Theorem \ref{wellfounded}). If follows that
$I(\mathcal{F}^{<\omega})$ is topologically scattered
(cf.\cite{misl});  hence all its subsets are topologically
scattered, in particular $J(\mathcal{F}^{<\omega})$. \\
$3)\Longrightarrow4)$. Suppose that $\mathcal{F}^\cup$ is not
ordered scatered. Let $f: \eta \rightarrow \mathcal{F}^\cup$ be an
embedding. For $r\in \eta$ set $\check f(r)=\bigcup \{f(r'):
r'<r\}$. Let $X:=\{\check f(r): r<\eta\}$. Clearly $X\subseteq
\mathcal{F}^\cup$. Furthermore $X$ contains no isolated point
(Indeed, since $\check f(r)=\bigcup \{\check f(r'): r'<r\}$, $\check
f(r)$ belongs to the topological closure of $\{\check f(r'):
r'<r\}$). Hence $\mathcal{F}^\cup$ is not topologically scatered.\\
$4)\Longrightarrow 5)$. Suppose that $\mathfrak{P}(\omega)$ embeds
in $\mathcal{F}^\cup$. Since $\eta\leq \mathfrak{P}(\omega)$, we
have
$\eta\leq \mathcal{F}^\cup$.\\
$5)\Longrightarrow 6)$. Suppose that $[\omega]^{<\omega}$ embeds in
$\mathcal{F}^{<\omega}$, then $J([\omega]^{<\omega})$ embeds in
$J(\mathcal{F}^{<\omega})$. Lemma \ref{cl2} assures that
$J(\mathcal{F}^{<\omega})$ is isomorphic to $\mathcal{F}^{\cup}$. In
the other hand $J([\omega]^{<\omega})$ is isomorphic to
$\mathfrak{P}(\omega)$. Hence $\mathfrak{P}(\omega)$ embeds in
$\mathcal{F}^{\cup}$.\\
$6)\Longrightarrow 7)$. Suppose $\mathcal{F}^{\cup}$ not
well-founded. Since $Q$ is well-founded, $a)$ of Theorem
\ref{wellfounded} assures $I_{<\omega}(Q)$ well-founded, but
$\mathcal{F}^{<\omega}\subseteq I_{<\omega}(Q)$, hence
$\mathcal{F}^{<\omega}$ is well-founded. Furthermore, since
$I_{<\omega}(Q)$ is closed under finite unions, we have
$\mathcal{F}^{<\omega}\subseteq I_{<\omega}(Q)$, Proposition \ref
{w-f} implies that $\underline\Omega (\omega^{*})$ does not embed in
$\mathcal{F}^{<\omega}$.
From Theorem \ref{thm4}, we have $\mathcal{F}^{<\omega}$ not well-founded.\\
$7)\Longrightarrow 1)$. Clearly, $\mathcal F$ is well-founded. If
$F_{0}, \dots, F_{n}\dots $ is an infinite antichain of members of
$\mathcal{F}$, define $f(i,j):[\omega]^{2}\rightarrow Q$, choosing
$f(i,j)$ arbitrary in $Max(F_{i})\setminus F_{j}$. Divide
$[\omega]^3$ into $R_{1}:=\{(i,j,k)\in [\omega]^3: f(i,j)=f(i,k)\}$
and $R_{2}:=[\omega]^3\setminus R_{1}$.
 From Ramsey's theorem,  cf. \cite {rams}, there is some infinite subset $X$ of $\omega$ such that $[X]^{3}$ is included in $R_{1}$ or in $R_{2}$.
The inclusion in $R_{2}$ is impossible since $\{f(i,j ): j<\omega
\}$, being included in $Max(F_{i})$,  is finite for every $i$. For
each $i\in X$, set $G_{i}:= \bigcup\{F_{j}: i\leq j\in X\}$. This
defines an $\omega^*$-chain in $\mathcal{F}^\cup$.
\end{proof}

\begin{remark} If $\mathcal{F}^{<\omega}$ is closed under finite
intersections then equivalence between $(3)$ and $(4)$ follows from
Mislove's Theorem mentioned in \cite{misl}.
\end{remark}
Theorem \ref {finitesubsets} above was obtained by the second author and M.Sobrani in the special case where $Q$ is an antichain \cite {pouzet, sobrani} .

 \begin{corollary} \label{provisoire}
If  $P$ is a join-subsemilattice of a join-semilattice of the form $[\omega]^{<\omega}$,
or more generally of the form $I_{<\omega} (Q)$ where $Q$ is some well-founded poset,
then $J(P)$  is well-founded if and only if $P$ has no infinite antichain.
\end{corollary}

{\bf Remark.} If, in Theorem \ref {finitesubsets} above, we suppose that $\mathcal  F$ is well-founded instead of $Q$,
all implications in the above chain hold, except $6)\Rightarrow 7)$.  A counterexample is provided by $Q:= \omega \oplus \omega^*$,
the direct sum of the chains $\omega$ and
$\omega^*$, and
$\mathcal F$, the image of $\underline \Omega (\omega^*)$ via a natural embedding.

\subsection{ Proof of  Theorem \ref {thmwf}}
 $(i)\Rightarrow (ii)$ Suppose that $(i)$ holds. Set $Q:= J(P)$. Since $P$ contains no infinite antichain,
 $P$ embeds as a join-subsemilattice in $ I_{<\omega}(Q)$ (Corollary \ref {corthm16}). From $b)$ of Theorem \ref{wellfounded} $Q$ is well-founded.
 Since $P$ has no infinite antichain, it has no infinite independent set.

  $(ii)\Rightarrow (i)$ Suppose that $(ii)$ holds. Since $Q$ is well-founded, then from $a)$ of Theorem \ref{wellfounded}, $I_{<\omega}(Q)$ is well-founded.
  Since $P$ embeds in $I_{<\omega}(Q)$, $P$ is well-founded. From our hypothesis, $P$ contains no infinite independent set.  According to  implication $(iii)\Rightarrow (i)$ of Theorem \ref{tm2.1} , it does not embed $[\omega]^{<\omega}$. From implication $6)\Rightarrow 1)$ of Theorem \ref {finitesubsets}, it has no infinite antichain. \endproof

\chapter[Length of chains in algebraic lattices]{On the length of chains in  algebraic lattices}\footnote{Les principaux r\'esultats de ce chapitre sont inclus dans l'article: I.Chakir, M.Pouzet, The length of chains in algebraic lattices, Les annales ROAD du LAID3, special issue 2008, pp 379-390 (proceedings of ISOR'08, Algiers, Algeria, Nov 2-6, 2008).} \label{chap:wellfoundedbis}

We study how the existence of a chain of a given type in an
algebraic lattice $L$ is reflected in the join-semilattice $K(L)$ of
its compact elements. We show that for every  chain $\alpha$ of size
$\kappa$, there is a set $\B$ of  at most $2^{\kappa}$
join-semilattices, each one having  a least element such that an
algebraic lattice $L$ contains no chain of order type $I(\alpha)$ if
and only if the join-semilattice $K(L )$ of its compact elements
contains no join-subsemilattice isomorphic to a member of $\B$. We
show that among the join-subsemilattices of $[\omega]^{<\omega}$
belonging to $\B$, one is embeddable in all the others. We
conjecture that if $\alpha$ is countable, there is a finite   $\B$.
We study some special cases, particularly  when $\alpha$ is an
ordinal.

\section{Introduction}
This paper is about the relationship between the length of chains in an algebraic lattice $L$ and the structure of the join-semilattice $K(L)$ of the compact elements of $L$.  We started such an investigation  in \cite{chak}, \cite{cp}, \cite{chapou}, \cite{chapou2}. We present first the motivation.

Let $P$ be an ordered set (poset). An \emph{ideal} of $P$ is any non-empty up-directed initial segment of $P$.  The set $J(P)$  of ideals of $P$, ordered by inclusion, is an interesting poset associated with $P$. For a concrete  example, if   $P:= [\kappa]^{<\omega}$   the set, ordered by inclusion, consisting of finite subsets of a set of size
$\kappa$, then $J([\kappa]^{<\omega})$ is isomorphic to $\mathfrak{P}(\kappa)$  the power set of $\kappa$ ordered by inclusion.  In \cite{chapou} we proved:
\begin{theorem} \label{thm0}
A poset $P$ contains a subset isomorphic to
$[\kappa]^{<\omega}$ if and only if $J(P)$ contains a subset isomorphic to
$\mathfrak {P}(\kappa)$.
\end{theorem}
Maximal chains in $\mathfrak {P}(\kappa)$ are of the form  $I(C)$, where  $I(C)$ is the chain of initial segments of an arbitrary chain $C$ of size $\kappa$
(cf. \cite {bonn-pouz}). Hence, if $J(P)$ contains a subset isomorphic to
$\mathfrak{P}(\kappa)$ it contains a  copy of $I(C)$ for every  chain $C$ of size $\kappa$, whereas chains in $P$ can be
 small: eg in $P:=  [\kappa]^{<\omega}$ they are finite or have order type $\omega$. What
happens if for a given order type $\alpha$, particularly a countable one,  $J(P)$ contains no chain of type $\alpha$? A partial answer was given by Pouzet, Zaguia, 1984 (cf. \cite{pz} Theorem 4, pp.62).  In order to state their result,  we recall that the  order type $\alpha$ of a chain $C$ is \emph{indecomposable} if $C$ can be embedded in each non-empty final segment of $C$.

\begin{theorem} \label{thm2}\footnote{In Theorem 4, $I(\alpha)$ is replaced  by $\alpha$.  This is due to the fact that  if $\alpha$ is a countable indecomposable order type and $P$ is a poset, $I(\alpha)$ can be embedded into $ J(P)$ if and only if $\alpha$ can be embedded  into $J(P)$. } Given an indecomposable countable order type $\alpha$, there is a finite list of ordered sets $A_{1}^{\alpha}, A_{2}^{\alpha},
\ldots, A_{n_{\alpha}}^{\alpha}$ such that for every poset $P$, the set $J(P)$ of ideals of $P$ contains no chain of type $I(\alpha)$
if and only if $P$ contains no subset isomorphic to one of the $A_{1}^{\alpha}, A_{2}^{\alpha}, \ldots,
A_{n_{\alpha}}^{\alpha}$.
\end{theorem}
If $P$ is a join-semilattice with a least element, $J(P)$ is an algebraic lattice, moreover every algebraic lattice is isomorphic to the poset $J(K(L))$ of ideals  of the join-semilattice $K(L)$ of the compact elements of $L$ (see \cite{grat}).
It is natural to ask whether  the two results above change if the poset $P$ is  a join-semilattice and if  one consider join-subsemilattices instead of subsets of $P$.

The specialization of Theorem \ref{thm0} to this case is immediate
(in fact easier to prove) and shows no difference.  Indeed {\it a  join-semilattice $P$ contains a subset isomorphic to
$[\kappa]^{<\omega}$ if and only if it contains a join-subsemilattice isomorphic to
$[\kappa]^{<\omega}$}.
The specialization of Theorem \ref{thm2} turns to be  different and is far from being  immediate.  In fact, we do not know yet whether  there is a finite list.

At first glance,  if the set  $J(P)$ of ideals of a join-semilattice  $P$ contains a chain of type  $I(\alpha)$ then according to  Theorem \ref{thm2},  $P$  contains, as a poset,  one of the   $A_{i}^{\alpha}$'s. Thus, $P$ contains,  as a join-subsemilattice,  the join-semilattice generated by  $A_{i}^{\alpha}$ in $P$. A description of  the  join-semilattices generated by the $A_{i}^{\alpha}$'s would lead to the specialization of Theorem \ref{thm0}. We have been unable to succeed in this direction. The only results we have obtained so far have been obtained by  mimicking the proof of Theorem \ref{thm2}.

In this result,  the $A_{1}^{\alpha}, \ldots, A_{n_\alpha}^{\alpha}$'s  are the "obstructions" to the existence of a chain of type  $\alpha$. Typical obstructions are built via  sierpinskisations. Let $\alpha$ be a countable chain and $\omega$ be the chain of non-negative integers.  A \emph{sierpinskisation} of $\alpha$ and $\omega$, or simply of $\alpha$,  is any poset $(S, \leq)$ such that the order on $S$ is the intersection of two linear orders on $S$, one of type $\alpha$,  the other of type $\omega$. Such a  sierpinskisation can   be obtained from a bijective map $\varphi:\omega \rightarrow \alpha$, setting $S:=\N$ and $x\leq y$ if $x\leq y$ w.r.t. the natural order on  $\N$ and  $\varphi(x)\leq \varphi(y)$ w.r.t. the order of type  $\alpha$. The proof of Theorem \ref{thm2} involves sierpinskisations of $\omega.\alpha$ and $\omega$, where $\omega.\alpha$ is the ordinal sum of $\alpha$ copies of the chain $\omega$, these sierpinskisations being obtained from bijective maps $\varphi:\omega\rightarrow \omega\alpha$ such that $\varphi^{-1}$ is order-preserving on each subset of the form $\omega\times \{\beta\}$ where $\beta\in \alpha$.  For brevity, we say that these sierpinskisations are \emph{monotonic}\index{monotonic sierpinskisation}. Augmented of a least element,  if it has none,  a monotonic sierpinskisation contains a chain  of ideals of type  $I(\alpha)$. Moreover, all posets obtained via this process can be embedded in each other. They are denoted by the same symbol $\underline {\Omega}(\alpha)$ (cf. \cite{pz} Lemma 3.4.3).  If $\alpha=\omega^*$ or $\alpha= \eta$, the list is reduced to $\underline {\Omega}(\alpha)$ and $\alpha$. For other order types, there are  other obstructions. They are  obtained by means of lexicographical sums of obstructions corresponding to chains of order type strictly less than  $\alpha$.


As we will see, among the monotonic sierpinskisations of $\omega\alpha$ and $\omega$ there are some which are join-subsemilattices of the direct product $\omega\times \alpha$ that  we call  \emph{lattice sierpinskisations}.   This suggests to prove  the specialization of Theorem \ref{thm2} along the same lines. We succeeded for $\alpha= \omega^*$. We did not for $\alpha= \eta$.

We  observe that for other countable chains, there are other obstructions that we have to take into account:

In  Theorem \ref{thm2},
$[\omega]^{<\omega}$ never occurs in the list $A_{1}^{\alpha}, A_{2}^{\alpha}, \ldots, A_{n_{\alpha}}^{\alpha}$.  We will prove in this paper that  $[\omega]^{<\omega}$ occurs necessarily in a list  if and ony if $\alpha$ is not an ordinal (Theorem \ref{ordinal}). And we will prove that if $\alpha$ is an ordinal  then $[\omega]^{<\omega}$ contains an   obstruction  which necessarily occurs in every list  of obstructions. This obstruction is the  join-semilattice $Q_{\alpha}:=I_{<\omega}(S_{\alpha})$ made of the finitely generated initial segments of $S_{\alpha}$, where $S_{\alpha}$ is  a sierpinskisation of $\alpha$ and $\omega$ (Theorem \ref{thm:qalpha}).

We conjecture that with these extra obstructions added, the specialization of Theorem \ref{thm2}  can be  obtained. We guess that the case of ordinal number is not far away. But we  are only able to give an answer in very few cases.

\section{Presentation of the results}Let $\A$, resp. $\J$,   be the class of algebraic lattices, resp. join-semilattices having a least element.  Given an order type  $\alpha$,  let $I(\alpha)$ be the order type of the chain $I(C)$ of initial segments of a chain $C$ of order type $\alpha$, let $\A_{\neg \alpha} $ be the class of algebraic lattices $L$ such that $L$ contains no chain of order type $I(\alpha)$, let $\J_{\neg \alpha}$ be the subclass of $P\in \J$ such that $J(P)\in \A_{\neg \alpha}$, let $\J_{\alpha}:= \J \setminus \J_{\neg \alpha}$ and,  for a  subcollection $\B$ of $\J$, let $Forb_{\J}(\B)$ be  the class of $P\in \J$ such that no member of $\B$ is isomorphic to a join-subsemilattice of $P$. We ask:

 \begin{question}Find $\B$ as simple as possible such that:

\begin{equation}\label {eqforbid}L\in \A_{\neg\alpha}  \; \text{if and only if}\;  K(L)\in Forb_{\J}(\B).
 \end{equation}
 or equivalently:
 \begin{equation}\label{eqforbid2}\J_{\neg \alpha}=Forb_{\J}(\B).
 \end{equation}

 \end{question}

 We  prove that we can find some $\mathbb{B}$ of size at most $2^{\mid\alpha\mid}$.
\begin{theorem}\label{thmfirst}
Let  $\alpha$ be an order type. There is a list
$\mathbb{B}$ of join-semilattices, of size at most  $2^{\mid\alpha\mid}$, such that for every join-semilattice  $P$, the lattice
$J(P)$ of ideals of  $P$ contains  no chain of order type  $I(\alpha)$ if and only if  $P$ contains no join-subsemilattice isomorphic to a member of  $\mathbb{B}$.
\end{theorem}

This is very weak. Indeed,  we cannot answer the following question.

 \begin{question}
If $\alpha$ is countable, does equation (\ref{eqforbid2}) holds for some \emph{finite} subset $\B$ of $\J$?
\end{question}

Questions and results above can be recast in terms of a quasi-order.
Let $P,P'\in \J$,  set $P\leq P'$ if $P'$ is isomorphic to a
join-subsemilattice of $P$. This relation is a quasi-order on $\J$.
If $\alpha$ is an order type, $\J_{\neg \alpha}$ is an initial
segment of $\J$, that is $P'\in\J_{\neg \alpha} $ and $P \leq P'$
imply $P \in \J_{\neg \alpha}$.  Indeed, from  $P\leq P'$ we get an
embedding from $J(P)$ into $J(P')$  which preserves arbitrary joins.
If $I(\alpha)$ was embeddable in $J(P)$ it would be embeddable in
$J(P')$, which is not the case. Hence, $P\in\J_{\neg \alpha} $.

A  class $\B$ satisfying (\ref{eqforbid}) is   {\it coinitial}   in  $\J_{\alpha}$,  in the sense that for every $P'\in \J_{\alpha}$ there is some $P\in \B$ such that $P'\leq P$.

The existence of a finite cofinal $\B$ amounts to the fact that,
w.r.t. the order on the quotient, $\J_{\alpha}$ has finitely many
minimal elements and every element of $\J_{\alpha}$ is above some.
Thus, as far we identify two join-semilattices which are embeddable
in each other as join-semilattices
\begin{lemma}  \label {lem:forbid1}$\J_{\alpha}$ contains $1+\alpha$ and $[E]^{<\omega}$, where $E$ is the domain of the chain $\alpha$.
\end{lemma}
\begin{proof} Since   $J(1+\alpha)=1+J(\alpha)=I(\alpha)$, $1+\alpha\in \J_{\alpha}$. As mentionned above,  $J([E]^{<\omega})$ is isomorphic to
$\mathfrak{P}(E)$ and since $\alpha$ is a linear order on $E$, $I(\alpha)$ is isomorphic to a maximal chain of $[E]^{<\omega}$, hence $[E]^{<\omega}\in \J_{\alpha}$. \end{proof}

As one can immediately see:
\begin{lemma}\label{lem:forbid2}
$1+\alpha$ belongs  to every $\B$ coinitial in $\J_{\alpha}$.
\end{lemma}
\begin{proof}
In terms of the quasi-order, this assertion amounts to the fact  that $1+\alpha$ is minimal in $\mathbb{\J}_{\alpha}$.
As shown in Lemma \ref{lem:forbid1},  $1+\alpha\in  \mathbb{J}_{\alpha}$. If  $Q\in \mathbb{J}_{\alpha}$ and   $Q\leq 1+\alpha$ then since $Q$ has a least element, we have
$Q=1+\beta$ with  $\beta \leq \alpha$. Since  $J(Q)=I(\beta)$, from $Q\in \mathbb{J}_{\alpha}$, we get
$I(\alpha)\leq I(\beta)$. This implies $\alpha\leq\beta$  and
$1+\alpha\leq Q$.  Hence   $1+\alpha$ is minimal in  $\mathbb{J}_{\alpha}$ as claimed.
\end{proof}

 If $\alpha$ is a finite chain, or the chain $\omega$ of non-negative integers, one can  easily see that  $1+\alpha$ is the least element of $\J_{\alpha}$.  Thus, one can take $\B=\{1+\alpha\}$.

Sierpinskisations come in the picture:

\begin{lemma}\label{lem:finitegenesierp} If $\alpha$ is a countably infinite order type and $S$ is a sierpinskisation of $\alpha$ and $\omega$ then the join-semilattice $I_{<\omega}(S)$, made of finitely generated initial segments of $S$, is isomorphic to a  join-subsemilattice of $[\omega]^{<\omega}$ and  belongs to $\J_{\alpha}$.\end{lemma}

\begin{proof}By definition,  the order on a sierpinskisation $S$ of $\alpha$  and $\omega$ has a
linear extension such that the resulting chain $\overline S$ has
order type $\alpha$. The chain $I(\overline S)$ is a maximal chain
of $I(S)$ of type $I(\alpha)$. The lattices $I(S)$ and
$J(I_{<\omega}(S))$ are isomorphic, thus $I_{<\omega}(S)\in
\J_{\alpha}$. The order on $S$ has  a linear extension of  type
$\omega$, thus  every principal initial segment of $S$ is finite and
more generally every finitely generated initial segment of $S$ is
finite. This tells us that $I_{<\omega}(S)$ is a join-subsemilattice
of $[S]^{<\omega}$. Since $S$  is countable, $I_{<\omega}(S)$
identifies to a join-subsemilattice of $[\omega]^{<\omega}$.
\end{proof}

\begin{remark}\label{remark:notordinal} If $\alpha$ is not an ordinal,
Lemma \ref{lem:finitegenesierp} tells us nothing new. Indeed, in
this case any sierpinskisation $S$ of $\alpha$ and $\omega$ contains
an infinite antichain, hence $I_{<\omega}(S)$ and
$[\omega]^{<\omega}$ are embeddable in each other as
join-semilattices.
\end{remark}

There is a  much deeper result:
\begin{theorem} \label{thm:mainalgebra}  If  $\alpha$ is a countable order type then among the join-subsemilattices $P$ of  $[\omega]^{<\omega}$  which  belong to $\J_{\alpha}$ there is one which embeds as a join-semilattice in all the others. This join-semilattice is of the form $I_{<\omega}(S_{\alpha})$ where $S_{\alpha}$ is a sierpinskisation of $\alpha$ and $\omega$.
\end{theorem}

We deduce it from Theorem \ref{ordinal} and Theorem \ref{thm:qalpha} below:
\begin{theorem}\label{ordinal}
Let $\alpha$ be a countable order type.  The join-semilattice $[\omega]^{<\omega}$ belongs  to every $\B$ coinitial in $\J_{\alpha}$  if and only if $\alpha$  is not an ordinal.
\end{theorem}

Let   $\alpha$ be an ordinal.   Set $S_{\alpha}:=\alpha$ if $\alpha<\omega$. If   $\alpha=\omega\alpha'+n$ with  $\alpha'\not =0$ and  $n<\omega$,  let  $S_{\alpha}:=\Omega(\alpha')\oplus n$  be the direct sum of  $\Omega(\alpha')$ and the chain $n$, where  $\Omega(\alpha')$ is a monotonic sierpinskisation  of  $\omega\alpha'$ and $\omega$. We note that for countably infinite $\alpha$'s, $S_{\alpha}$ is a sierpinskisation of $\alpha$ and $\omega$. We prove that $Q_{\alpha}:= I_{<\omega}(S_{\alpha})$ has the property stated in Theorem \ref{thm:mainalgebra}:

\begin{theorem} \label{thm:qalpha}If  $\alpha$ is an ordinal then   $I_{<\omega}(S_{\alpha})$ is a join-subsemilattice of
  $[\omega]^{<\omega}$  which  belongs to $\J_{\alpha}$ and is embeddable  as a join-semilattice in all join-subsemilattices  of $[\omega]^{<\omega}$  which  belongs to $\J_{\alpha}$
\end{theorem}

The deduction of Theorem \ref{thm:mainalgebra} from these two results is immediate:

If $\alpha$ is an ordinal, apply Theorem \ref{thm:qalpha}.  If $\alpha$ is not an ordinal, then according to Remark \ref{remark:notordinal} above, the conclusion of Theorem \ref{thm:mainalgebra} amounts to the fact that $[\omega]^{<\omega}$ is minimal in $\J_{\alpha}$.

The proof of Theorem \ref{thm:qalpha} is given in Section \ref{section:minimal}. We discuss here the  proof of Theorem \ref{ordinal}.

The "only if" part of Theorem \ref{ordinal} is easy. It follows from Lemma \ref{lem:finitegenesierp} and the following:
\begin{lemma} \label{lem:wqo}If $\alpha$ is  an ordinal and $S$ is a sierpinskisation of $\alpha$ and $\omega$, then  $[\omega]^{<\omega}$ is not embeddable in  $I_{<\omega}(S)$.
\end{lemma}

This simple fact relies on the important notion of
well-quasi-ordering introduced by Higman \cite{higm}. We recall that
a poset $P$ is \emph{well-quasi-ordered} (briefly w.q.o.) if every
non-empty subset $A$ of $P$ has at least a minimal element and the
number of these minimal elements is finite. As shown by Higman, this
is equivalent to the fact that $I(P)$ is well-founded \cite{higm}.

 Well-ordered set are trivially w.q.o. and, as it is well known, the direct product of finitely many w.q.o. is w.q.o.  Lemma \ref{lem:wqo} follows immediately from this. Indeed, if   $S$ is a sierpinskisation of $\alpha$ and $\omega$, it embeds in the direct product $\omega\times \alpha$.  Thus $S$  is w.q.o. and consequently  $I(S)$  is well-founded. This implies that $[\omega]^{<\omega}$ is not  embeddable in  $I_{<\omega}(S)$. Otherwise  $J([\omega]^{<\omega})$ would be embeddable in $J(I_{<\omega}(S))$, that is  $\mathfrak P(\omega)$ would be embeddable in $I(S)$. Since $\mathfrak P(\omega)$ is not well-founded, this would contradict the well-foundedness of $I(S)$.

The "if" part is based on our earlier work on well-founded algebraic lattices, that we record here.

Let   $\omega^*$ be  the chain of negative integers. Let  $\underline\Omega(\omega^*)$ be the join-semilattice obtained by adding a least element to the  set $[ \omega]^2$ of two-element subsets of $\omega$, identified to pairs  $(i,j)$, $i<j<\omega$,
ordered so that
$(i,j)\leq (i',j')$ if and only if
$i'\leq i$ and
$j\leq j'$.

We claim that  if $\alpha= \omega^*$,  $\J_{\alpha}$ has a coinitial  set made of three join-semilattices, namely $\omega^{*}$,  $[\omega]^{<\omega}$ and $\underline\Omega(\omega^*)$. Alternatively
\begin{equation}\J_{\neg \omega^*}=Forb_{\J}(\{1+\omega^*, \underline \Omega(\omega^*), [\omega]^{<\omega}\}).
\end{equation}

\begin{figure}[htbp]
\centering
\includegraphics[width=3in]{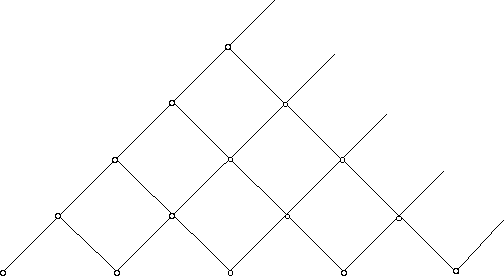}
\caption{$\underline {\Omega}(\omega^*)$}
\label{Omega}
\end{figure}

This fact is not straightforward. Reformulated in simpler terms, as follows, this is the main result of \cite{chapou2} (cf. Chapter \ref{chap:wellfounded}, {Theorem \ref{thm4}).

\begin{theorem}\label{thm4chap2}
An algebraic lattice $L$  is well-founded if and only if
$K(L)$ is well-founded  and contains no join-subsemilattice  isomorphic to $\underline \Omega(\omega^*)$ or to $[\omega]^{<\omega}$.\\
\end{theorem}

From Theorem \ref {thm4chap2}, we obtained:

 \begin{theorem} \label{thm:posetchap2}(Corollary \ref{poset}, \cite{chapou2})  A join-subsemilattice $P$ of $[\omega]^{<\omega}$ contains either  $[\omega]^{<\omega}$ as a join-semilattice or is well-quasi-ordered. In the latter case,   $J(P)$ is well-founded.
\end{theorem}

With this result, the proof of the "if" part of Theorem \ref{ordinal} is immediate.  Indeed, suppose  that $\alpha$ is not an ordinal.
Let $P\in\mathbb{J}_{\alpha}$. The lattice $J(P)$ contains a chain
isomorphic to $I(\alpha)$. Since $\alpha$ is not an ordinal,
$\omega^{*}\leq\alpha$. Hence, $J(P)$ is not well-founded. If $P$ is embeddable  in $[\omega]^{<\omega}$ as a
join-semilattice then, from
Theorem \ref{thm:posetchap2},  $P$ contains a join-subsemilattice
isomorphic to $[\omega]^{<\omega}$. Thus  $[\omega]^{<\omega}$ is
minimal in $\mathbb{J}_{\alpha}$.

A sierpinskisation $S$ of a countable order type $\alpha$ and $\omega$ is embeddable into $[\omega]^{<\omega}$ as a poset.  A consequence of Theorem \ref{thm:posetchap2} is the following
\begin{lemma}\label{lem:sierpnotinw} If  $S$  can be  embedded in $[\omega]^{<\omega}$ as a join-semilattice, $\alpha$ must be an ordinal.
\end{lemma}
\begin{proof} Otherwise,  $S$ contains an infinite antichain and by Theorem \ref{thm:posetchap2} it contains a copy of $[\omega]^{<\omega}$. But this poset cannot be embedded in a sierpinskisation. Indeed,  a sierpinskisation is embeddable  into a product of two chains, whereas $[\omega]^{<\omega}$ cannot be embedded in a product of finitely many chains (for every integer $n$, it contains the power set $\mathfrak P(\{0,\dots ,n-1\})$ which cannot be embedded into  a product of less than $n$ chains;  its dimension, in the sense of Dushnik-Miller's notion of dimension, is infinite, see \cite{trotter})\end{proof}.

 This fact invited us to restrict our  notion of sierpinskisation to some which are join-semilattices.  In fact, there are lattices of a particular kind.

 To a countable order type   $\alpha$, we associate a join-subsemilattice $\Omega_L(\alpha)$ of  the direct product $\omega\times\alpha$ obtained via a sierpinskisation of $\omega\alpha$ and $\omega$. We add a least element, if there is none, and we denote by   $\underline\Omega_L(\alpha)$  the resulting poset.  We also associate the join-semilattice
  $P_{\alpha}$ defined as follows:

If  $1+\alpha\not \leq \alpha$, in which case   $\alpha= n+\alpha'$ with $n<\omega$ and  $\alpha'$ without a first element, we set $P_{\alpha}:=n+\underline \Omega_L(\alpha')$. If not, and  $\alpha$ is equimorphic \index{equimorphic} to $\omega+ \alpha'$ we set $P_{\alpha}:= \underline \Omega_L(1+\alpha')$, otherwise, we set $P_{\alpha}= \underline \Omega_L(\alpha)$.

%
The importance of this kind of sierpinskisation  steems from the following
  result:
 \begin{theorem}\label{thm:serpinskalpha}

 If $\alpha$ is countably infinite, $P_{\alpha}$ belongs to every $\B$ coinitial in $
  \J_{\alpha}$.
 \end{theorem}

With Theorem \ref {thm:serpinskalpha} and Theorem \ref{ordinal}, Lemma \ref{lem:sierpnotinw}  yields :

$\bullet$ \emph{If $\alpha$ is not an ordinal $P_{\alpha}$ and  $[\omega]^{<\omega}$ are two incomparable minimal members of $\J_{\alpha}$}.

 According to Theorem \ref{thm:qalpha} and Theorem \ref{thm:serpinskalpha}:

  $\bullet$\emph {If $\alpha$ is a countably infinite ordinal, $Q_{\alpha}$ and $P_{\alpha}$ are minimal obstructions.}

In fact, using this new kind of sierpinskisation,  if $\alpha\leq \omega+\omega=\omega2$, $Q_{\alpha}$ and $P_{\alpha}$ coincide.

Indeed, if $\alpha=\omega+n$ with  $n<\omega$, then $S_{\alpha}=\Omega_L (1)\oplus  n$. In this case, $S_{\alpha}$ is isomorphic to $\omega\oplus n$, hence $Q_{\alpha}$ is isomorphic to the direct product $\omega\times (n+1)$ which in turn is isomorphic to $\Omega_L(n+1)=P_{\alpha}$. If $\alpha=\omega+\omega= \omega2$, $S_{\alpha}=\Omega_L (2)$. This poset is isomorphic to the direct product $\omega\times 2$. In this case,  $Q_{\alpha}$ is isomorphic to $[\omega]^2$, the subset of the product $\omega \times  \omega$ made of pairs $(i,j)$ with $i<j$. In turn  this poset is isomorphic to $\underline\Omega_L(\omega)=P_{\alpha}$.

Beyond $\omega2$,  no $Q_{\alpha}$ is a sierpinskisation.

For that, we apply the following  refinement of Lemma \ref{lem:sierpnotinw} obtained in
 \cite {chapou2} (see Example \ref{ex:ordinal},  Chapter \ref{chap:wellfounded}).
\begin{proposition}\label{w-f}
Let  $\gamma$ be a countable order type.  Then  $\underline\Omega_L(\gamma)$  is embeddable  in $[\omega]^{<\omega}$ as a join-semilattice  if  and only if  $\gamma\leq \omega$.
\end{proposition}

A consequence of Proposition \ref{w-f} is the following:\begin{corollary} \label {incomp} If  $\alpha$ is a  countable chain, then  $P_{\alpha}$ and  $[\omega]^{<\omega}$
are two incomparable members of $\J_{\alpha}$ if and only if
$\alpha$ is not embeddable in  $\omega2$
\end{corollary}

\begin{proof}
Let  $\alpha \leq \omega2$. As  we have seen $P_{\alpha}$ is isomorphic to $\omega\times (1+\alpha')$ if
$\alpha'<\omega$  and  to $[\omega]^{2}$ if $\alpha=\omega2$. In both cases $P_{\alpha}$ is embeddable , as a join-semilattice,
in $[\omega]^{<\omega}$. Conversely, if $P_{\alpha}$ and
$[\omega]^{<\omega}$ are comparable, as join-semilattices, then, necessarily $P_{\alpha}$ is embeddable  into
$[\omega]^{<\omega}$ as a join-semilattice. From Proposition
\ref{w-f}, it follows that $\alpha \leq \omega2$.
\end{proof}

This work leaves open the following questions. We are only able to give some examples of ordinals for which the answer to the first question is positive.

\begin{questions}
\begin{enumerate}
\item If  $\alpha$ is a countably infinite ordinal, does the minimal obstructions are $\alpha$, $P_{\alpha}$, $Q_\alpha$ and some  lexicographical sums of
obstructions corresponding to smaller  ordinal?
\item  If $\alpha$ is a scattered order type which is not an ordinal, does the minimal  obstructions are
$\alpha$, $P_{\alpha}$, $[\omega]^{<\omega}$ and some  lexicographical sums of
obstructions corresponding to smaller  scattered order types?
\item If $\alpha$ is the order type $\eta$ of the chain of rational numbers, does $\J_{\neg \eta}= Forb_{\J}( \{1+\eta, [\omega]^{<\omega}, \underline \Omega(\eta)\})$ where $\underline \Omega(\eta)$ is the lattice serpinskisation represented Figure \ref {Omegachap2}?
\end{enumerate}
\end{questions}

\begin{figure}[htbp]
\centering
\includegraphics[width=3in]{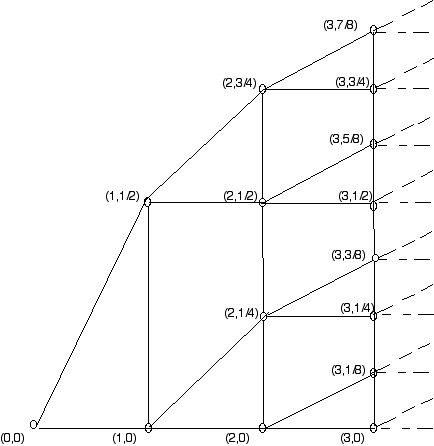}
\caption{$ \underline \Omega(\eta)$}
\label{Omegachap2}
\end{figure}







\section[Join-semilattices  and obstructions]{Join-semilattices and a proof of Theorem \ref{thmfirst}}
A \emph{join-semilattice} is a poset $P$ such that every two elements $x,y$ have a least upper-bound, or join, denoted by $x\vee y$. If $P$ has a least element, that we denote $0$, this amounts to say that every finite subset of $P$ has a join. In the sequel, we will mostly consider join-semilattices with a least element. Let $Q$ and $P$ be such join-semilattices. A map $f:Q\rightarrow P$ is \emph{join-preserving} if:
\begin{equation}\label{eq:def:joinpreserving}
f(x\vee y)=f(x)\vee f(y)
\end{equation}
for all $x,y\in Q$.

This map \emph{preserves finite (resp. arbitrary) joins} if
\begin{equation}\label{eq:def:finitepreserving}
f(\bigvee X)=\bigvee \{f(x): x\in X\}
\end{equation}
for every finite (resp. arbitrary) subset $X$ of $Q$.

If $P$ is a join-semilattice with a least element, the set $J(P)$ of ideals of $P$ ordered by inclusion is a complete lattice. If $A$ is a subset of  $P$, there is a least ideal  containing $A$, that we denote $<A>$. An ideal $I$ is \emph{generated} by a subset $A$ of $P$ if $I=<A>$.  If $[A]^{<\omega}$ denotes the collection of finite subsets of $A$ we have:
\begin{equation}\label{eq:def:generated}
<A>= \downarrow \{ \bigvee X: X\in [A]^{<\omega}\}
\end{equation}

\begin{lemma} \label{lem:arbitraryjoins}Let $Q$ be a join-semilattice with a least element and  $L$ be a complete lattice.  To a map  $g:Q\rightarrow L$ associate $\overline g: \mathfrak P(Q)\rightarrow L$ defined by setting  $\overline g(X):= \bigvee \{g(x): x\in X \}$ for every   $X \subseteq Q$. Then $\overline g$ induces a map from $J(Q)$ in $L$ which preserves arbitrary joins whenever $g$ preserves finite joins.  \end{lemma}
\begin{proof}

\begin{claim}\label{claim: supgenerates}Let $I\in J(Q)$. If $g$ preserves finite joins and $A$ generates $I$ then $\overline g(I) =\bigvee\{g(x): x\in A\}$.
\end{claim}
\noindent{\bf Proof of {claim} \ref{claim: supgenerates}.} Since $I=<A>$ and $g$ preserves finite joins, $<\{g(x): x\in I\}>=<\{g(x): x\in A\}>$. The claimed equality follows.
 \endproof

Now, let $\mathcal I\subseteq J(Q)$ and $I:=\bigvee \mathcal I$. Clearly, $A:= \bigcup \mathcal I$ generates $I$.   Claim \ref{claim: supgenerates}  yields  $\overline g(I) =\bigvee\{g(x): x\in A\}= \bigvee \bigcup \{\{g(x): x\in J\}: J\in \mathcal I\}=\bigvee \{\bigvee \{g(x): x\in J\}: J\in \mathcal I\}= \bigvee \{\overline g(J): J\in \mathcal I\}$. This proves that $\overline g$ preserves arbitrary joins. \end{proof}
\begin{remark} Note that in Lemma \ref{lem:arbitraryjoins} the fact that $g$ is one-to-one  does not necessarily transfer to the map induced by $\overline g $ on $J(Q)$.
For an example, let $\kappa:=2^{\aleph_{0}}$ and   let $Q:= [\kappa] ^{<\omega}$ and $L:= \mathfrak P(\aleph_{0})$ ordered by inclusion. The lattice  $J(Q)$ is isomorphic to $\mathfrak P(\kappa)$, thus $\vert Q\vert= 2^{\kappa}> \kappa= \vert L\vert$, proving that $J(Q)$ is not embeddable  in $L$. On an other hand,  as it is well known,  $\mathfrak P(\aleph_{0})$ contains a subset  $A$ of size $2^{\aleph _{0}}$ made of infinite sets which are pairwise almost disjoint. As it is easy to check, the set of finite unions of members of $A$ is a join-subsemilattice of $L$ isomorphic to $Q$.
\end{remark}

\begin{lemma}\label{lem:11} Let $R$ be a poset and $P$ be a join-semilattice with a least element.
The following properties are equivalent:
\begin{enumerate} [{(i)}]
\item \label{item:0}There is an
embedding from $I(R)$ in $J(P)$ which preserves arbitrary joins.

\item \label{item:1}There is an
embedding from $I(R)$ in $J(P)$.

\item \label{item:2}There is a map $g$
from $I_{<\omega}(R)$ in $P$ such that
\begin{equation}  \label{eq:notsup}X\nsubseteq
Y_{1}\cup\ldots\cup Y_{n}   \Rightarrow g(X)\not\leq
g(Y_{1})\vee\ldots\vee g(Y_{n})
\end{equation} for all $X, Y_{1}, \ldots,
Y_{n}\in I_{<\omega}(R)$.

\item\label{item:3} There is a map $h:R \rightarrow
P$ such that \begin{equation} \label{eq:notleq}\forall i (1\leq i \leq
n\Rightarrow x\not\leq y_{i}) \Rightarrow h(x)\not\leq h(y_{1})\vee\ldots\vee h(y_{n})\end{equation} for all
$x, y_{1}, \ldots, y_{n}\in R$.
\end{enumerate}\end{lemma}

\begin{proof}\noindent $(\ref{item:0})\Rightarrow (\ref{item:1})$. Obvious.

 \noindent $(\ref{item:1})\Rightarrow (\ref{item:2})$  Let $f$ be an embedding from
$I(R)$ in $J(P)$. Let  $X \in I_{<\omega}(R)$. The set  $A:= Max(X)$ of maximal elements of $X$ is finite and $X:=\downarrow A$. Set $F(X):= \{f(R\setminus \uparrow a): a \in Max (X)\}$ and  $C(X):= f(X)\setminus \bigcup F(X)$.

\begin {claim} \label{claim:notempty}$C(X)\not =\emptyset$ for every $X \in I_{<\omega}(R)$\end{claim}
\noindent{\bf Proof of Claim \ref{claim:notempty}.}
 If $X= \emptyset$, $F(X)= \emptyset$. Thus  $C(X)=f(X)$ and our assertion is proved. We may then assume  $X\not = \emptyset$.  Suppose  $C(X)=\emptyset$, that is $f(X)\subseteq \bigcup F(X)$. Since $f(X)$ is an ideal of $P$ and  $\bigcup F(X)$ is a finite union of initial segments  of $P$, this implies that $f(X)$ is included in some, that is $f(X) \subseteq f(R\setminus \uparrow a)$ for some $a \in Max(X)$. Since $f$ is an embedding, this implies $X\subseteq  R\setminus \uparrow a$ hence $a\not \in X$. A contradiction.
\endproof

 Claim \ref{claim:notempty} allows us to pick an element $g(X)\in
C(X)$ for each $X \in I_{<\omega}(R)$. Let $g$ be the map defined by this process. We show that implication (\ref{eq:notsup}) holds. Let $X, Y_{1},\ldots,
Y_{n}\in I_{<\omega}(R)$. We have $g( Y_{i})\in f( Y_{i})$ for
every $1\leq i\leq n$. Since $f$ is an embedding
$f(Y_{i})\subseteq f( Y_{1}\cup \ldots \cup Y_{n})$ for every
$1\leq i\leq n$. But $f( Y_{1}\cup \ldots \cup Y_{n})$ is an
ideal. Hence $g(Y_{1})\vee\ldots\vee g(Y_{n})\in f( Y_{1} \cup
\ldots \cup Y_{n})$. Suppose $X\nsubseteq Y_{1}\cup \ldots \cup
Y_{n}$. There is $a\in Max(X)$ such that $Y_{1}\cup \ldots \cup
Y_{n}\subseteq R\setminus \uparrow a$.  And since $f$ is an embedding, $f(Y_{1}\cup \ldots \cup
Y_{n})\subseteq f(R\setminus \uparrow a)$. If $g(X)\leq g(Y_{1})\vee\ldots\vee g(Y_{n})$ then
$g(X)\in f( Y_{1}\cup \ldots\cup Y_{n})$. Hence $g(X)\in f(R\setminus \uparrow a)$,  contradicting
$g(X)\in C(X)$.

\noindent$(\ref{item:2})\Rightarrow (\ref{item:3})$.  Let $g: I_{<\omega}(R) \rightarrow P$ such that  implication (\ref{eq:notsup}) holds. Let  $h$ be the map induced by $g$ on  $R$ by setting $h(x):=g(\downarrow x)$ for $x\in R$. Let $x, y_{1}, \ldots,
y_{n}\in R$. If $x\not\leq y_{i}$ for every $1\leq i \leq n$, then
$\downarrow x \nsubseteq (\downarrow y_{1} \cup \ldots \cup
\downarrow y_{n})$. Since $g$ satisfies implication (\ref{eq:notsup}),
we have $h(x):=g(\downarrow x)\not\leq g(\downarrow y_{1})\vee
\ldots \vee g(\downarrow y_{n})=h(y_{1})\vee\ldots\vee
h(y_{n})$. Hence  implication (\ref{eq:notleq}) holds.

\noindent$(\ref{item:3})\Rightarrow (\ref{item:0})$ Let $h: R\rightarrow P$
such that implication (\ref{eq:notleq}) holds. Define $f: I(R)\rightarrow
J(P)$ by setting $f(I):=<\{h(x): x\in I\}>$, the ideal generated by
$\{h(x): x\in I\}$, for $I\in I(R)$. Since in $I(R)$ the join is the union,    $f$ preserves arbitrary joins.  We claim that $f$ is one-to-one.
Let $I, J\in I(R)$ such that $I\nsubseteq J$. Let $x\in
I\setminus  J$. Clearly $h(x)\in f(I)$. We claim that
$h(x)\not\in f(J)$. Indeed, if $h(x)\in f(J)$, then
$h(x)\leq\bigvee\{h(y): y\in F\}$ for some finite subset $F$ of
$J$. Since implication (\ref{eq:notleq}) holds, we have $x\leq
y$ for some $y\in F$. Hence $x\in J$, contradiction. Consequently
$f(I)\nsubseteq f(J)$. Thus  $f$ is one-to-one as claimed.
\end{proof}

The following proposition rassembles the main properties of the comparizon of join-semilattices.
\begin{proposition} \label{1.3} Let  $P$, $Q$ be two join-semilattices with a least element. Then:
\begin{enumerate}
\item \label{item 1lem1.3}$Q$ is embeddable in $P$ by a join-preserving map  iff $Q$ is embeddable in  $P$ by a map preserving finite joins.
\item \label{item 2lem1.3}If $Q$ is embeddable  in  $P$ by a join-preserving map then $J (Q)$ is embeddable  in  $J (P )$
by a map preserving arbitrary joins.

Suppose $Q := I_{<\omega} (R)$ for some poset $R$. Then:
\item  \label{item 3lem1.3}$Q$ is embeddable in $P$ as a poset iff $Q$ is embeddable in   $P$ by a map preserving finite joins.
\item \label{item 4lem1.3}$J (Q)$ is embeddable in $J (P )$ as a poset iff $J (Q)$  is embeddable in   $J(P)$ by a map preserving arbitrary  joins.
\item  \label{item 5lem1.3}If $\downarrow  x$ is finite for every $x \in  R$ then $Q$ is embeddable  in $P$ as a poset iff $J (Q)$
is embeddable in  $J (P )$ as a poset.
\end{enumerate}
\end{proposition}

\begin{proof}
\begin{enumerate}[{}]
\item (\ref{item 1lem1.3}) Let $f : Q \rightarrow  P$ satisfying $f (x \vee y) = f (x)
\vee f (y)$ for all $x, y \in  Q$. Set $g(x) := f (x)$  if
$x\not=0$ and $g(0) := 0$. Then g preserves finite joins.

\item (\ref{item 2lem1.3}) Let $f : Q \rightarrow  P$ and $\overline f : J (Q)
\rightarrow J (P )$ defined by $f (I ) :=\downarrow \{f (x) : x
\in I \}$. If $f$ preserves finite joins, then $\overline f$
preserves arbitrary joins. Furthermore, $\overline f$ is one to one provided that $f$ is one-to-one.

\item (\ref{item 3lem1.3}) Let $ f : Q \rightarrow P$ . Taking account that $Q :=
I_{<\omega} (R)$, set $g(\emptyset) := 0$ and $g(I ):= \bigvee \{f
(\downarrow x) : x \in I \}$ for each $I \in   I_{<\omega}
(R)\setminus \{\emptyset\}$. Since in $Q$ the join is the union, the map $g$ preserves finite unions.

\item (\ref{item 4lem1.3}) This is equivalence $(\ref{item:0})\Longleftrightarrow (\ref{item:3})$ of Lemma \ref{lem:11}.
\item (\ref{item 5lem1.3}) If $Q$ is embeddable in  $P$ as a poset, then from  Item (\ref{item 3lem1.3}), $Q$ is embeddable in
 $P$ by a map preserving finite joins. Hence from Item $(\ref{item 1lem1.3})$, $J(Q)$
is embeddable  in $J(P)$ by a map preserving arbitrary joins. Conversely, suppose that $J(Q)$ is embeddable in $J(P)$ as
a poset. Since $J(Q)$ is isomorphic to  $I(R)$, Lemma \ref{lem:11} implies that
there is map $h$ from $R$ in $P$ such that implication
$(\ref{eq:notleq})$ holds.  According to the proof of Lemma \ref{lem:11}, the map $f: I(R) \rightarrow J(P)$ defined by setting
$f(I):=\bigvee\{h(x): x\in I\}$ is an embedding preserving arbitrary joins. Since $\downarrow x$ is finite for every $x\in R$,  $I$ is  finite , hence $f(I)$ has a largest element, for   every $I\in I_{<\omega}(R)$. Thus $f$ induces an embedding from $Q$ in $P$ preserving finite joins. \end{enumerate}
\end{proof}

\begin{theorem} \label{1.4} Let $R$ be a poset, $Q := I_{<\omega} (R)$ and $\kappa:=\vert Q\vert$. Then, there is a set $\B$, of size
at most $2^\kappa$, made of join-semilattices, such that for every join-semilattice
$P$, the join-semilattice
$J (Q)$ is not embeddable  in $J (P )$ by a map preserving arbitrary joins if and only
if no member $Q$ of $\B$ is embeddable in $P$ as a join-semilattice.
\end{theorem}

\begin{proof}
If   $\downarrow x$ is finite for every $x\in R$ the conclusion of the theorem holds with $\mathbb B= \{ Q\}$ (apply Item (\ref{item 3lem1.3}), (\ref{item 4lem1.3}), (\ref {item 5lem1.3}) of Proposition \ref{1.3}). So we may assume that $R$ is infinite. Let $P$ be a join-semilattice. Suppose that there is an embedding $f$ from $J (Q)$ in $J (P )$ which  preserves arbitrary
joins.
\begin{claim}\label{claim:counting}There is a join-semilattice $Q_f$ such that
\begin{enumerate}
\item $Q_f$
embeds in $P$ as a join-semilattice.
\item $J (Q)$
embeds in $J (Q_f )$ by a map preserving arbitrary joins.
\item  $\vert Q_f \vert =
\vert Q\vert$.
\end{enumerate}
\end{claim}
\noindent{\bf Proof of Claim \ref{claim:counting}.}
 From Lemma \ref{lem:11}, there is a map $g : Q
\rightarrow P$  such that inequality   $(\ref{eq:notsup})$ holds. Let $Q_f $
be the join-semilattice of $P$ generated by $\{g(x) : x \in Q\}$. This inequality holds when $P$ is replaced by $Q_f$. Thus from Proposition  \ref{1.3}, $J (Q)$ embeds into $J (Q_f )$ by a map
preserving arbitrary joins. Since $R$ is infinite,  $\vert Q_f \vert =
\vert Q\vert=\vert R\vert$.
\endproof

For each join-semilattice $P$ and  each embedding $f:J (Q)\rightarrow J (P )$ select $Q_f$, given by Claim \ref{claim:counting}, on a fixed set of size $\kappa$.  Let $\B$ be the collection of this join-semilattices. Since the number of join-semilattices on a set of size $\kappa$ is at most $2^{\kappa}$, $\vert\mathbb B\vert \leq 2^\kappa$.
\end{proof}

\noindent{\bf Proof of Theorem \ref{thmfirst}.}
Let $\alpha$ be the order type of a chain $C$. Apply Theorem above with $R:=C$. Since,  in
 this case,  $J(Q)$ is embeddable  in $J(P)$ if and only if $J(Q)$ is embeddable in $J(P)$ by a map preserving arbitrary joins, Theorem  \ref{thmfirst}  follows. \endproof

\section [Monotonic sierpinskisations]{Monotonic sierpinskisations and a proof of Theorem  \ref{thm:qalpha}}\label{section:minimal}
In this section, we use the notion of  sierpinskisation studied  in \cite{pz} that we have recalled in the introduction under the name of monotonic sierpinskisation.

We start with some basic properties of ordinary sierpinskisations. For that, let $\alpha$ be a countably infinite order type. Let $S$ be a sierpinskisation of $\alpha$ and $\omega$. We assume that $S=(\N, \leq)$ and the order on $\N$ is the intersection of a linear order $\leq_{\alpha}$ on $\N$ with the natural order on $\N$, such that $L:=(\N, \leq_{\alpha})$ has order type $\alpha$.

\begin{lemma}\label{lem:basicsierp1} Let $A$ be a non-empty subset of $\N$. Then, the following properties are equivalent:
\begin{enumerate}[{(i)}]
\item \label{item1:lem:basicsierp1} No element of $A$ is maximal w.r.t. $S$.
\item \label{item2:lem:basicsierp1} No element of $A$ is maximal w.r.t. $L$.

\item \label{item3:lem:basicsierp1} $A$ is up-directed w.r.t. $S$ and is infinite.

\end{enumerate}

Furthermore, when one of these conditions holds,  $\downarrow_{L} A= \downarrow_{S} A$.
\end{lemma}
\begin{proof} $(\ref {item1:lem:basicsierp1})\Rightarrow
(\ref {item2:lem:basicsierp1})$.  Observe that, since $\leq_{\alpha}$ is a linear extension of $\leq$, an element $a\in A$ which is maximal w.r.t. $L$ is maximal w.r.t. $S$.

\noindent $(\ref {item2:lem:basicsierp1})\Rightarrow
(\ref {item3:lem:basicsierp1})$  Since $A$ is non-empty, the inexistence of a maximal element implies that $A$ is infinite.  Let us check that $A$  is up-directed. Let $x, y\in A$. We may suppose $x\leq_{\alpha}y$. Since $y$ is not maximal in $A$ w.r.t. $L$, $\uparrow_{L}y \cap A$ is infinite. Hence there is some $z\in \uparrow_{L}y \cap A$ such that $x, y\leq_{\omega} z$. Clearly, $x,y\leq z$, proving that $A$ is up-directed.

%

\noindent $(\ref {item3:lem:basicsierp1})\Rightarrow
(\ref {item1:lem:basicsierp1})$ Let $a\in A$. If $a$ is maximal w.r.t. $S$  then, since $A$ is up-directed, $a$ is the largest element of $A$. From this fact $A\subseteq \downarrow_{S} a$. Since the natural order on $\N$ is a linear extension of $\leq$, $A\subseteq \downarrow_{\omega} a$. This latter set being  finite, $A$ is finite, a contradiction.

Let us prove the second assertion. Since $\leq_{\alpha}$ is a linear extension of $\leq$, we have $\downarrow_{S} A \subseteq \downarrow_{L} A$. Conversely, let $x\in \downarrow_{L} A$. Let $y\in A$ such that $x\leq _{L} y$.  Since $A$ satisfies $(\ref {item2:lem:basicsierp1})$,  we proceed as in the proof of $(\ref {item2:lem:basicsierp1})\Rightarrow
(\ref {item3:lem:basicsierp1})$. From the fact that  $y$ is not maximal in $A$ w.r.t. $L$, $\uparrow_{L}y \cap A$ is infinite. Hence there is some $z\in \uparrow_{L}y \cap A$ such that $x, y\leq_{\omega} z$, proving that $x\in \downarrow_{S}A$.
\end{proof}

\begin{proposition} \label{prop:basicsierp2}Le $I$ be a  subset of $\N$. Then, the following properties are equivalent:
\begin{enumerate}[{(i)}]
\item \label{item1:lem:basicsierp2} $I$ is a non-principal ideal of $S$.
\item \label{item2:lem:basicsierp2} $I$ is an  non-empty non-principal  initial segment of $L$.
\end{enumerate}

Furthermore, when one of these conditions holds,  then for every subset $A$ of $I$:
$$I=\downarrow_{L} A \; \text{ if and only if}\;   I=\downarrow_{S} A$$.
\end{proposition}

\begin{proof}$(\ref{item1:lem:basicsierp2}) \Rightarrow (\ref{item2:lem:basicsierp2})$ Assume that $I$ is a non-principal ideal of $S$. According to Lemma  \ref{lem:basicsierp1}, no element of $I$ is maximal w.r.t.  $L$. Moreover $\downarrow_{L} I= \downarrow_{S} I= I$. Hence, $I$ is a non-principal initial segment of $L$.

$(\ref{item2:lem:basicsierp2}) \Rightarrow (\ref{item1:lem:basicsierp2})$ Assume that $I$ is   an  non-empty non-principal  initial segment of $L$. Lemma \ref{lem:basicsierp1} yields that $I$ is up-directed w.r.t. $S$ and infinite. Since $\leq_L$ is a linear  extension of $\leq$, $I$ is an initial segment of $S$, hence $I$ is an ideal of $S$.  From Lemma \ref{lem:basicsierp1} again, it is not principal.

For the second assertion, let $A$ be a subset of $I$. Note that if
$\downarrow _{L} A=I$, resp. $\downarrow _{S} A=I$,  then no element
of $A$ is maximal w.r.t. $L$, resp. w.r.t. $S$. Apply Lemma
\ref{lem:basicsierp1}.
\end{proof}

\begin{theorem} Let $\alpha$ be a countably infinite order type. If $S$ is  a sierpinskisation of a chain of type $\alpha$ and a chain of type $\omega$, the set $J^{\neg \downarrow}(S)$ of non-principal ideals of $S$ forms a chain and this chain has the same order type as the subset  of  $I(\alpha)$ made of non-principal initial segments of $\alpha$. If $\alpha= \omega\alpha'$, every chain $C\subseteq J(S)$ extends to a chain whose order type is either  $J(\alpha')$ or $\omega+I(\alpha'')$, where $\alpha''$ is a proper final segment of $\alpha'$.  \end{theorem}

\begin{proof}The first sentence is an  immediate consequence of the equivalence between $(\ref{item1:lem:basicsierp2})$ and  $\ref{item2:lem:basicsierp2})$ of Proposition \ref{prop:basicsierp2}.

Concerning the second sentence, note that $J^{\neg \downarrow}(\omega\alpha')$ is isomorphic to $J(\alpha')$, hence $J(S)$ contains a chain of this type. Let $C\subseteq J(S)$, $C':= \{I\in C: I= \downarrow_S x \; \text{for some}\;  x\in S\}$ and $C'':= \{I\in J^{\neg \downarrow}(S): \bigcup C'\subseteq I\}$. If $C'$ is empty, $C''$ contains $C$ and has  order type $J(\alpha')$. So in order to complete the proof of the lemma, we may assume that $C'$ is non-empty.

\begin{claim}\label{claim:sierpnonemptyideal}The set $I_{0}:=\bigcap C''$ is a non-empty and non-principal ideal of $S$.
\end{claim}
 \noindent {\bf Proof of Claim \ref {claim:sierpnonemptyideal}.}
Clearly $\bigcup C'\subseteq I_{0}$, hence $I_{0}$ is non-empty. To
see that $I_{0}$ is a non-principal ideal,  we introduce  some
notations. We suppose that $S=(\N, \leq)$ where $\leq$ is the
intersection of the natural order on $\N$ with a linear order
$\leq_{\omega\alpha'}$ such that $L:=(\N, \leq_{\omega\alpha'})$ has
order type $\omega\alpha'$, this linear order  be given by  a
bijection $\varphi$ between $\N$ and $\N\times \A'$.  According to
Proposition \ref{prop:basicsierp2}, each member of $C''$ is a
non-empty non-principal initial segment of $L$, thus $I_{0}$ is an
initial segment of $L$. We claim that $I_{0}$ is a non-principal
initial segment of $L$. Suppose for a contradiction that
$I_{0}=\downarrow_{L}x_{0}$. Since the order type of $L$ is
$\omega\alpha'$, every $x\in \N$ has a successor $x'$ w.r.t. $L$.
Let $x'_{0}$ be the successor of $x_{0}$. Then $x'_{0}\not \in I$
for some $I\in C''$. But since $I_{0}\subseteq I$, we have
$I_{0}=I$. Thus $I$ is a principal initial segment of $L$. With
Proposition \ref{prop:basicsierp2}, this contradicts the definition
of $C''$.  Now, since $I_{0}$ is a non-empty non-principal initial
segment of $L$, it follows from Proposition \ref{prop:basicsierp2}
that $I_{0}$ is a  non-principal ideal of $S$.\endproof

\begin{claim}\label{claim:sierplengthchain}
The chain $C'$ extends to a chain $C'_1\subseteq J(I_0)$ which  has order type $\omega$.
The chain $C''$ has order  type $I(\alpha'')$, where $\alpha''$ is a proper final segment of $\alpha'$.
\end{claim}
\noindent {\bf Proof of Claim \ref {claim:sierplengthchain}.}
Let $\nu$ be the order type of $C'$. Since each principal initial segment of $S$ is finite, $\nu\leq\omega$.
If $\nu=\omega$, set $C'_1:=C'$. Otherwise, let $x_0$ be the largest  element of $C'$. Then $x_0\in I_0$. Apply Claim \ref{claim:sierpnonemptyideal}. Since $I_0$ is a non-principal ideal of $S$, $\uparrow_S x_0\cap I_0$  contains a chain $D$ of type $\omega$. Set $C'_1:= C'\cup D$.  Again, since $I_{0}$ is a non-principal segment of $L$, $F:= \N\setminus I_{0}$ is a final segment of $(\N, \leq_{\omega\alpha'})$ whose image under $\varphi$ is of the form $\N\times A''$ where $A''$ is a proper final segment of $A'$. In this case $C''$ is isomorphic to $I(A'')$, hence its order type is $I(\alpha'')$ where $\alpha''$ is the order type of $A''$. \endproof

The conclusion of the lemma follows readily from Claim \ref {claim:sierplengthchain}. Indeed,
the set  $C'_1 \cup C''$ is a chain containing $C$. According to  Claim \ref {claim:sierplengthchain}, its order type is $\omega+I(\alpha'')$. \end{proof}

 \begin{corollary}\label{cor:ordertypesierp} Let  $S$ be a sierpinskisation of $\omega \alpha'$ and $\omega$  then the order types of chains which are embeddable in $J(S)$ depends only upon $\alpha'$. Moreover, if  $\alpha'$ is equimorphic to  a chain of order type $1+\alpha''$, $\omega+I(\alpha'')$ is the largest order type of the chains of ideals of $S$.
 \end{corollary}
The well-known fact that $I(\alpha)$ is embeddable in $I(\beta)$ if and only $\alpha$ is embeddable in $\beta$ yields easily the
 following refinement:
  \begin{lemma} \label{lem:somewhatcrucial}Let  $S$, resp.   $T$ , be a  sierpinskisation of $\omega\alpha$ , resp. $\omega\beta$,  and $\omega$.  Then every chain which is embeddable in $J(S)$ is embeddable in $J(T)$ if and only if $\alpha$ is embeddable in $\beta$.
  \end{lemma}
  We recall that  for a countable order type $\alpha'$, two monotonic sierpinskisations of $\omega\alpha'$ and $\omega$ are embeddable in each other and denoted by the same symbol $\Omega(\alpha')$ and we recall the  following result (cf. \cite{pz} Proposition 3.4.6. pp. 168)

\begin{lemma}\label{lem: proppouzag}
Let $\alpha'$ be a countable order type. Then $\Omega(\alpha')$ is embeddable in every sierpinskisation $S'$ of $\omega\alpha'$ and $\omega$.
\end{lemma}

\begin{lemma}\label{lem:trivial}
Let $\alpha$ be a countably infinite order type and $S$ be a sierpinskisation of $\alpha$ and $\omega$.
Assume that $\alpha= \omega\alpha'+n$  where $n<\omega$. Then there is a subset of $S$ which is the direct sum $S'\oplus F$ of a sierpinskisation $S'$ of  $\omega\alpha'$ and $\omega$ with an $n$-element poset $F$.
\end{lemma}
\begin{proof} Assume that $S$ is given by a bijective map  $\varphi$  from $\N$ onto a chain $C$ having order type $\alpha$. Let $A'$ be the set of the $n$ last elements of $C$,  $A:=\varphi^{-1}(A')$  and   $a$ be the largest element of $A$ in $\N$. The image of $]a \rightarrow )$ has order type $\omega\alpha'$, thus $S$ induces on $]a \rightarrow )$ a sierpinskisation $S'$ of $\omega\alpha'$ and $\omega$.  Let $F$ be the poset induced by $S$ on $A$. Since every  element of $S'$ is incomparable to every element of $F$
 these  two posets form a direct sum. \end{proof}
\begin{lemma}\label{lem:qalphaprelim}
Let $\alpha$ be a countably infinite order type and $S$ be a sierpinskisation of $\alpha$ and $\omega$.
If  $\alpha= \omega\alpha'+n$  where $n<\omega$ then $Q_{\alpha}:=
I_{<\omega}( \Omega(\alpha')\oplus n)$ is embeddable in $I_{<\omega}(S)$ by a map preserving finite joins.
\end{lemma}
\begin{proof}

{\bf Case 1.} $n=0$. By Lemma \ref{lem: proppouzag}  $
\Omega(\alpha')$ is embeddable in $S$. Thus $Q_{\alpha}$ is
embeddable in $I_{<\omega}(S)$ by a map preserving finite joins.

{\bf Case 2.} $n\not = 0$. Apply Lemma \ref{lem:trivial}. According to Case 1, $I_{<\omega}( \Omega(\alpha'))$ is embeddable in $I_{<\omega}(S')$. On an other hand $n+1$ is embeddable in $I_{<\omega}(F)=I(F)$.
Thus  $Q_{\alpha}$ which is isomorphic to the product $I_{<\omega}( \Omega(\alpha'))\times (n+1)$ is embeddable in  the product  $I_{<\omega}(S')\times I_{<\omega}(F)$. This product is   itself isomorphic to $I_{<\omega}(S'\oplus F)$. Since $S'\oplus F$ is embeddable in $S$, $I_{<\omega}(S'\oplus F)$ is embeddable in $I_{<\omega}(S)$ by a map preserving finite joins.  It follows that $Q_{\alpha}$ is embeddable in $I_{<\omega}(S)$ by a map preserving finite joins.
\end{proof}


%

%
\begin{theorem} \label{thm:qalphaprelim}Let $\alpha$ be a countable ordinal and $P\in \J_{\alpha}$. If $P$ is embeddable in $[\omega]^{<\omega}$ by a map preserving finite joins there is sierpinskisation $S$ of $\alpha$ and $\omega$ such that $I_{<\omega}(S)\in \J_{\alpha}$ and $I_{<\omega}(S)$ is embeddable in $P$ by a map preserving finite joins.
\end{theorem}
\begin{proof}
We construct first $R$  such that $I_{<\omega}(R)\in \J_{\alpha}$ and $I_{<\omega}(R)$ is embeddable in $P$ by a map preserving finite joins.

We may suppose that $P$ is a subset of $[\omega]^{<\omega}$ closed under finite unions. Thus   $J(P)$ identifies with  the set of arbitrary unions of members of $P$.
Let $(I_{\beta})_{\beta<\alpha+1}$ be  a strictly increasing sequence of ideals of $P$. For each $\beta<\alpha$ pick $x_{\beta}\in I_{\beta+1}\setminus I_{\beta}$ and $F_{\beta}\in P$ such that $x_\beta \in F_{\beta}\subseteq  I_{\beta+1}$.  Set $X:= \{x_{\beta}: \beta<\alpha\}$, $\rho:=\{(x_{\beta'},x_{\beta''}): \beta'<\beta''<\alpha\;  \text{and}\;  x_{\beta'}\in F_{\beta''}\}$. Let  $\hat \rho$ be the reflexive transitive closure of $\rho$. Since $\theta:= \{(x_{\beta'},x_{\beta''}): \beta'<\beta''<\alpha\}$ is a linear order containing $\rho$, $\hat \rho$ is an order on $X$. Let   $R:=(X, \hat \rho)$ be the resulting poset.
\begin{claim}\label{claim:jalpha} $I_{<\omega}(R)\in \J_{\alpha}$.
\end{claim}
\noindent { \bf Proof of claim \ref{claim:jalpha}.} The linear order
$\theta$ extends the order $\hat\rho$ and has type $\alpha$, thus
$I(R)$ has a maximal chain of type $I(\alpha)$. Since
$J(I_{<\omega}(R))$ is isomorphic to $I(R)$, $I_{<\omega}(R)$
belongs to $\J_{\alpha}$ as claimed.\endproof

\begin{claim}\label{claim:finite} For each $x\in X$, the initial segment $\downarrow x$ in  $R$ is finite.
\end{claim}
\noindent { \bf Proof of claim \ref{claim:finite}.} Suppose not. Let $\beta$ be minimum such that
for $x:=x_{\beta}$, $\downarrow x$ is infinite.
For each $y\in X$ with $y<x$ in $R$ select a finite sequence $(z_i(y))_{i\leq n_y}$ such that:
\begin{enumerate}
\item $z_{0}(y)=x$ and $z_{n_y}=y$.
\item \label{item:finite2}$(z_{i+1(y)}, z_{i}(y))\in \rho$ for all $i<n_{y}$.
\end{enumerate}
According to item \ref{item:finite2}, $z_{1}(y)\in F_{\beta}$. Since $F_\beta$ is finite, it contains some $x':=x_{\beta'}$ such that $z_{1}(y)= x'$ for infinitely many $y$. These elements belong to $\downarrow x'$. The fact that  $\beta'<\beta$  contradicts the choice of $x$.
\endproof

\begin{claim} \label{claim:union} Let $\phi$ be defined by  $\phi(I):= \bigcup \{F_{\beta}: x_\beta\in I\}$
for each $I\subseteq X$. Then:
\item $\phi$ induces an embedding  of $I(R)$ in $ J(P)$ and an embedding of $I_{<\omega}(R)$ in $P$.
\end{claim}
\noindent { \bf Proof of claim \ref{claim:union}.} We prove the
first part of the claim. Clearly, $\phi (I)\in J(P)$ for each
$I\subseteq X$. And trivially,  $\phi$ preserves arbitrary unions.
In particular, $\phi$ is order preserving. Its remains to show that
$\phi$ is one-to-one. For that, let $I,J\in I(R)$ such that
$\phi(I)=\phi(J)$. Suppose $J\not \subseteq I$.  Let $x_{\beta}\in
J\setminus I$, Since $x_\beta \in J$, $x_\beta\in F_{\beta}\subseteq
\phi(J)$. Since  $\phi(J)= \phi(I)$, $x_{\beta}\in \phi(I)$. Hence
$x_\beta \in F_{\beta'}$ for some $\beta'\in I$. If $\beta'<\beta$
then since $F_{\beta'}\subseteq I_{\beta'+1}\subseteq I_{\beta}$ and
$x_{\beta}\not \in I_{\beta}$, $x_{\beta}\not \in F_{\beta'}$. A
contradiction. On the other hand, if  $\beta<\beta'$ then,  since
$x_\beta\in F_{\beta'}$, $(x_\beta, x_{\beta' })\in \rho$. Since $I$
is an initial segment of $R$, $x_\beta\in I$. A contradiction too.
Consequently $J\subseteq I$. Exchanging the roles of $I$ and $J$,
yields $I\subseteq J$.  The equality $I=J$ follows. For the second
part of the claim, it suffices to show that $\phi(I)\in P$ for every
$I\in I_{<\omega }(R)$. This fact is a straightforward consequence
of Claim \ref{claim:finite}. Indeed, from this claim $I$ is finite.
Hence $\phi(I)$ is finite and thus belongs to $P$.

\begin{claim} \label{claim:linear} The order $\hat \rho$ has a linear extension of type  $\omega$.
\end{claim}
\noindent { \bf Proof of claim \ref{claim:linear}.}  Clearly, $[\omega]^{<\omega}$ has a linear extension of type $\omega$.  Since $R$ embeds  in $[\omega]^{<\omega}$, via an embedding  in $P$, the induced linear extension on $R$ has order type $\omega$.\endproof

Let $\rho'$ be the intersection of such a linear extension with the order $\theta$ and let $S:= (X, \rho')$.

\begin{claim} \label{claim:newembedding} For every $I\in I(S)$, resp. $I\in I_{<\omega}(S)$ we have $I\in I(R)$, resp. $I\in I_{<\omega}(R)$. \end{claim}
\noindent { \bf Proof of claim \ref{claim:newembedding}.} The first part of the proof follows directly from the fact that $\rho'$ is a linear extension of $\hat \rho$. The second part follows from the fact that each $I\in I_{<\omega}(S)$ is finite.
\endproof

It is then easy to check that the poset $S$ satisfies the properties stated in the theorem. \end{proof}

\subsection{Proof of Theorem \ref{thm:qalpha}.} Apply Theorem
 \ref{thm:qalphaprelim} and Lemma \ref{lem:qalphaprelim}. \endproof

\section[Lattice sierpinskisations]{Lattice sierpinskisations and a proof of Theorem \ref{thm:serpinskalpha}}
We show that monotonic sierpinskisations can be identified to  special subsets of the direct product of $\omega$ and $\alpha$ equipped with the induced ordering. Among these subsets we look at those which are join-subsemilattices of the direct product $\omega\times \alpha$, that we call \emph{lattice sierpinskisations}.

\begin{notation}\label{notat:sierp}Let $A$ be a set. Let $p_1$, resp. $p_2$, be the first , resp.  the second projection from the cartesian product  $\N\times A$ onto $\N$, resp. onto $A$. Let $\mathcal S$ be the set of subsets $S$ of $\N\times A$
such that:

\begin{enumerate}

\item Every \emph{vertical} line $S(n) :=
\{\beta \in A : (n, a ) \in  S\}$ is non-empty and finite;
\item Every \emph{horizontal} line $S^{-1} (a ) :=
\{n \in  \N : (n, a ) \in  S\}$ is an infinite subset of $\N$.
\end{enumerate}
\end{notation}

We equip  $A$ with a linear order $\leq$. Let $\alpha$ be the order
type of $(A, \leq )$. The set $\N$ will be equipped  with the
natural order, providing  a chain of order type $\omega$. We equip
the cartesian product $\N\times A$ with the direct product of these
two orders, and each subset $S$ of $\N\times A$ with the induced
order. We denote by $\mathcal S(\alpha)$ the collection of these
posets. We denote by $L_1$ the lexicographic ordering on $\N\times
A$, that is $(n', \beta')\leq_{L_1} (n'', \beta'')$ if either $n'<
n''$ w.r.t. the natural ordering on $\N$ or $n'=n''$ and $\beta'
\leq \beta''$ w.r.t. the order on $A$. And we denote  by $L_2$ the
reverse lexicographic ordering on $\N\times A$, that is $(n',
\beta')\leq_{L_2} (n'', \beta'')$ if either $\beta' < \beta''$
w.r.t. the order on $A$ or  $\beta' =\beta''$ and $n'\leq n''$
w.r.t. the natural ordering on $\N$. With these orders, $(\N\times
A, L_1)$ and $(\N\times A, L_2)$ have respectively type
$\alpha\omega$ and $\omega\alpha$. We consider bijective maps
$\varphi: \N \rightarrow \N\times A$ such that $\varphi^{-1}$ is
order-preserving from $\N\times \{a\}$ onto $\N$, these sets being
equipped with the natural ordering.

\begin{proposition}There is a one-to-one correspondence between   members of $\mathcal S(\alpha)$ and monotonic sierpinskisations of $\omega\alpha$ and $\alpha$.
\end{proposition}
\begin{proof}  First, if $\alpha=1$, $\mathcal S(\alpha)= \{\N\times A \}$, whereas there is just one monotonic sierpinskisation of $\alpha$ and $\omega$. Hence, we  may suppose $\alpha\not =1$. Let $S\in \mathcal S(\alpha)$. According to item 1 of Notation \ref{notat:sierp}, $(S, L_{1 \restriction S})$ has order type $\omega$. Let $\vartheta_1$ be the unique order isomorphism from  $(S, L_{1 \restriction S})$ on $(\N, \leq)$. Similarly, according to item 2 of Notation \ref{notat:sierp}, $(S, L_{2 \restriction S})$ has order type $\omega\alpha$. Let $\vartheta_2$ be the unique order-isomorphism from  $(S, L_{2 \restriction S})$ onto  $(\N\times A, L_2)$  such that $p_2(\vartheta_2(n, \beta))=\beta$ for all $\beta\in A$. And let   $\varphi:= \vartheta_2\circ \vartheta _1^{-1}$. The chains $(\N, \leq)$ and $(\N\times A, L_2)$ have respective order types $\omega$ and $\omega\alpha$. The map $\varphi$ defines on $\N$  a monotonic sierpinskisation of $\omega\alpha$ and $\omega$. Clearly, $\vartheta^{-1}$ is an order-isomorphism from $S$ equipped with the order induced by the direct product $\omega\times \alpha$ and the monotonic sierpinskisation of $\omega\alpha$ and $\alpha$ associated to $\varphi$.
Conversely, let $\varphi : \omega \rightarrow \omega\alpha$ be a map defining a monotonic sierpinskisation $S:= (\N, \leq)$.  That is $\varphi$ is a map from $(N, \leq)$ into $(\N\times A, L_2)$ such that $\varphi^{-1}$  is order preserving on each set of the form $\N\times \{a\}$ for $a \in A$. And the order $\leq$ on $\N$ is the intersection of the natural order on $\N$ and the inverse image of the order $L_2$.

\begin{claim}\label{claim:welldefined}There is a  surjective map $r:\N \rightarrow \N$, such that for every $n\in \N$,
$r^{-1}(n)$\;  \text{is the largest initial segment  of}\;  $G(n)= \N\setminus \bigcup \{r^{-1}(n'):n'<n\}$ \; \text {on which}\;  $p_2\circ \varphi \; \text{is strictly increasing}$. This map is order preserving.
\end{claim}
\noindent{\bf Proof of Claim \ref{claim:welldefined}.}
Applying induction on $n$, we may observe that $\N\setminus G(n)$  is an initial segment of $\N$ w.r.t. the natural order on $\N$ and that  $p_2\circ \varphi$ cannot be strictly increasing on $G(n)$.
\endproof

Let $\theta: \N\rightarrow \N\times A$ defined by setting $\theta(n):=(r(n), p_2\circ \varphi(n))$ for every $n \in \N$.
\begin{claim}\label{claim:sierpmonotonic}
$\theta$ is an embedding of $S$ in $\N\times A$ equipped with the product ordering. Its image $S'$ belongs to $\mathcal S(\alpha)$. \end{claim}

\noindent{\bf Proof of Claim \ref{claim:sierpmonotonic}.} From the
fact that $r^{-1}(n)$ is finite, $S'(n)$ is finite. Also
$S'^{-1}(a)$ is infinite for every $a\in A$.  Hence $S'\in \mathcal
S(\alpha)$. The first part of the claim amounts to the fact  that
$(n', \varphi(n')) \leq (n'', \varphi(n''))$ is equivalent to
$(r(n'), p_2\circ \varphi(n')) \leq (r(n''), p_2\circ\varphi(n''))$.
Suppose  $(n', \varphi(n')) \leq (n'', \varphi(n''))$. This amounts
to $n'\leq n''$ and  $\varphi (n') \leq \varphi(n'')$. Since $r$ and
$p_2$ are order-preserving, we get $r(n')\leq r(n'')$ and  $p_2\circ
\varphi (n') \leq p_2\circ \varphi(n'')$, that is $(r(n'), p_2\circ
\varphi(n')) \leq (r(n''), p_2\circ\varphi(n''))$. Conversely,
suppose $(r(n'), p_2\circ \varphi(n')) \leq (r(n''),
p_2\circ\varphi(n''))$. This amounts to $r(n')\leq r(n'')$ and
$p_2\circ \varphi(n')) \leq p_2\circ\varphi(n'')$.

First $n'\leq n''$.
\noindent {\bf Case 1}. $r(n')=r(n'')$. Let $n:=r(n')$. By definition of $r$, $p_2\circ \varphi $ is strictly increasing on $r^{-1}(n)$. Since $p_2\circ \varphi(n') \leq p_2\circ\varphi(n'')$, this implies $n'\leq n''$.
\noindent {\bf Case 2}. $r(n')\not =r(n'')$. In this case, we have $r(n')<r(n'')$ and,  since $r$ is order-preserving, $n'<n''$.

Next $\varphi(n')\leq \varphi(n'')$.
{\bf Case 1}. $p_2\circ \varphi(n') \not = p_2\circ\varphi(n'')$. In this case  $p_2\circ \varphi(n')< p_2\circ\varphi(n'')$. From the definition of the ordering $\leq _{\omega\alpha}$,  $\varphi(n')<\varphi (n'')$.
{\bf Case 2}.  $p_2\circ \varphi(n')=p_2\circ\varphi(n'')$. Since $\varphi^{-1}$ is order-preserving  on each set of the form $\N\times \{a\}$, and $n'\leq n''$, $\varphi(n')\leq \varphi(n'')$.

From this we get  $(n', \varphi(n')) \leq (n'', \varphi(n''))$ as required. \endproof

With this claim the proof of the lemma is complete. \end{proof}

 From now on, we
 identify monotonic sierpinskisations of $\omega\alpha$ and $\omega$ with members of $\mathcal S(\alpha)$.

\begin{definition} We say that  a member of $\mathcal S(\alpha)$ which is a join-subsemilattice  of $\omega\times \alpha$ is a \emph{lattice sierpinskisation} of $\omega\alpha$ and $\omega$.  We will denote by $\mathcal S_{Lat}(\alpha)$ the subset of $\mathcal S(\alpha)$  consisting of lattice sierpinskisations.
\end{definition}

\begin{lemma} Let  $S$ be a subset of $\N\times A$
such that:
\begin{enumerate}
\item Every \emph{vertical} line $S(n) :=
\{\beta \in A : (n, a ) \in  S\}$ is non-empty and finite;
\item \label{item:sierp2}Every \emph{horizontal} line $S^{-1} (a ) :=
\{n \in  \N : (n, a ) \in  S\}$ is a cofinite subset of $\N$.
\end{enumerate}
Then, $S$ equipped with the order induced by the direct product $\omega\times \alpha$ belongs to $\mathcal S_{Lat}(\alpha)$.
\end{lemma}
\begin{proof}The fact that $S$ is a join-subsemilattice of $\N\times \A$  follows from the second condition. Indeed,  let $x:=(n, \beta), y:=(m,\gamma)\in S$. W.l.o.g. we may assume $\beta\leq \gamma$. If $n\leq m$ we have $x\leq y$ and their supremum is $y$. If $m \not \leq n$ then $n<m$. According to Condition (\ref{item:sierp2}),  $z:=(n, \beta)\in S$. Since $z$  is the supremum of $x$ and $y$ in $\omega\times \alpha$,  $z$ is their supremum in $S$.
\end{proof}

\begin{proposition} If $S, S'\in \mathcal S_{Lat}(\alpha)$ there is a map $t:\N\rightarrow \N$ preserving the natural order such that the map $(t, 1_A)$ induces a join-embedding map from $S$ to $S'$

\end{proposition}
\begin{proof}
\begin{claim} \label{claim:sierpinduction} For every $n,n'\in \N$ such that $n<n'$ there is some
$n''\in \N$ such that $n''>n'$ and $S(n)\subseteq S'(n'')$.
\end{claim}
\noindent{\bf Proof of Claim \ref{claim:sierpinduction}.} Since $S'^{-1}(a)$ is infinite for every $a\in A$, we may select $n'_a> n'$ such that $(n'_a, a)\in S'$ for every $a\in A$. Let $m:=Max \{n'_a:a\in S(n)\}$. Let $a_0$ be the least element of $S(n)$ w.r.t. the ordering on $A$. Let $n''\geq m$ such that $(n'', a_0)\in S'$. Then $S(n)\subseteq S'(n'')$. Indeed, let $a\in S(n)$. If $a=a_0$, $a\in S'(n'')$ by definition. If $a\not =a_0$  then, since  $(n_a, a)\in S'$, $(n'',a_0)\vee(n_a, a)=(n'',a)\in S'$, hence $a\in S'$ as required.
\endproof

Claim \ref {claim:sierpinduction} allows us to define $t$
inductively. We suppose $t$ defined for all $m\in \N$ such that
$m<n$. From Claim \ref {claim:sierpinduction} there is some $n''$
such that $n''>t(m)$ for all $m<n$ and $S(n)\subseteq S'(n'')$. We
set $t(n):=n''$ where $n''$ is the least $n''$ satisfying this
property. The pair $(t, 1_A)$ is a join-embedding map from $S$ to
$S'$. Indeed, let $(n,a), (n',a')\in S$. W.l.o.g. we may assume
$n\leq n'$, hence  $(n,a)\vee(n',a')=(n',a'')$, where
$a'':=Max_A(\{a,a'\})$. Since $S(n')\subseteq S'(t(n'))$, $a''\in
S'(t(n'))$, hence $(t(n), a) \vee(t(n'), a')=(t(n'), a'')$, as
required. \end{proof}

\begin{notation} Since all members of $\mathcal S_{Lat}$ are embeddable  in each others as  join-semilattices, we may  denote by a single expression, namely $\Omega_L(\alpha)$,  an arbitrary member $S$
of $\mathcal{S}_{Lat}(\alpha)$ and by  $\underline\Omega_L(\alpha)$ the
join-semilattice $\underline S$ obtained by adding  a least element to $S$ (if it does not have one).  Since for each $x\in \underline S$ the initial segment $\downarrow x$ is finite, $\underline S$ is in fact a lattice (but not necessarily a sublattice of $\omega\times\alpha$) hence the name of lattice sierpinskisation. \end{notation}

\begin{example}
If $\alpha=1$, monotonic and lattice sierpinskisations of $\omega$ and $\omega$ are isomorphic to $\omega$, thus we have $\Omega_L(1)=\Omega(1)=\omega$. If $\alpha=n$, with $0<n<\omega$, the direct product $\omega\times n$ is a  lattice sierpinskisation of $\omega n$ and $\omega$. There are others, but we will not hesitate to write  $\Omega_L(n)=\omega\times n$. If $\alpha=\omega$, the subset $X:=\{(i,j): j\leq i<\omega\}$ of the direct product $\omega\times \omega$ is a lattice sierpinskisation of $\omega^2$ and $\omega$. This is a lattice isomorphic to $[\omega]^2$, the set of pairs $(i,j)$ such that $i<j<\omega$ componentwise ordered, and we will write $\Omega_L(\omega)=[\omega]^2$. We may also observe that $\mathcal S(\omega^*)$ contains the join-semilattice $\Omega(\omega^*)$  defined previously.  Also,
$\mathcal S(1+\eta)$ contains the lattice represented Figure
\ref {Omegachap2} and denoted by $\Omega(\eta)$ as well as $\underline\Omega(\eta)$.
\end{example}
\begin{proposition} \label{prop:sierpkey}Let $(A,\leq)$ be a chain and  $S$ be a join-subsemilattice of the product $\N\times A$ equipped with the product ordering. Let $A':=p_2(S)$, $F:= \{a'\in A': S^{-1}(a')\;  \text{is infinite}\; \}$, $I:=A'\setminus F$ and  $S_I:=S\cap (\N \times I)$. Let $\alpha$ be the order type of $(F, \leq_{\restriction F})$ and $m$ be the length of the longuest chain in $S_I$ if $S_I$ is finite.
 If  the  vertical lines $S(n)$ of $S$ are finite for all integers $n$ then $S$ contains a join-subsemilattice $S'$ such that  \begin{enumerate}
\item $S'$ is isomorphic to $m+ S''$ where $S''\in \Omega_{L}(\alpha)$ if $S_I$ is finite.
\item   $S'\in \Omega_{L}(1+\alpha)$ if $S_I$ is infinite.
\end{enumerate}
Moreover,  a chain is embeddable in $J(S)$ if and only if it is embeddable in $J(S')$.
\end{proposition}


%

\begin{proof}
\begin{claim} \label{item:lem:sierpkey0} Let $a\in A$ such that $ S^{-1}(a)\not= \emptyset$. Then $(\N\times\{a\}) \cap S$ is cofinal \index{cofinal} in
$(\N\times (\leftarrow a]) \cap S$.
\end{claim}
\noindent{\bf Proof of Claim \ref{item:lem:sierpkey0}.}  Let $(n', a')\in (\N\times (\leftarrow a]) \cap S$. Let $n\in \N$ such that $(n,a)\in S$. Since $S$ is a join-subsemilattice of $\N\times A$,  $(n', a') \vee (n, a)= (Max \{n, n'\}, a )\in S$. Hence, $(n', a')$ is majorized by $(Max \{n, n'\}, a )\in (\N\times\{a\}) \cap S$. This proves that $(\N\times\{a\}) \cap S$ is cofinal. \endproof

\begin{claim} \label{item:lem:sierpkey1}The set $F$    is a final segment of $A'$ equipped with the order induced by the order on $A$.
\end{claim}

\noindent {\bf Proof of Claim \ref{item:lem:sierpkey1}.} Let $a'\in F$ and $a\in A'$ such that $a'\leq a$.  Clearly, each element $(a,n)\in S$ dominates only finitely many elements $(a', n')\in S$. According to Claim \ref{item:lem:sierpkey0}, $(\N\times\{a\}) \cap S$ is cofinal, thus if $S^{-1}(a)$ is finite,  $S^{-1}(a')$ is finite too. A contradiction. Hence, $a\in F$. Proving that $F$ is a final segment. \endproof

Let $S_F:= S\cap (N\times F)$.
\begin{claim} \label{claim:sierpkey1bis} $S_I$  is a join-subsemilattice of $S$ and if $N'$ is a final segment of $p_1(S_F)$ equipped with the natural order on $\N$, $S\cap (N'\times F)$  is a join-subsemilattice of $S$. Furthermore, if  the  vertical lines $S(n)$ are finite for   all integers $n$, $S\cap (N'\times F)$ is a lattice sierpinskisation of $\omega\alpha$ and $\omega$.
\end{claim}

\noindent{\bf Proof of Claim \ref{claim:sierpkey1bis}.} Since $F$ is a final segment of $A'$ (Claim \ref{item:lem:sierpkey1}), $I$ is an initial segment of $A'$, hence $S_I$ is a join-subsemilattice of $S$.  By the same token $S\cap (N'\times F)$  is a join-subsemilattice of $S$.   If the  vertical lines $S(n)$ are finite for   all integers $n$, then  $S\cap (N'\times F)$ satisfies the conditions of Notation \ref{notat:sierp} with $N'$ and $F$ instead of $\N$ and $A$.\endproof

From now on, we suppose that the  vertical lines $S(n)$ are finite for   all integers $n$.
\begin{claim} \label{item:lem:sierpkey2} Every proper ideal of $S_I$ is finite.
\end{claim}
\noindent {\bf Proof of Claim \ref{item:lem:sierpkey2}.}
We prove first that  if  $\nu$ is  the order type of $(I, \leq_{\restriction I})$ then $\nu\leq \omega$.  For that, it suffices to prove that for every $a\in I$, the initial segment $(\leftarrow a]\cap A'$ is finite. Let $a\in I$. With the fact that all  the  vertical lines $S(n)$ are finite, we get that the initial segment of $\N\times A$ (equipped with the product order) generated by  $(\N\times\{a\}) \cap S$ is finite.  Since from   Claim \ref{item:lem:sierpkey0},  $(\N\times\{a\}) \cap S$ is cofinal in
$(\N\times (\leftarrow a]) \cap S$, it follows that this latter set is finite. Thus, its second projection is finite too. This second projection being $(\leftarrow a]\cap A'$,  our assertion is proved. Now, let $I'$ be  a proper ideal of $S_I$. Suppose by contradiction that $I'$ is infinite. Then necessarily,  $p_2(I')\not=I$. Otherwise since $I'$ is a proper ideal of $S_I$, $p_1(I')\not =p_1(S_I)$. Since $p_1(I')$ is an initial, segment of $p_1(S_I)$, $p_1(I')$ is finite; since each vertical line $S(n)$  for $n\in p_1(I')$ is finite, $I'$ is finite. Now, pick  $a\in I\setminus p_2(I')$.  As above,  $(\N\times\{a\}) \cap S$ is cofinal in
$(\N\times (\leftarrow a]) \cap S$, thus this set, and in particular $I'$ is finite. \endproof

With these claims, the proof of the proposition goes as follows:

\noindent {\bf Case 1.} $S_I$ is finite. In this case, let $N':=p_1(S_F)\cap [n \rightarrow)$ where $n=0$ if $S_I$ is empty and $n=p_1(x)$ where $x$ is the largest element of $S_I$ otherwise. Let  $M$ be the largest sized subchain of $S_I$, $m:=\vert M \vert$,   let  $S'':= S\cap (N'\times F)$ and let $S':=M+S''$.

\noindent {\bf Case 2.} $S_I$ is  infinite. In this case, it follows from Claim \ref{item:lem:sierpkey2} that $S_I$ contains a cofinal chain of type $\omega$. Let $D$ be such a chain. Add an extra element $\{a\}$ to $F$ with the  requirement that $a\leq b$ for all $b\in F$, set $N':= p_1(S_F)$, $S'':= S_F$.   Let $S'$ be the subset of $\N'\times (\{a\} \cup F)$ made of $S''$ and $(p_1(D)\cap N') \times \{a\}$.

In case 1, it  follows from   Claim \ref{claim:sierpkey1bis} that $S''$  is a lattice sierpinskisation of $\omega\alpha$ and $\omega$. By construction $S'$ is a join-subsemilattice of $S$ isomorphic to $m+ S''$. In case 2, $S'$ is a lattice sierpinskisation of $\omega(1+ \alpha)$ and $\omega$. And one can chek that $S'$ is embedabble in $S$  as join-semilattice.

To conclude, we only need to check that the same chains are embedabble in $J(S)$ and $J(S')$. Let $C\subseteq J(S)$ be a chain. Set $C'':= C\cap (\N \times F) $ and $C':= C\cap (\N\times I)$. Since  $S\cap (N'\times F)$ and $S''$ are  sierpinskisations of $\omega\alpha$ and $\omega$ (Claim \ref{claim:sierpkey1bis}), it follows from Corollary \ref{cor:ordertypesierp} that $C''$ is embeddable in $J(S'')$. In case 1, since $\vert C'\vert \leq \vert M\vert$, it follows that $C= C'+C''$ is embeddable in $J(S)$. In case 2, the order type of $C'$ is at most $\omega+1$.  By the same token,  $C$ is embedabble in $J(S')$.
\endproof

\end{proof}

{\bf Proof of Theorem \ref{thm:serpinskalpha}. }Let  $\alpha$. First, $P_{\alpha}\in \J_{\alpha}$. This follows readily from the fact that $n+\underline \Omega_L(\alpha')\in \J_{n+ \alpha'}$,  $\underline \Omega_L (1+\alpha')\in \J_{\omega+\alpha'}$ and $\underline \Omega_L (\alpha')\in \J_{\alpha'}$ for every $\alpha'$ and $n<\omega$.  Next, let $S$ be a join-subsemilattice of $P_{\alpha}$ having a least element and  belonging  to $\J_{\alpha}$. Clearly, the join-semilattice $P_{\alpha}$ is embeddable as a join-semilattice in a product $A\times \N$, where $A$ is a chain. Let $S'$ be a join-subsemilattice  of $S$ satisfying the properties of  Proposition  \ref{prop:sierpkey}.  Clearly  $J(S')\in \J_{\alpha}$. We may suppose that $S'$ has a least element (otherwise, add the least element of $S$). Thus the  join-semilattice $S'$ is of the form $m+\underline \Omega_L(\beta)$ with $m<\omega$. We claim that $P_{\alpha}$ is embeddable in $S'$ by a join-preserving map.  If $\alpha+1 \leq \alpha$, $P_{\alpha}$ is of the form $\underline \Omega_L(\gamma)$. Lemma \ref {lem:somewhatcrucial} yields that $\gamma$ and $\beta$ are equimorphic.  If $\alpha+1\not \leq \alpha$,  $P_{\alpha}:=n+\underline \Omega_L(\alpha')$ where $n<\omega$, $\alpha'$ has no least element and $\alpha=n+ \alpha'$. In this case, using Lemma \ref {lem:somewhatcrucial} one obtain  that $m\geq n$ and $\beta$ is equimorphic to $\alpha$. In both cases the existence of an embedding follows. \endproof

\section{ Some examples of obstructions}
As already observed, we have $J_{\neg \alpha}=Forb_{\J}(\{1+\alpha\})$ for every $\alpha\leq \omega$.
The first interesting case is $\alpha= \omega+1$.
We have:

\begin{lemma}\label{omega+1} \begin{equation}
\mathbb{J}_{\neg (\omega+1)}=Forb_{\J}(\{\omega+1, \omega\times
2\})=Forb_{\J}(\{\omega+1, Q_{\omega+1}\}).
\end{equation}
\end{lemma}

\begin{proof}
Let $P\in\J_{\alpha}$ such that $\omega+1\nleqslant P$. In $J(P)$ we
have a strictly increasing chain $(I_{\gamma})_{\gamma<\omega+2}$.
We construct in $P$ a join-subsemilattice $\{x_{i,j}: 0\leq i\leq 1,
j<\omega\}$ isomorphic to $\omega\times 2$ such that $x_{1,j}\in
(I_{\omega+1}\setminus I_{\omega})\cap\uparrow x_{0,j}$ for every
$j<\omega$. Pick $x_{0,0}$ in $I_{0}$ and $x_{1,0}$ in
$(I_{\omega+1}\setminus I_{\omega})\cap\uparrow x_{0,0}$. Suppose
$x_{0,0}, x_{1,0}, \ldots, x_{0,j}, x_{1,j}$ constructed. Since
$\omega+1\nleq P$ and $x_{1,j}\in(I_{\omega+1}\setminus
I_{\omega})\cap\uparrow x_{0,j}$, there is $n_{j}<\omega$ such that
$x_{1,j}$ does not dominate $I_{n}\cap\uparrow x_{0,j}$ for $n\geq
n_{j}$. Pick $x_{0,j+1}\in I_{n_{j}}\cap\uparrow x_{0,j}$ and put
$x_{1,j+1}:= x_{0,j+1}\vee x_{1,j}$.
\end{proof}

Next, we show:
\begin{equation} \mathbb{J}_{\neg(\omega+2)}=Forb_{\J}(\{\omega+2, (\omega \times 2)+1, \omega\times 3\})
\end{equation}

More generally, we  solve the case $\omega+n$.

\begin{lemma}\label{omega(k)} Let $n<\omega$ and $1\leq k \leq
n+1$. Then $\Omega(k)+n+1-k \in \J_{\omega+n}$.
\end{lemma}
\begin{proof} We may suppose $\Omega(k)=\omega\times k$. Hence $J(\Omega(k)+n+1-k)=J((\omega\times k)+n+1-k)=(J(\omega) \times J(k))+n+1-k=((\omega+1)\times
k)+n+1-k$. Hence $I(\omega+n)=\omega+n+1\leq J((\omega \times
k)+n+1-k)$.
\end{proof}

\begin{lemma}\label{omega+n} Let $\alpha:=\omega+n$ with $n\geq 1$. Then:
\begin{equation}
\mathbb{J}_{\neg \alpha}=Forb_{\J}(\{\Omega(k)+n+1-k: 1\leq k\leq n+1\}).
\end{equation}
\end{lemma}

\begin{proof} Lemma \ref{omega(k)} implies $\Omega(k)+n+1-k\in\J_{\omega+n}$ for every $1\leq k\leq
n+1$. To prove that this is a complete set of obstructions,  we proceed by recurrence on $n$. The case $n=1$ is solved in
 Lemma \ref{omega+1}. Suppose
$\J_{\omega+n}=\uparrow\{\Omega(k)+n+1-k: 1\leq k\leq n+1\}$. Let
$P\in\J_{\omega+n+1}$ such that $P$ contains no join-subsemilattice
isomorphic to $\Omega(k)+n+2-k$ for every $1\leq k\leq n+1$. In
$J(P)$ we have a strictly increasing chain
$(I_{\gamma})_{\gamma<\omega+n+2}$. Pick $x\in I_{0}$ and $x'\in
(I_{\omega+n+1}\setminus I_{\omega+n})\cap\uparrow x$. Put $P':=[x,
x'[$. Clearly $P'\in\J_{\omega+n}$. Since $P$ does not contain a
join-subsemilattice isomorphic to $\Omega(k)+n+2-k$ for every $1\leq
k\leq n+1$, $P'$ does not contain a join-subsemilattice isomorphic
to $\Omega(k)+n+1-k$, for every $1\leq k\leq n$. Recurrence's
hypothesis implies that $P'$ contains a join-subsemilattice
$\{z_{i,j}:0\leq i\leq n, j<\omega\}$ isomorphic to $\Omega(n+1)$.
We construct in $P$ a join-subsemilattice $\{x_{i,j}:0\leq i\leq
n+1, j<\omega\}$ isomorphic to $\Omega(n+2)$ such that $a)$ for
$0\leq i\leq n$, $x_{i,j}=z_{i,k}$ for some $k\geq j$  and $b)$
$x_{n+1,j}\in (I_{\omega+n+1}\setminus I_{\omega+n})\cap\uparrow
x_{n,j}$. Put $x_{i,0}:=z_{i,0}$ for $0\leq i\leq n$ and pick
$x_{n+1,0}\in (I_{\omega+n+1}\setminus I_{\omega+n})\cap\uparrow
x_{n,0}$. Suppose $x_{i,j}$ constructed for $0\leq i\leq n+1$ and
$0\leq j\leq m$. Since $P$ does not contain a join-subsemilattice
isomorphic to $\Omega(n+1)+1$ and $x_{n+1,m}\in
(I_{\omega+n+1}\setminus I_{\omega+n})\cap\uparrow x_{n,m}$ there is
$j_{m}>m$ such that $x_{n+1,m}\ngeq z_{n,j}$ for $j\geq j_{m}$. Put
$x_{i,m+1}:=z_{i,j_{m}}$ for $0\leq i\leq n$ and $x_{n+1,m+1}:=
x_{n,m+1}\vee x_{n+1,m}$.
\end{proof}

\begin{lemma}\label{omega2}
$\J_{\neg \omega2}=Forb_{\J}(\{\omega2,
\underline\Omega(\omega)\})$.
\end{lemma}

\begin{proof} Let $P\in\J_{\omega2}$. Suppose $\omega2\nleq P$.
Let $(I_{\gamma})_{\gamma<\omega2}$ be a strictly increasing chain
in $J(P)$. We construct a join-subsemilattice $\{x_{i,j}: i\leq
j<\omega\}$ of $P$, isomorphic to $\underline\Omega(\omega)$ such
that $x_{i,i}\in (I_{\omega+i}\setminus I_{\omega+i-1})\cap\uparrow
x_{i-1,i}$ for every $i<\omega$. Pick $x_{0,0}\in I_{0}$,
$x_{0,1}\in I_{1}\cap\uparrow x_{0,0}$ and $x_{1,1}\in
(I_{\omega+1}\setminus I_{\omega})\cap\uparrow x_{0,1}$. Suppose
$x_{i,j}$ constructed for $0\leq i\leq j\leq n$. Since $\omega2\nleq
P$ and $x_{n,n}\in (I_{\omega+n}\setminus
I_{\omega+n-1})\cap\uparrow x_{n-1,n}$, we have $\omega+1\nleq
\downarrow x_{n,n}$. Hence, there is some $j_{n}<\omega$ such that
$x_{n,n}$ does not dominate $I_{j}\cap\uparrow x_{0,n}$ for all $j$
such that $ j_{n}\leq j<\omega$. Pick $x_{0,n+1}\in
I_{j_{n}}\cap\uparrow x_{0,n}$, put $x_{i+1,n+1}:=x_{i,n+1}\vee
x_{i+1,n}$ for $0\leq i\leq n-1$ and pick $x_{n+1,n+1}\in
(I_{\omega+n+1}\setminus I_{\omega+n})\cap\uparrow x_{n,n+1}$.
\end{proof}

\chapter[Chains in modular algebraic lattices]{The length of chains in modular algebraic lattices}\label{chap:algebraic}
We show that, for a large class of countable \index{countable}
order types \index{order type} $\alpha$, a modular \index{modular
lattice} algebraic lattice \index{algebraic lattice} $L$ contains
no chain \index{chain}of type $\alpha$ if and only if $K(L)$, the
join-semilattice \index{join-semilattice} of compact elements
\index{compact element} of $L$, contains neither a chain of type
$\alpha$ nor a subset isomorphic \index{isomorphic} to
$[\omega]^{<\omega}$, the set of finite subsets of $\omega$.

\section{Introduction and presentation of the results}

This paper is about  the relationship between the order structure
of an algebraic lattice $L$ and the order structure of the
join-semilattice $K(L)$, made of  the compact elements of $L$,
particularly,  the relationship between the length of chains in
$L$ and in $K(L)$.  The lattice $L$ is isomorphic to $J(K(L))$,
the collection of ideals \index{ideal} of $K(L)$ ordered by
inclusion, but this does not make this relationship immediately
apparent. For an example, if $L$ is  $\mathfrak P(E)$, the power
set  of a set $E$, then $K(E)$ is  $[E]^{<\omega}$, the collection
of finite subsets of $E$.
 If $E$ is infinite, chains in $[E]^{<\omega}$ and in $\mathfrak P(E)$ are quite far apart:
maximal chains \index{maximal chain} in $[E]^{<\omega}$  have
order type $\omega$ whereas in $\mathfrak P(E)$  each maximal
chain is made of the set  $I(C)$ of initial segments of a chain
$C:= (E, \leq)$  where $\leq $ is a linear order on $E$. And, from
this follows that $\mathfrak P(E)$ contains  uncountable chains.
On an other hand, there are classes of lattices $L$ for which,
except this case, chains in $L$ and in $K(L)$ are about the same.
Our results are about a class $\mathbb L$  of algebraic lattices,
including the modular  ones.

Let $\alpha$ be a chain, we denote by $I(\alpha)$ the set of
initial segments \index{initial segment} of $\alpha$, ordered by
inclusion, we denote by $L_{\alpha}:= 1+(1\oplus \alpha)+1$ the
lattice made of the direct sum \index{direct sum} of the
one-element chain $1$ and the chain $\alpha$, with top \index{top}
and bottom \index{bottom} added. Note that $I(\alpha)\cong
J(1+\alpha)$ and that the algebraic lattice $J(L_{\alpha})$ made
of  the set of ideals of $L_{\alpha}$, ordered by inclusion, is
isomorphic to $L_{J(\alpha)}$.

\begin{figure}[htbp]
\centering
\includegraphics[width=3in]{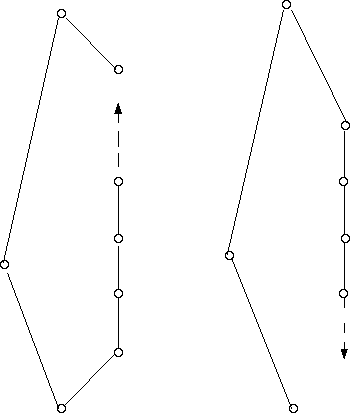}
\caption{$L_{\omega+1}$,     $L_{\omega^*}$} \label{L(omega)}
\end{figure}

\begin{figure}[htbp]
\centering
\includegraphics[width=0.8in]{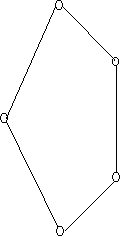}
\caption{$M_{5}$} \label{M5}
\end{figure}

Let $\mathbb A $ be the class of algebraic lattices  and let
$\mathbb L$  be the collection of $L\in\mathbb A$ such that $L$
contains no sublattice \index{sublattice} isomorphic to
$L_{\omega+1}$ or to $L_{\omega^*}$. Since $L_{2}$ is isomorphic
to $M_{5}$, the five element non-modular lattice, every modular
algebraic lattice belongs to $\mathbb L$. Let $\mathbb A_{\neg
\alpha}$ (resp. $\mathbb L_{\neg \alpha}$) be the collection of
$L\in\mathbb A$ (resp. $L\in\mathbb L$) such that $L$ contains no
chain of type $I(\alpha)$. Let $\mathbb{K}$ be the class of order
types $\alpha$ such that $L \in  \mathbb L_{\neg \alpha}$ whenever
$K(L)$ contains no chain of type $1+\alpha$ and no subset
isomorphic to $[\omega]^{<\omega}$. If $\alpha$ is countable then
these two conditions are necessary in order that $L\in\mathbb
L_{\neg \alpha}$; if, moreover, $ \alpha$ is indecomposable
\index{indecomposable}, this is equivalent to the fact that $L$
contains no chain of type $\alpha$.

\begin{theorem}\label{thm6}
The class $\mathbb{K}$ satisfies the following properties:\\
$(p_{1})$ $0\in\mathbb{K}$, $1\in \mathbb{K}$;\\ $(p_{2})$ If
$\alpha+1\in\mathbb{K}$ and $\beta\in\mathbb{K}$ then
$\alpha+1+\beta\in\mathbb{K}$;\\ $(p_{3})$ If
$\alpha_{n}+1\in\mathbb{K}$
 for every $n<\omega$ then the $\omega$-sum
$\gamma:=\alpha_{0}+1+\alpha_{1}+1+\ldots+\alpha_{n}+1+\ldots$
belongs to $\mathbb{K}$;\\ $(p_{4})$ If $\alpha_{n}\in\mathbb{K}$
and $\{m: \alpha_{n}\leq\alpha_{m}\}$ is infinite for every
$n<\omega$ then the $\omega^*$-sum $\delta:=
\ldots+\alpha_{n+1}+\alpha_{n}+\ldots+\alpha_{1}+\alpha_{0}$
belongs to $\mathbb{K}$.\\ $(p_{5})$ If $\alpha$ is a countable
scattered order type \index{scattered order type}, then $\alpha\in
\mathbb{K}$ if and only if
$\alpha=\alpha_{0}+\alpha_{1}+...+\alpha_{n}$ with $\alpha_{i}\in
\mathbb{K}$ for $i\leq n$, $\alpha_{i}$ strictly
left-indecomposable  \index{strictly left-indecomposable} for
$i<n$ and $\alpha_{n}$ indecomposable;\\
 $(p_{6})$  $\eta\in \mathbb{K}$.\\
\end{theorem}

\begin{corollary}\label {cor: scatteredmodular} A modular algebraic lattice is well-founded, respectively scattered, if and only if the join-semilattice made of its compact element is well-founded,   resp. scattered and does not contains a join-subsemilattice isomorphic to $[\omega]^{<\omega}$.
\end{corollary}

Let $\mathbb{P}$ be the smallest class of order-types satisfying
properties $(p_{1})$ to $(p_{4})$ above. For exemple, $\omega,
\omega^{*}, \omega(\omega^{\alpha})^{*}, \omega^{*}\omega,
\omega^{*}\omega\omega^{*}$ belong to $\mathbb{P}$. We have
$\mathbb{P}\subseteq \mathbb{K}\setminus \{\eta \}$. We do not
know whether every countable $\alpha\in \mathbb{K}\setminus \{\eta
\}$ is equimorphic \index{equimorphic} to some $\alpha' \in
\mathbb{P}$. As a matter of fact, the class  of order types
$\alpha$ which are equimorphic to some $\alpha' \in \mathbb{P}$
satisfies also $(p_{5})$ hence, to answer our question, we may
suppose that $\alpha$ is indecomposable. We only succeed when
$\alpha$ is indivisible \index{indivisible}. We eliminate
 those outside $\mathbb{P}$ by building algebraic lattices of the form $J(T)$ where $T$ is an appropriate distributive lattices \index{distributive lattice}.

\begin{theorem} \label{thm7}
 A countable indivisible order type  $\alpha$ is equimorphic to some $\alpha' \in \mathbb {P} \cup \{\eta\}$ if and only if for
 for every modular algebraic  lattice $L$, $L$ contains   a chain of type $ I(\alpha)$ if and only if  $L$ contains a
chain of type $1+\alpha$ or a subset isomorphic to
$[\omega]^{<\omega}$.
\end{theorem}

This contains the fact that for a countable ordinal $\alpha$,
$\alpha\in\K$ if and only if $\alpha\leq\omega$, a fact which can
be obtained directly by a straightforward sierpinskization
\index{sierpinskization}.

A characterization, by means of obstructions, of posets with no
chain of ideals of a given type in given in \cite{pz}. The
motivation for our results can be found in \cite {bonn}, a paper
describing a characterization of countable order types $\alpha$
such that the lattice \index{lattice} $I(P)$ of initial segments
of a poset $P$ contains no chain of order type $I(\alpha)$ if and
only if $P$ contains no chain of order type $\alpha$  and no
infinite antichain \index{antichain} (see  \cite {bonn-pouz} for a
proof). Results presented here were obtained first for the special
case of distributive algebraic lattices. They were included in the
thesis of the first author presented before the University
Claude-Bernard in december 1992\cite{chak}, and announced in \cite
{cp}.

\section{Definitions and basic notions}

Our definitions and notations are standard and agree with \cite
{grat} and \cite {rose} except on minor points that we will
mention.

\subsection{Posets, well-founded and scattered posets, ordinals and scattered order types}

Let $P$ be a set. A {\it quasi-order} \index{quasi-order} on $P$
is a binary relation \index{binary relation} $\rho$ on $P$ which
is reflexive \index{reflexive} and transitive \index{transitive};
the set $P$ equipped with this quasi-order is a {\it qoset}. As
usual, $x\leq y$ stands for $(x,y)\in\rho$.  An antisymmetric
\index{antisymmetric} quasi-order is an {\it order} \index{order}
and $P$ equipped with this order is a {\it poset} \index{poset}.
The {\it dual} of a poset $P$, denoted $P^*$,  is the set $P$
equipped with the {\it dual order} \index{dual order} defined by
$x\leq y$ in $P^*$ if and only if $y\leq x$ in $P$. Let $P$ and
$Q$ be two posets. A map $f:P\rightarrow Q$ is {\it
order-preserving} \index{order-preserving} if $x\leq y$ in $P$
implies $f(x)\leq f(y)$ in $Q$; this is an {\it embedding} \index{
embedding} if $x\leq y$ in $P$ is equivalent to $f(x)\leq f(y)$ in
$Q$; if, in addition, $f$ is onto, then this is an {\it
order-isomorphism} \index{order-isomorphism}. We say that $P$ {\it
embeds} into $Q$ if there is an embedding from $P$ into $Q$, a
fact we denote $P\leq Q$; if $P\leq Q$ and $Q\leq P$ then $P$ and
$Q$ are {\it equimorphic} \index{equimorphic}, a fact we denote
$P\equiv Q$; if there is an order-isomorphism from $P$ onto $Q$ we
say that $P$ and $Q$ are {\it isomorphic} \index{isomorphic} or
have the same {\it order-type} \index{order-type}, a fact we
denote $P\cong Q$. We denote $\omega$ the order type of $\N$, the
set of natural integers, $\omega^*$ the order type of the set of
negative integers and $\eta$ the order type of $\Q$, the set of
rational numbers. A poset $P$ is {\it well-founded}
\index{well-founded } if every non-empty subset $X$ of $P$ has a
minimal element. With the Axiom of dependent choice, this amounts
to the fact that the order type $\omega ^{*}$ of the negative
integers does not embed into $P$. If furthermore, $P$ has no
infinite antichain then $P$  is {\it well-quasi-ordered}
\index{well-quasi-ordered}, w.q.o. in brief. A well-founded chain
is {\it well-ordered} \index{well-ordered}; its order type is an
{\it ordinal} \index{ordinal}. A poset $P$ is {\it scattered}
\index{scattered} if it does not embed the chain of rationals. We
denote order types of chains and ordinals by greek letters
$\alpha$, $\alpha'$, $\beta$, $\beta'$\dots. If $(P_{i})_{i\in I}$
is a family of posets indexed by a poset $I$, the {\it
lexicographic sum} \index{lexicographic sum} of this family is the
poset, denoted $\Sigma_{i\in I}P_{i}$, defined on the disjoint
union of the $P_{i}$, that is formally the set of $(i,x)$ such
that $i\in I$ and $x\in P_{i}$, equipped with the order
$(i,x)\leq(j,y)$ if either $i<j$ in $I$ or $i=j$ and $x\leq y$ in
$P_{i}$. When $I$ is the finite chain $n:=\{0,1,\ldots,n-1\}$ this
sum is denoted $P_{0}+P_{1}+\ldots+P_{n-1}$. When $I:=\omega$ this
sum is denoted $\Sigma_{i<\omega}P_{i}$ or
$P_{0}+P_{1}+\ldots+P_{n}+\ldots$. We denote $\Sigma^*_{i\in
I}P_{i}$ for $\Sigma_{i\in I^*}P_{i}$; when $I:=\omega$ we denote
$\Sigma^*_{i<\omega}P_{i}$ or $\ldots+P_{n}+\ldots+P_{1}+P_{0}$.
When $I$ (resp. $I^*$) is well ordered, or is an ordinal, we call
{\it ordinal sum} \index{ordinal sum} (resp. {\it antiordinal sum}
\index{antiordinal sum} ) instead of lexicographic sum. When all
the $P_{i}$ are equal to the same poset $P$, the lexicographic sum
is denoted $P.I$, and called the {\it lexicographic product}
\index{lexicographic product} of $P$ and $I$. These definitions
extend to order types, particularly to order types of chains (note
that $2\omega=\omega$ whereas $\omega2=\omega+\omega$). Two order
type $\alpha$ and $\beta$ are {\it equimorphic} if
$\alpha\leq\beta$ and $\beta\leq\alpha$, a fact we denote
$\alpha\equiv\beta$. An order type $\alpha$ is {\it
indecomposable} \index{indecomposable} if $\alpha=\beta+\gamma$
implies $\alpha\leq\beta$ or $\alpha\leq\gamma$ (for exemple
$\eta$ is indecomposable); it is {\it right-indecomposable}
\index{right-indecomposable} if $\alpha=\beta+\gamma$ with
$\gamma\not=0$ implies $\alpha\leq\gamma$; it is {\it strictly
right-indecomposable} \index{strictly right-indecomposable} if
$\alpha=\beta+\gamma$ with $\gamma\not=0$ implies
$\beta<\alpha\leq\gamma$. The {\it left-indecomposability}
\index{left-indecomposable} and the {\it strict
left-indecomposability} \index{strictly left-indecomposable} are
defined in the same way. The notions of indecomposability and
strict right or left-indecomposability are preserved under
equimorphy (but not the right or the left-indecomposability). A
sequence $(\alpha)_{n<\omega}$ of order types is {\it
quasi-monotonic} \index{quasi-monotonic} if $\{m: \alpha_{n} \leq
\alpha_{m} \}$ is infinite for each $n$, its sum
$\alpha=\Sigma_{n<\omega}\alpha_{n}$ is right-indecomposable or
equivalently $\alpha'=\Sigma^*_{n<\omega}\alpha_{n}$ is
left-indecomposable. If an order type $\alpha$ can be written
under the form $\alpha=\alpha_{0}+\alpha_{1}+\ldots+\alpha_{n-1}$
with each $\alpha_{i}$ indecomposable, then it is a {\it
decomposition} \index{decomposition} of $\alpha$ of {\it lenght}
\index{lenght} $n$. This decomposition is {\it canonical} \index{
canonical} if its lenght is minimal.  An order type $\alpha$ is
{\it indivisible} \index{indivisible} if for every partition of a
chain $A$ of order type $\alpha$ into $B\cup C$, then either
$A\leq B$ or $A\leq C$. An indivisible order type is
indecomposable, an indecomposable scattered order type distinct of
$0$ and $1$ is strictly right-indecomposable or strictly
left-indecomposable (indecomposable, resp. indivisible, order
types are said additively indecomposable, resp. indecomposable, by
J.G.Rosenstein \cite{rose}).  \\ The following theorem rassembles
the fundamental properties of scattered order types
 (cf. the exposition given by J.G.Rosenstein \cite{rose}).
 \begin{theorem} \label {tmI.4.1}
$(i)$ Every countable order type is scattered or equimorphic to
$\eta$. (G. Cantor, 1897 ).\\ $(ii)$ The class $\mathcal{D}$ of
scattered order types is the smallest class of order types
containing $0, 1$ and stable by ordinal and antiordinal sum. (F.
Hausdorff, 1908).\\ $(iii)$ The class $\mathcal{D}$  is
well-quasi-ordered under embeddability. (R. Laver, 1971).\\ $(iv)$
For every countable scattered order type $\alpha$, the set
$\mathcal{D}(\alpha)$ of order types $\beta$ which embed into
$\alpha$, considered up to equimorphy, is at most countable. (R.
Laver, 1971).\\
\end{theorem}
An easy consequence of $(iii)$ of Theorem \ref{tmI.4.1} above is
this:
\begin{corollary} \label {corI.4.2} (R.Laver, 1971)
$(i)$ Every scattered order type is a finite sum of indecomposable
order types and has a partition into finitely many indivisible
order types. \\ $(ii)$ The class  of indecomposable countable
scattered order types is the smallest class of order types
containing $0, 1$ and stable by $\omega$ and and $\omega^*$  sums
of quasi-monotonic sequences.\\
\end{corollary}
We give, without proof, two,  more technical, results  we need.
\begin{lemma} \label {lemI.4.4}
Let $\alpha$ be a scattered order type distinct from  $0$ and
$1$.\\ $(i)$ The set  $Idiv(\alpha)$ of indivisible order types
$\beta$ which embeds into $\alpha$ is a finitely generated initial
segment of the collection of indivisible order types;\\ $(ii)$ let
$\nu$ be a maximal member of $Idiv(\alpha)$ then  $\nu+1$ (resp.
$1+\nu$) does not embed into $\alpha$ and $\nu$ is strictly right
(resp.left)-indecomposable if $\alpha$ is strictly right
(resp.left)-indecomposable.\\
\end{lemma}
\begin{lemma} \label{lemI.4.5}
Let $\alpha$ be a countable scattered order type. If $\alpha$  is
strictly right-indecomposable then
$\alpha=\Sigma_{\lambda<\mu}\alpha_{\lambda}$ where $\mu$ is an
indecomposable ordinal, every $\alpha_{\lambda}$ is strictly
left-indecomposable and verifies $\alpha_{\lambda}<\alpha$.
\end{lemma}
\subsection{Initial segments and ideals of a poset.} \label{subsection2.2}
Let $P$ be a poset, a subset $I$ of $P$ is an {\it initial
segment} \index{initial segment} of $P$ if $x\in P$, $y\in I$ and
$x\leq y$ imply $x\in I$. If
 $A$ is a subset of $P$, then $\downarrow A=\{x\in P: x\leq y$ for some $y\in A\}$ denotes the least initial segment containing
$A$. If $I= \downarrow A$  we say that $I$ is {\it generated}
\index{initial segment generated by} by $A$ or $A$ is {\it
cofinal} in \index{cofinal} $I$. If $A=\{a\}$ then $I$ is a {\it
principal initial segment} \index{principal initial segment} and
we write $\downarrow a$ instead of $\downarrow \{a\}$. A {\it
final segment} \index{final segment} of $P$ is any initial segment
of $P^*$. We denote by $\uparrow A$ the final segment generated by
$A$. If $A=\{a\}$ we write $\uparrow a$ instead of $\uparrow
\{a\}$. A subset $I$ of $P$ is {\it up-directed}
\index{up-directed} if every pair of elements of $I$ has a common
upper-bound \index{upper-bound} in $I$. An {\it ideal}
\index{ideal} is a non empty up-directed initial segment of $P$
(in some other texts, the empty set is an ideal). We denote
$I(P)$, resp. $I_ {<\omega }(P)$, resp. $J(P)$, the set of initial
segments, resp. finitely generated initial segments, resp. ideals
of $P$ ordered by inclusion and we set
 $ J_ *(P):=J(P)\cup\{\emptyset\}$, $I_ 0 (P):=I_ {<\omega}
(P)\setminus\{\emptyset\}$. We note that $I_{ <\omega}(P)$ is the
set of compact elements of $I(P)$, hence $J(I_{ <\omega}(P))\cong
I(P)$. Moreover $I_ {<\omega}(P)$ is  a lattice and in fact a
distributive lattice, if and only if, $P$ is {\it
$\downarrow$-closed} \index{$\downarrow$-closed} that is   the
intersection of two principal initial segments of $P$ is a finite
union, possibly empty, of principal initial segments.

We will need the following facts.
\begin{lemma} \label{final}
Let $P$ be  a poset  and $x\in P$; then, once ordered by
inclusion, the set $J_{x}(P):= \{J\in J(P): x\in J\}$ is
order-isomorphic to $J(P\cap\uparrow x)$.
\end{lemma}
\begin{proof}
Let $\psi: J_{x}(P)\rightarrow J(P\cap\uparrow x)$ be defined by
$\psi(J):=J \cap \uparrow x$. This map is clearly
order-preserving. The fact that it is an embedding holds on the
remark that if an ideal $J$ of $P$ contains  $x$  then $J\cap
\uparrow x$ is cofinal in $J$. Indeed, let $I,J\in J_{x}(P)$ such
that  $\psi(I)\subseteq\psi(J)$; then
$I=\downarrow\psi(I)\subseteq \downarrow \psi(J)=J$.
\end{proof}

Let $P$ and $Q$ be two posets, the {\it direct product}
\index{direct product} of $P$ and $Q$ denoted $P\times Q$ is the
set of $(p, q)$ for $p\in P$ and $q\in Q$, equipped with the
product order; that is $(p,q)\leq (p',q')$ if $p\leq p'$ and
$q\leq q'$. The {\it direct sum} \index{direct sum} of $P$ and $Q$
denoted $P\oplus Q$ is the disjoint union of $P$ and $Q$ with no
comparability between the elements of $P$ and the elements of $Q$
(formally $P\oplus Q$ is the set of couples $(x,0)$ with $x\in P$
and $(y,1)$ with $y\in Q$ equipped with the order $(p,q)\leq
(p',q')$ if $p\leq p'$ and $q=q'$).

\begin{lemma} \label{lemI.3.5}
Let $A_{1}$, $A_{2}$ be two posets, then $J(A_{1}+ A_{2})\cong
J(A_{1})+ J(A_{2})$, $J(A_{1}\times A_{2})\cong J(A_{1})\times
J(A_{2})$  and $I(A_{1}\oplus A_{2})\cong I(A_{1})\times
I(A_{2})$.\\
\end{lemma}
\begin{lemma} \label{propI.3.7}
If an indivisible chain  embeds into a product $A_1\times A_2$
then it embeds into $A_1$ or into $A_2$.
\end{lemma}
\begin{proof}
Let $Q$ be a chain such that $Q\leq  A_{1}\times A_{2}$. Then
$Q\leq  A'_{1}\times A'_{2}$ where $A'_i$ is a chain for $i=1,2$.
We have  $J(Q)\leq J(A'_{1}\times A'_{2})$. From Lemma
\ref{lemI.3.5},  $J(A'_{1}\times A'_{2})\cong J(A'_{1})\times
J(A'_{2})$. Since $I(A'_{i})=1+J(A'_{i})$, we have $I(Q)\leq
I(A'_{1})\times I(A'_{2})$. From Lemma \ref{lemI.3.5},
$I(A'_{1})\times I(A'_{2})\cong I(A'_{1}\oplus A'_{2})$. A maximal
chain of $I(A'_{1}\oplus A'_{2})$ extending $I(Q)$ is of the form
$I(C)$ where $C$ is  a chain extending the order on $A'_{1}\oplus
A'_{2}$. Since $Q\leq C$, $Q= B_1\cup B_2$ with  $B_i\leq A'_i$
for $i=1,2$. Hence, if $Q$ is indivisible,$Q\leq A'_{i}$, hence
$Q\leq A_i$,  for some $i\in \{1,2\}$. \end{proof}\\

\subsection{Join-semilattices.}
A {\it join-semilattice} \index{join-semilattice} is a poset $P$
such that  arbitrary elements $x, y$ have a join that we denote $x\vee y$. We denote $\mathbb J$ the class of join-semilattices
having a least element. If $P\in \J$ then, since  $J(P)$ is an
algebraic lattice, every maximal chain \index{maximal chain} $C$
is an algebraic lattice too, hence is of the form $I(D)$ where $D$
is some chain. Given an order type $\alpha$, let $\J_{\alpha}$ be
the class of $P\in \J$ such that $J(P)$ contains a chain of type
$I(\alpha)$ and let $\J_{\neg\alpha}:=\J\setminus\J_{\alpha}$. Let
$\mathbb B$ be a subset of $\J$, let $\uparrow \mathbb B$ be the
class of $P\in \J$ such that $P$ contains as a join-semilattice a
member of $\mathbb B$ and set $Forb_{\J}(\mathbb
B):=\J\setminus\uparrow \mathbb B$. It is natural to ask whether
for every countable order type $\alpha$, there is some finite
$\mathbb B$ such that $\J_{\neg \alpha}=Forb_{\J}(\mathbb B)$.
Looking at the class $\J'$ of $P\in \J$ such that $J(P)\in \mathbb
L$ we are just able to provide many $\alpha$'s such that:
$$\J'_{\neg \alpha}=Forb_{\J'}(\{1+\alpha, [\omega]^{<\omega}\}$$

\begin{lemma}\label{lem01} Let $P,Q$ be two join-semilattices with
a zero. Then:\\ $(i)$ $Q$ embeds into $P$ as a join-semilattice
iff $Q$ embeds into $P$ as a join-semilattice, with the zero
preserved.\\ $(ii)$ If $Q$ embeds into $P$ as a join-semilattice,
then $J(Q)$ embeds into $J(P)$ by a map preserving arbitrary
joins.\\ Suppose $Q:=I_{<\omega}(R)$ for some poset $R$, then:\\
$(iii)$ $Q$ embeds into $P$ as a poset iff $Q$ embeds into $P$ as
a join-semilattice.\\ $(iv)$ $J(Q)$ embeds into $J(P)$ as a poset
iff $J(Q)$ embeds into $J(P)$ by a map preserving arbitrary
joins.\\ $(v)$ If $\downarrow x$ is finite for every $x\in R$ then
$Q$ embeds into $P$ as a poset iff $J(Q)$ embeds into $J(P)$ as a
poset.\\
\end{lemma}
\begin{proof}
$(i)$ Let $f: Q\rightarrow P$ satisfying $f(x\vee y)=f(x)\vee
f(y)$ for all $x, y\in Q$. Set $g(x):=f(x)$ if $x\neq 0$ and
$g(0):=0$. Then g preserves arbitrary finite joins.\\ $(ii)$ Let
$f: Q\rightarrow P$ and $\overline{f}: J(Q)\rightarrow J(P)$
defined by $\overline{f}(I):=\downarrow\{f(x): x\in I\}$. If $f$
preserve finite joins, then $\overline{f}$ preserves arbitrary
joins.\\ $(iii)$ Let $f: Q\rightarrow P$. Taking account that
$Q:=I_{<\omega}(R)$, set $g(\emptyset):=0$ and
$g(I):=\bigvee\{f(\downarrow x): x\in I\}$ for each $I\in
I_{<\omega}(R)\setminus \{\emptyset\}$.\\ $(iv)$ Let $f$ be an
embedding from $J(Q)$ into $J(P)$. Taking account that $J(Q)\cong
I(R)$, set $g(I):=\bigvee\{f(\downarrow x): x\in I\}$ for each
$I\in I(R)\setminus\{\emptyset\}$ and $g(\emptyset):=0$. Note that
$g(\downarrow x)=f(\downarrow x)$ for all $x\in I$ and
$g(I)\subseteq f(I)$ for all $I\in I(R)$. Suppose $I\nsubseteq J$.
Let $x\in I\setminus J$. Since $x\in I$, we have $f(\downarrow
x)\nsubseteq f(J)$. Since $g(J)\subseteq f(J)$ we have
$f(\downarrow x)\nsubseteq g(J)$. Hence $g(I)\nsubseteq g(J)$.
\end{proof}\\

%
The fact that a join-semilattice $P$ contains a
join-subsemilattice isomorphic to $[\omega]^{<\omega}$  amounts to
the existence of an infinite independent set \index{independent
set}. Let us recall that a  subset $X$ of a join-semilattice $P$
is {\it independent} if $x\not \leq \bigvee F$ for every $x\in X$
and every non empty finite subset $F$ of $X\setminus \{x\}$.

\begin{theorem} \cite {chapou}  \cite {lmp3} \label  {tm4.1} Let $\kappa$ be a cardinal number; for a join-semilattice $P$
the following properties are equivalent:\\ $(i)$ $P$ contains an
independent set of size $\kappa$;\\ $(ii)$ $P$ contains  a
join-subsemilattice isomorphic to $[\kappa]^{<\omega}$;\\ $(iii)$
$P$ contains  a subposet isomorphic to $[\kappa]^{<\omega}$;\\
$(iv)$ $J(P)$ contains a subposet isomorphic to $\mathfrak P
(\kappa)$;\\ $(v)$ $\mathfrak P (\kappa)$ embeds into $J(P)$ via a
map preserving arbitrary joins.\end{theorem}

\begin{proposition} \label{corI.1.2}
Let $P$ be a poset. The following properties are equivalent:\\
$(i)$ $P$ contains an antichain of size $\kappa$;\\ $(ii)$
$I_{<\omega}(P)$ contains a subset isomorphic to
$[\kappa]^{<\omega}$ ;\\ $(iii)$ $\mathfrak {P}(\kappa)$ embeds
into $I(P)$.\\
\end{proposition}

\section [The class $\mathbb{K}$]{The class $\mathbb{K}$ and a proof of Theorem \ref {thm6}}
\subsection{Property $(p_{1})$}
The fact $(p_{1})$ holds is obvious.
\subsection{ Properties $(p_{2})$ and $(p_{3})$}
\begin{lemma} \label{lemIII.4.2}
    Let $\alpha$  and $\beta$ be two order types such that
$\alpha+1\in\mathbb{K}$. If $P\in\mathbb L_{\alpha+1+\beta}$ and
$P$ contains no infinite independent set then $P$ contains an
element $x$ such that $1+\alpha+1$ embeds into $P':=\downarrow x$
and $P'':=\uparrow x\in\mathbb L_{\beta}$.
\end{lemma}
\begin{proof}
Let $C$ be a chain of type $\alpha+1+\beta$, let $c\in C$ such
that the type of $A:=\downarrow c$ is $\alpha+1$ and the type of
$B:=\{y : c<y\}$ is $\beta$. Let $\varphi$ be an embedding from
$I(C)$ into $J(P)$ and let $P_{1}:=\varphi(A)$. Since $P_{1}\in
J(P)$ then  $J(P_{1})$ is a sublattice of $J(P)$, and since
$J(P_{1})$ embeds $I(\alpha +1)$,   $P_{1}\in \mathbb
L_{\alpha+1}$.   Since $P_{1}$  contains no infinite independent
set and $\alpha+1\in \mathbb{K}$, then $1+\alpha+1$ embeds into
$P_{1}$. Let $f$ be an embedding from $A$ into $P_{1}$, $x: =f(c)$
and $P':=\downarrow x$. Clearly $1+\alpha+1\leq P'$.  Let
$P'':=\uparrow x$.  Clearly $P''\in \mathbb L$. According to Lemma
\ref {final}, $J(P'')\cong J_{x}(P)$; since trivially
$I(\beta)\leq   J_{x}(P)$,
 we get $P''\in \mathbb L_{\beta}$, as required.
\end{proof}
\subsubsection{Proof of  Property $(p_{2})$}
    Let $\alpha$ and $\beta$ be two order types such that $\alpha+1, \beta\in\mathbb{K}$.
Let $P\in\mathbb L_{\alpha+1+\beta}$. If $P$ contains no infinite
independent set then, from Lemma \ref{lemIII.4.2}, $P$ contains an
element $x$ such that $1+\alpha+1$ embeds into $P':=\downarrow x$
and $P'':=\uparrow x\in\mathbb L_{\beta}$. Since
$\beta\in\mathbb{K}$ and $P''\in\mathbb L_{\beta}$ then $1+\beta$
embeds into $ P''$, hence $1+\alpha+1+\beta$ embeds into $P$.
\endproof
\subsubsection{ Proof of Property $(p_{3})$}
    Let $\gamma:=\Sigma_{n<\omega}(\alpha_{n}+1)$  and  let
$\beta_{n}:=\Sigma_{m\geq n}(\alpha_{m}+1)$ for $n<\omega$. Let
$P\in\mathbb L_{\gamma}$. Suppose that $P$ has no infinite
independent set. We construct an infinite
 increasing sequence
$x_{0}, x_{1}, \ldots, x_{n}, \ldots$ of elements of $P$ such
that, for each $k<\omega$, we have $1+\alpha_{k}+1\leq Q_{k}$ and
$I(\beta_{k})\leq J(P_{k})$ where $Q_{k}$ is the interval $[x_{k},
x_{k+1}]$ and $P_{k}:=\uparrow x_{k}$. Put $x_{0}:=0$,
$P_{0}:=\uparrow x_{0}$. Since $P_{0}=P$ and $\beta_{0}=\gamma $,
we have $I(\beta_{0})\leq J(P_{0})$. Suppose $x_{0}, \ldots,
x_{n}$ constructed. We have $I(\beta_{n})\leq J(P_{n})$. Since
$\beta_{n}=\alpha_{n}+1+\beta_{n+1}$ and
$\alpha_{n}+1\in\mathbb{K}$, Lemma \ref{lemIII.4.2} asserts that
the semilattice $P_{n}$ contains an element $x_{n+1}$ such that
$1+\alpha_{n}+1\leq Q_{n}$ and $I(\beta_{n+1})\leq J(P_{n+1})$ for
$Q_{n}:=[x_{n}, x_{n+1}]$, $P_{n+1}:=\uparrow x_{n+1}$. Since
$1+\alpha_{n}+1$ embeds into $[x_{n}, x_{n+1}]$ for all $n<\omega$
it follows that $1+\gamma= 1+\Sigma_{n<\omega}(\alpha_{n}+1)$
embeds into $P$.

\subsection{ Property $(p_{4})$}
The fact that $\omega^*\in \K$ is a consequence of  Theorem 1.3
\cite{chapou2}. We use similar ingredients for Property
$(p_{4})$. Given a join-semilattice $P$, $x\in P$ and $J\in J(P)$
we denote $\downarrow x \bigvee J$, as well as $\{x\}\bigvee J$,
the join - in $J(P)$- of $\downarrow x$ and $J$. We say that   a
non-empty chain $\mathcal I$ of ideals of $P$ is {\it separating}
\index{separating chain} if for every $I\in\mathcal
I\setminus\{\cup\mathcal I\}$, every $x\in \cup\mathcal I\setminus
I$, there is some $J\in\mathcal I$ such that
$I\not\subseteq\{x\}\bigvee J$. We recall the following fact
(Lemma 3.1 \cite{chapou2}). For reader 's convenience we give the
proof.
\begin{lemma}\label{independent}
A join-semilattice $P$ contains an infinite independent set if and
only if it contains an infinite separating chain of ideals.
\end{lemma}
\begin{proof}
Let $\mathcal I$ be a separating chain of ideals. Define
inductively an infinite sequence $x_{0}, I_{0}, \ldots,
x_{n},I_{n}, \ldots$ such that $I_{0}\in\mathcal
I\setminus\{\cup\mathcal I\}, x_{0}\in \cup\mathcal I \setminus
I_{0}$ and
 such that:\\
 $a_{n})$
$I_{n}\in\mathcal I$;\\ $b_{n})$ $I_{n}\subset I_{n-1}$;\\
 $c_{n})$
$x_{n}\in I_{n-1}\setminus(\{x_{0}\vee \ldots \vee
x_{n-1}\}\bigvee I_{n})$ for every  $n\geq 1$.\\ The construction
is immediate. Indeed, since $\mathcal I$ is infinite then
$\mathcal I\setminus\{\cup\mathcal I\}\not =\emptyset$. Choose
arbitrary $I_{0}\in\mathcal I\setminus\{\cup\mathcal I\}$ and
$x_{0}\in\cup\mathcal I\setminus I_{0}$. Let $n\geq 1$. Suppose
$x_{k}, I_{k}$ defined and satisfying $a_{k}), b_{k}), c_{k})$ for
all $k\leq n-1$. Set $I:=I_{n-1}$ and $x:=x_{0}\vee \ldots \vee
x_{n-1}$. Since $I\in\mathcal I$ and $x\in\cup\mathcal I\setminus
I$, there is some $J\in\mathcal I$ such that
$I\not\subseteq\{x\}\bigvee J$. Let $z\in I\setminus(\{x\}\bigvee
J)$. Set $x_{n}:=z$, $I_{n}:=J$. The set $X:=\{x_{n}: n<\omega\}$
is independent. Indeed if $x\in X$ then since $x=x_{n}$ for some
$n$, $n<\omega$, condition $c_{n})$ asserts that there is some
ideal containing $X\setminus \{x\}$ and excluding $x$.
\end{proof}\\

\begin{lemma}\label{lem III.5.1}
    Let $\alpha$  be a countable order type and $P\in\mathbb J_{\alpha}$ satisfying the following conditions:\\
$1)$ For every $x$ in $P$, the chain $I(\alpha)$ does not embed
into $J(\downarrow x)$;\\ $2)$ For every $a\in \alpha$, the order
type $\alpha$ embeds into $\downarrow a$;\\ $3)$ The lattice
$L_{J(\beta)}$ does not embed into $J(P)$ as a sublattice for
$\beta:=\alpha$ if $\alpha=\eta$ and $\beta\in \{\omega,
\omega^*\}$ otherwise. \\Then
 $P$ contains an infinite independent set.
\end{lemma}
\begin{proof}
Let $\mathcal C\subseteq J(P)$  be a chain isomorphic to
$I(\alpha)$ and $\mathcal C ^*:= \mathcal C \setminus \{I_{*}\}$,
where $I_{*}$  is the least element of  $\mathcal C$. According to
Lemma \ref{independent} it suffices to prove that  $\mathcal C^*$
is separating. Suppose it is not.  Let $I_{0}\in \mathcal
C^*\setminus  \bigcup\mathcal C^*$ and $x_{0}\in\bigcup\mathcal
C^*\setminus I_{0}$ witnessing it. Set $ \mathcal C^*_{0}:=\{I\in
\mathcal C^*: I \subseteq I_{0}\}$. For every $I \in \mathcal
C^*_{0}$ we have $(\downarrow x_{0})\bigvee I=(\downarrow
x_{0})\bigvee I_{0}$.  Since $I_{0}$ is non-empty then,  from
condition $2)$,
 $ \mathcal C^*_{0}$ contains
a subset $\mathcal A$
 isomorphic to
$J(\alpha)$. Let $\mathcal B$ be the image  of  $\mathcal A \cup
\{I_{*}\}$ by the  map $\phi : J(P) \rightarrow J(\downarrow
x_{0})$ defined by $\phi(I) :=I\cap (\downarrow x_{0})$ for every
$I\in J(P)$.  The chain $\mathcal A \cup \{I_{*}\}$ is isomorphic
to $I(\alpha)$, whereas its image $\mathcal B$ does not embed
$I(\alpha)$  because of condition $1)$. Since $\alpha$ is
countable and, from condition $2)$, indecomposable,
it follows that $\phi$ is constant on a subset $\mathcal D$ of
$\mathcal A$ which is isomorphic to $J(\beta)$, with
$\beta:=\alpha$ if $\alpha=\eta$ and $\beta\in \{\omega, \omega^*
\}$ otherwise(if $\alpha=\eta$ then $\phi$ cannot be one-to-one on
a subset of $\mathcal A$ isomorphic to $\alpha$, whereas if
$\alpha$ is scattered then, being indecomposable, it is
equimorphic to
 $J(\alpha)$  and, as it is easy to see, a map from $\alpha$ on a strictly smaller order type is constant on an infinite subset, hence on a subset of type $\omega$ or
$\omega^*$). The sublattice of $J(P)$ generated by $\downarrow x $
and $\mathcal D$ is isomorphic to $L_{J(\beta)}$.
\end{proof}\\
This, added to a very simple trick, gives $\omega^*\in \mathbb K$
and, more precisely:
\begin{theorem} If $L_{\omega^*}$ does not embed, as a lattice, into the lattice $J(P)$ of ideals of a join-semilattice $P$, then $J(P)$
 is well-founded if and only if $P$ is well-founded and has no infinite independent set.
\end{theorem}
\begin{proof}
Suppose $P$ well-founded but $J(P)$ not well-founded. Then $P$
contains an ideal $P'$ such that $J(P')$ is not well-founded but
$J(\downarrow x)$ is well-founded for every $x\in P'$.  Indeed, if
$P$ itself is not suitable, select $x$ minimal in $P$ such that
$J(\downarrow x)$ is not well-founded. Every  $P'\in J(\downarrow
x)$ such that $J(P')$ is not well-founded will do. From Lemma
\ref{lem III.5.1} above, $P'$ contains an  infinite independent
set, hence $P$ contains an infinite independent set. \end{proof}\\
In the next lemma, we extend the scope of the trick used above, in
the same vein as in \cite{pz}, Lemma 3.4.7.
\begin{lemma}\label{lem III.5.2}
Let $\alpha:=\Sigma^*_{n<\omega} \alpha_{n}$ be an  $\omega^*$-sum
of non-zero order types which is left-indecomposable. If $P \in
\mathbb J_{\alpha}$ then either:\\ $(i)$  $P$ contains some ideal
$P'\in \mathbb J_{\alpha}$ such that for every $x\in P'$, the
chain $I(\alpha)$ does not embed into $J(\downarrow x)$;\\ or\\
$(ii)$ $P$ contains an $\omega^*$-sum  $P'':= \Sigma^*_{n<\omega}
P_{n}$ where $P_{n}\in \mathbb J_{\alpha_{n}}$ is a convex
\index{convex} subset of $P$  for every $n<\omega$.
\end{lemma}
\begin{proof}
Set $E:=\{x\in P: I (\alpha)\leq J(\downarrow x)\}$.
\\
{\bf Case 1}. There is some  $J\in \mathbb J_{\alpha}\cap J(P)$
such that $E\cap J=\emptyset$. Set $P':=J$. \\ {\bf Case 2}.
$E\cap J\not =\emptyset$ for every  $J\in \mathbb J_{\alpha}\cap
J(P)$.  \\ {\bf Claim. } For every $\beta$,  $1+\beta\leq \alpha$,
and every $J\in \mathbb J_{\alpha}\cap J(P)$ there is some $x\in
E\cap J$ such that $I(\beta)\leq J(J\cap\uparrow x)$.\\ {\bf Proof
of the Claim.}
 Let $\mathcal C\subseteq J(J)$  be a chain isomorphic to $I(\alpha)$ and
$I_{*}$  be the least element of  $\mathcal C$. There is some
$J'\in J(J)\setminus\{J_{*}\}$ such that $I(\beta)\leq \{J''\in
J(J): J'\subseteq J''\}$. Since $\alpha$ is left-indecomposable,
$I(\alpha)\leq J(J')$. Since we are in Case  2, $E\cap J'\not =
\emptyset $.  Let $x\in E\cap J'$. From Lemma \ref {final} the
poset   $J(J\cap\uparrow x)$ is isomorphic to
 $\{J''\in J(J): x\in J''\}$ which trivially embeds  $I(\beta)$.  This proves  our claim. \endproof

With the help of this claim, we construct   a sequence $$J_{0},
x_{0}, J_{1}, x_{1}, \ldots, J_{n}, x_{n},\ldots$$ such that
$J_{n}\in J(P)$, $x_{n}\in E\cap J_{n}$, $I(\alpha_{n})\leq
J(J_{n}\cap \uparrow x_{n})$
 and $J_{n+1}\subseteq\downarrow x_{n}$ for every $n<\omega$. Indeed,
set $J_{0}:=P$, choose $x_{0}\in E$ such that  $I(\alpha_{0})\leq
J(\uparrow x_{0})$. Suppose $J_{0}, x_{0}, \ldots, J_{n}, x_{n}$
constructed. Since $x_{n}\in E\cap  J_{n}$, we may choose
$J_{n+1}\in J(\downarrow x_{n})$ such that $I(\alpha)\leq
J(J_{n+1})$. According to the claim above,  we may choose
$x_{n+1}\in E\cap J_{n+1}$ such that $I(\alpha_{n+1})\leq
J(J_{n+1}\cap \uparrow x_{n+1})$.\\ To conclude,  set
$P_{n}:=J_{n}\cap \uparrow x_{n}$  and $P'':= \Sigma^*_{n<\omega}
P_{n}$.
\end{proof}
\subsubsection{Proof of  Property $(p_{4})$}
    Let $\alpha:=\Sigma^*_{n<\omega} \alpha_{n}$ satisfying the conditions of
$(p_{4})$ and let $P\in \mathbb L_{\alpha}$. We need to show that
either $P$
 contains a  subchain isomorphic to
$1+\alpha$ or an infinite independent set. Since for every integer
$n$, the set $\{m : \alpha_{n}\leq \alpha_{m}\}$ is infinite, the
order type $\alpha$ is left-indecomposable and we may apply Lemma
\ref {lem III.5.2}.\\ {\bf Case 1.}  $P$ contains some ideal
$P'\in \mathbb J_{\alpha}$ such that for every $x\in P'$, the
chain $I(\alpha)$ does not embed into $J(\downarrow x)$.\\ This
says that  $P'$ satisfies condition $1)$ of Lemma  \ref {lem
III.5.1}. Since $P\in \mathbb L_{\alpha}$, then neither
$L_{\omega+1}$ nor $L_{\omega^*} $  embeds into $J(P)$ as a
sublattice.  The same holds for $J(P')$, hence $P'$ satisfies
condition $3)$ of Lemma \ref {lem III.5.1}. Since $\alpha$ is
left-indecomposable,
 condition $2)$ of this lemma is satisfied too. It follows then that $P'$, thus $P$, contains an infinite independent set.\\
{\bf Case 2.} $P$ contains an $\omega^*$-sum  $P'':=
\Sigma^*_{n<\omega} P_{n}$ where $P_{n}\in \mathbb J_{\alpha_{n}}$
is a convex subset of $P$  for every $n<\omega$. \\
 Let $n<\omega$.  Clearly $P_{n}\in \mathbb L$, hence $P_{n}\in \mathbb L_{\alpha_{n}}$. Since $\alpha_{n}\in  \mathbb K$ then $P_{n}$ contains either
a subchain isomorphic to  $1+\alpha_{n}$ or an infinite
independent set. If some $P_{n}$ contains an infinite independent
set, such a set is independent in $P$ too. If not, then each
$P_{n}$ contains a subchain isomorphic to $1+\alpha_{n}$, hence
$P''$ contains a  chain  of type $\Sigma^*_{n<\omega}
(1+\alpha_{n})$. Since  $P$ has a least element,  it contains a
chain of type $1+\alpha$. \endproof

\subsection{Property $(p_{6})$}

\begin{lemma}\label{lem III.6.1}
    Let $P\in \mathbb J_ {\eta}$ then
$P$ contains either a copy of $\eta$ or a convex subset $P'\in
\mathbb J_\eta$ such that $\downarrow x\cap P'\not \in \mathbb J_
{\eta}$, for every $x\in P'$.
\end{lemma}
\begin{proof}
Set $P\in \mathbb J^*_{\eta}$  if   $P\in \mathbb J_ {\eta}$ and
the second part of the above assertion does not hold. \\ {\bf
Claim} If  $P\in \mathbb J^*_{\eta}$  then $P$ contains an element
$x$ such that  $P':=\downarrow x$ and $P'':=\uparrow x$ belong to
$\mathbb J^*_ {\eta}$.\\
 {\bf Proof of the claim }
let $\mathcal C\subseteq J(P)$
 be a
chain isomorphic to $I(\eta)$. Let $I\in \mathcal C\setminus X$
where $X: =Min (\mathcal C)\cup Max(\mathcal C)$. The set $I$ is
convex and belongs to $\mathbb J_ {\eta}$. Since $P \in \mathbb
J^*_ {\eta}$
 there
is $x\in I$ such that $I(\eta)\leq J(\downarrow x)$. Put
$P':=\downarrow x$ and $P'': =\uparrow x$.  From our choice,
$P'\in\mathbb J_ {\eta}$.   Since $x\in I\not \in  Max(\mathcal
C)$, we have $I(\eta)\leq J_{x}(P)$, hence from Lemma \ref
{final}, $I(\eta)\leq J(P'')$, that is  $P''\in\mathbb J_ {\eta}$.
If $Q\in\mathbb J_ {\eta}$ is a convex subset of $P'$ or $P''$,
this is  a convex subset  of $P$ too, hence it contains some $y$
such that $\downarrow y\cap Q \in \mathbb J_ {\eta}$, proving that
$P', P'' \in \mathbb J^*_ {\eta}$, as required. \endproof

From this, it  follows easily that if $P \in \mathbb J^*_ {\eta}$,
the chain of the dyadic numbers \index{dyadic number} of the
interval $]0,1[$ embeds into $P$.  Indeed, asociate to each dyadic
$d$ a convex subset $P_{d} \in \mathbb J^*_{\eta}$ and an element
$x_{d}\in P_{d}$ as follows. Put  $P_{1/2}:=P$. Let
$d:=\frac{2m+1}{2n}$, denote $d':=\frac{4m+1}{2n+1}$,
$d'':=\frac{4m+3}{2n+1}$. Suppose $P_{d}$ defined. From the claim
above, get an element, denoted $x_{d}$. Put $P_{d'}:=\downarrow
x_{d}\cap P_{d}$, $P_{d''}:=\uparrow x_{d}\cap P_{d}$. The
correspondance $d\rightarrow x_{d}$ is an embedding. Thus
$\eta\leq P$.
\end{proof}
\subsubsection{Proof of  Property $(p_{6})$}
    Let $P\in \mathbb L_{\eta}$ . If
$\eta\not\leq P$ then, from  Lemma \ref{lem III.6.1}, $P$ contains
a convex subset $P'\in \mathbb J_{\eta}$, such that Condition $1)$
of Lemma \ref{lem III.5.1} is verified for $\alpha:=\eta$.
Condition $2)$ is trivially verified.  Since $P'$ is convex in
$P$, then $P'\in \mathbb L$, hence $L(\eta)$ does not embed in
$J(P')$ as a sublattice. Thus Condition $3)$ is verified too.
Applying Lemma  \ref{lem III.5.1}, we get that $P'$, hence $P$,
contains an infinite independent set. Thus $\eta\in\mathbb K$.

\subsection{Property $(p_{5})$}
\begin{lemma} \label {v1}
    Let $\delta$, $\beta$, $\mu$  be three order types and $\alpha:=\delta+\beta+\mu$. If
$\alpha\in \mathbb{K}$ and if, for every order type $\beta'$, the
order type $\beta$ embeds into $\beta'$ whenever $\alpha$ embeds
into $\alpha':=\delta+\beta'+\mu$, then $\beta\in  \mathbb{K}$.
\end{lemma}
\begin{proof}
Let $\beta_{*}$ obtained from $\beta$ by taking out its first
element, if any. We prove $\beta_{*}\in \mathbb{K}$. Properties
$(p_{1})$ and $(p_{2})$ insure $\beta \in \mathbb{K}$. Let $P\in
\mathbb L_{\beta_{*}}$. Suppose  $P$ contains no infinite
independent set.  Let $Q:= 1+ \delta+P+\mu$. Then $Q$ is a
join-semilattice with $0$ and with no infinite independent set.
Moreover, $Q\in \mathbb L_{\alpha}$. Indeed, from $J(Q)\cong
1+J(\delta)+J(P)+J(\mu)$ , we get first $I(\alpha)\leq Q$ that is
$Q\in \mathbb J_{\alpha}$  (  $I(\alpha)\cong 1+J(\alpha)\cong
1+J(\delta)+J(\beta)+J(\mu) \leq 1+
J(\delta)+I(\beta_{*})+J(\mu)\leq 1+J(\delta)+J(P)+J(\mu)\cong
J(Q)$); next,  we get that $L(\lambda)\leq J( Q)$ implies
$L(\lambda)\leq J( P)$  for every order-type $\lambda\not = 0 $,
from which $Q\in \mathbb L_{\alpha}$  follows. Since $\alpha\in
\mathbb{K}$ we have $1+ \alpha\leq Q$ . Let $\varphi$ be an order
isomorphism from $1+ \alpha$ into $Q $ ($=1+\delta+P+\mu$) and let
$\beta':=\varphi(\alpha)\cap P$. The order type $\alpha$ embeds
into $\delta+\beta'+\mu$ hence, from our hypothesis, $\beta$
embeds into $\beta'$. Since $\beta'$ is a subchain of $P$, $\beta$
embeds into $P$. Since $P$ has a $0$, this implies that
$1+\beta_{*}$ embeds into $P$, hence $\beta_{*}\in \mathbb{K}$ as
required.
\end{proof}
\begin{lemma} \label{v2}
    Let $\alpha\in \mathbb{K}$. If $\alpha$ has canonical form
$\alpha=\alpha_{0}+\alpha_{1}+\ldots +\alpha_{n}$  then
$\alpha_{p}+\alpha_{p+1}+\ldots +\alpha_{q}\in \mathbb{K}$ for
$0\leq p\leq q\leq n$.
\end{lemma}
\begin{proof}
Let $p,q$ so that $0\leq p\leq q\leq n$. Set $\delta:
=\alpha_{0}+\alpha_{1}+\ldots+\alpha_{p-1}$,
$\beta:=\alpha_{p}+\alpha_{p+1}+\ldots+\alpha_{q}$ and
$\mu:=\alpha_{q+1}+\ldots+\alpha_{n}$. Let $\beta'$ be an order
type such that $\alpha$ embeds into $\alpha':=\delta+\beta'+\mu$.
We prove $\beta \leq \beta'$. Lemma \ref {v1} insures $\beta \in
\mathbb{K}$. Let $A$ and $A'$ be two chains having order type
$\alpha$ and $\alpha'$ respectively; let us consider a partition
of $A$ in three intervals $D$, $B$, $M$ of types $\delta$,
$\beta$, $\mu$ respectively, such that $A=D+B+M$ and similarly
consider a  partition of $A'$ in three intervals $D'$, $B'$, $M'$
of  types $\delta$, $\beta'$, $\mu$ respectively, such that
$A'=D'+B'+M'$. Let $\varphi$ be an embedding from $A$ into $A'$
and $B_{1}:=\varphi^{-1}(B')$. We show that $\beta$ embeds into
$B_{1}$, which implies that $\beta$ embeds into $\beta'$. Indeed,
since the decomposition of $\alpha$ is canonical,
$\alpha_{0}+\alpha_{1}+\dots+\alpha_{p}\not\leq\alpha_{0}+\alpha_{1}+\dots+\alpha_{p-1}$,
hence $\alpha_{p}\not\leq \varphi^{-1}(D')\cap B$.  Similarly
$\alpha_{q}\not\leq \varphi^{-1}(M')\cap B$. But $\alpha_{p}$ and
$\alpha_{q}$ are indecomposable, hence if we decompose $B$ into
intervals $A_{i}$ of type $\alpha_{i}$ such that
$B=A_{p}+...+A_{q}$ we obtain $\alpha_{p}\leq A_{p}\setminus
\varphi^{-1}(D')$ and $\alpha_{q}\leq A_{q}\setminus
\varphi^{-1}(M')$. It follows that
$\beta=\alpha_{p}+\ldots+\alpha_{q}$ embeds into $B_{1}$.
\end{proof}
\begin{lemma}\label {v3}
    Let $\gamma:=\alpha +\beta$ be a scattered order-type, where $\alpha$  is infinite and strictly right-indecomposable,
$\beta$ indecomposable, and $\gamma\not\leq\beta$. Then
$\gamma\not \in \mathbb{K}$.
\end{lemma}
\begin{proof}
We construct a distributive lattice $T$, direct product of two
chains, one with type $1+\alpha_{0}$ for some $\alpha_{0}\leq
\alpha$, the other of type $1+\beta$, such that  $a)$  $T$
contains no infinite independent set; $b)$ $I(\gamma)\leq J(T)$;
$c)$ $1+\gamma\not\leq T$. \\ Since $\alpha$ is infinite and
strictly right-indecomposable then $1+\alpha \leq \alpha$, hence
$1+\gamma\leq \gamma$. Let $C$ be a chain of type $1+\gamma$ and
$A$, $B$ be a partition of $C$ into an initial segment $A$ of type
$1+\alpha$ and a
 final segment $B$ of type $\beta$. Let  $1+B$ the chain obtained by
adding to $B$ a least element $b_{0}\not\in B$ and let
$Q:=A\times(1+B)$. We distinguish two cases.\\ {\bf Case 1}
$\beta$ is strictly left-indecomposable. We set $T:=Q$ . We check
that $T$ satisfies  properties $a)$, $b)$, $c)$ above. For
property  $a)$  this is trivial: as a product of two chains, $T$
contains no independent set with three elements. For $b)$ let
$\phi: J(C)\rightarrow J(T)$ defined  by
$\phi(I):=I\times\{b_{0}\}$  if $I\subseteq A$ and
$\phi(I):=A\times((I\setminus A)\cup\{b_{0}\})$ otherwise. This is
an embedding, hence $I(\gamma) \leq J(T)$. Finally for $c)$,
suppose for contradiction $1+\gamma\leq T$. Let $\varphi$ be an
embedding from $C$ into $T$. Let  $b\in B$ ,  $(x,y):=\varphi(b)$
and $\alpha_{1}$, $1+\beta_{1}$ be  the order types of $\downarrow
x$ and $\downarrow y$ respectively; we have   $\alpha \leq
\alpha_{1}\times (1+\beta_{1})$. Since $\alpha$ is infinite, $(i)$
of Lemma \ref {lemI.4.4} applies: among the indivisible order
types which embed into $\alpha$, there is a maximal one. Let $\nu$
be such an order type. {\bf Claim 1.} $\nu\leq\beta$.  First,
since $\nu$ is indivisible  and $\nu\leq\alpha_{1}\times
(1+\beta_{1})$ then, from Lemma \ref {propI.3.7}, either
$\nu\leq\alpha_{1}$ or $\nu\leq 1+\beta_{1}$ .  Since $ \alpha$ is
infinite and strictly right-indecomposable then $\alpha_{1}+1\leq
\alpha$.  Thus,  if  $\nu \leq \alpha_{1}$ then $\nu +1 \leq
\alpha$  which is impossible from  $(ii)$ of Lemma \ref
{lemI.4.4}. Hence,  $\nu\not \leq \alpha_{1}$. Thus $\nu\leq
1+\beta_{1}$. Since $\nu$ is infinite then $\nu\leq \beta$,
proving our claim. Now, from this claim, $\beta$ is infinite and
we can repeat what we did with $\alpha$.  Let $\xi$ be maximal
among the indivisible order types embedding $\nu$ and  which embed
into $\beta$. The same arguments as above show {\bf Claim 2.} $\xi
\leq \alpha$. From these two claims, we get   $\nu\equiv \xi$
which is impossible since, from $(ii)$ of Lemma   \ref {lemI.4.4},
$\nu$ and $\xi$ must be respectively strictly right and strictly
left-indecomposable. So $1+\gamma\not\leq T$.\\ {\bf Case 2}
$\beta$ strictly right-indecomposable. Because of Case  $1$, we
may suppose $\beta\not=1$. Let  $\mathcal{A}$ be the set of order
types $ \alpha', \alpha' \leq\alpha$,  such that $\alpha \leq
\alpha'\times(1+\beta')$ for some $\beta'$ such that
$\beta'+1\leq\beta$. This set is non empty (it contains $\alpha$).
Hence, it contains a minimal element w.r.t. embeddability (Theorem
\ref {tmI.4.1} (iii)). Let $\alpha_{0}$ be such an element and
$A_{0}$ be a subchain of $A$ of type $1+\alpha_{0}$. Let
$T:=A_{0}\times (1+B)$. We check that T satisfies the properties
$a)$, $b)$, $c)$ above. This is trivial for $a)$ and for the same
reasons that in Case $1$. For $b)$ observe that  since
$\alpha_{0}\in \mathcal {A}$ there is an element $b_{0}\in B$ and
an embedding $\varphi$ from $A$ into $A_{0}\times\{y\in 1+B :
y<b_{0}\}$. Since $\beta$ is strictly right-indecomposable, there
is an embedding $\theta$ from $B$ into $\{y\in1+B : y\geq
b_{0}\}$. Let $\phi$: $J(C)\rightarrow J(T)$ defined by
$\phi(I):=\{z\in T : z\leq\varphi(x)$ for some $x\in I\}$ if $I
\subseteq A$ and $\phi(I):=A_{0}\times\{y\in 1+B : y\leq\theta(x)$
for some $x\in I\}$ otherwise. This is an embedding, hence
$I(\gamma) \leq J(T)$. For $c)$ we prove  $\alpha+1\not\leq T$.
For a contradiction, suppose  $\alpha+1\leq T$. Let $A+1$ be the
chain obtained by adding to $A$ a largest element $a\not\in A$.
Since $1+\alpha \leq \alpha$ there is an  embedding $\psi$ from
$A+1$ into $T$. Let  $(x,y):=\psi(a)$. Then $\psi$ is an embedding
from $A+1$ into $\downarrow x\times\downarrow y$. Let
$\alpha_{1}:=\alpha'_{1}+1$ be the type of $\downarrow x$ and
$1+\beta_{1}$ the type of $\downarrow y$. We have $\alpha _{1}\leq
1+\alpha_{0}$ and  $1+\alpha +1\leq \alpha_{1}\times (
1+\beta_{1})$. Since $1+\alpha\leq \alpha$ then $\alpha_{1}\leq
\alpha$ and since $\beta$ is strictly right-indecomposable and
distinct from $1$, we have $\beta_{1}+1\leq\beta$. Hence
$\alpha_{1}\in \mathcal{A}$. {\bf Claim 3.} $\alpha_{0}$ is
strictly right-indecomposable and distinct from $1$. Indeed, if
$\alpha_{0}=1$, then $\alpha\leq 1+\beta'_{1}$ for some
$\beta'_{1}$ such that $\beta'_{1}+1\leq \beta$. Since $1+\alpha
\leq \alpha$, we get $\alpha\leq\beta'_{1}$ and since $\beta$ is
strictly right-indecomposable, $\alpha+\beta\leq\beta$, that is
$\gamma\leq \beta$,  which contradicts our
 assumption. Next,
suppose  $\alpha_{0}=\alpha'_{0}+\alpha''_{0}$ with
$\alpha''_{0}\not=0$. Since $\alpha_{0}\in \mathcal{A}$ there is
some $\beta'$ such that $\beta'+1\leq\beta$ and  $\alpha
\leq\alpha_{0}\times(1+\beta')$. Moreover, since $\alpha_{0}$ is
minimal, we may suppose that for some embedding $\varphi_{0}:
\alpha \rightarrow \alpha_{0}\times(1+\beta')$ the projection on
$\alpha_{0}$ is surjective. Since $\alpha_{0}\times(1+\beta')$
decomposes  into an initial segment of type
$\alpha'_{0}\times(1+\beta')$ and a final segment  of type
$\alpha''_{0}\times(1+\beta')$, then $\varphi_{0}$ divides
$\alpha$ into $\alpha'$ and $\alpha''$ in such a way that
$\alpha'\leq\alpha'_{0}\times(1+\beta')$ and
$\alpha''\leq\alpha''_{0}\times(1+\beta')$. Since $\alpha$ is
strictly right- indecomposable and $\alpha''\not=0$ then
$\alpha\leq\alpha''$, hence $\alpha''_{0}\in \mathcal{A}$. From
the minimality of $\alpha_{0}$  follows
$\alpha_{0}\leq\alpha''_{0}$, hence $\alpha_{0}$ is strictly
right-indecomposable as claimed. From this claim, the inequality
$\alpha_{1}:=\alpha'_{1}+1\leq 1+ \alpha_{0}$ implies
$\alpha_{1}<\alpha_{0}$. Since $\alpha_{1} \in \mathcal{A}$, this
contradicts the minimality of $\alpha_{0}$.
\end{proof}
\subsubsection{Proof of Property  $(p_{5})$}
    Let $\alpha$ be a countable scattered order type. Suppose $\alpha \in \mathbb{K}$.
From Corollary \ref {corI.4.2}, $\alpha$ has  decomposition into
finitely indecomposables order types, hence a canonical
decomposition
 $\alpha=\alpha_{0}+\alpha_{1}+\ldots
+\alpha_{n}$. From Lemma \ref {v2}, $\alpha_{i}\in \mathbb{K}$ for
all $i\leq n$. Let $i<n$. If $\alpha_{i}$ is not strictly
left-indecomposable then it is infinite and strictly
right-indecomposable. Since the decomposition is canonical,
$\alpha_{i}+\alpha_{i+1}$  embeds neither into $\alpha_{i}$ nor
into $\alpha_{i+1}$. From Lemma \ref {v3},
$\alpha_{i}+\alpha_{i+1}\not \in \mathbb{K}$. This contradicts
Lemma \ref {v2}. Thus, $\alpha$ has the  form given in Property
$(p_5)$. Conversely, suppose $\alpha=\alpha_{0}+\alpha_{1}+\ldots
+\alpha_{n}$ with $\alpha_{i}\in \mathbb{K}$, for all $i\leq n$,
$\alpha_{i}$ strictly left-indecomposable for $i<n$ and
$\alpha_{n}$ indecomposable. If $n=0$ then $\alpha=\alpha_{0}\in
\mathbb{K}$. If $n>0$, suppose $n$ minimal. Then
$\alpha_{i}\not=0$ for all $i<n$, hence $\alpha_{i}\equiv
\alpha'_{i}+1$ and $\alpha\equiv
\alpha'=\alpha'_{0}+1+\alpha'_{1}+1+\ldots
+\alpha'_{n-1}+1+\alpha_{n}$. From Property $(p_{2})$, $\alpha'\in
\mathbb{K}$. That is $\alpha\in \mathbb{K}$.

\section[The class $\mathbb{P}$]{The class $\mathbb{P}$ and a proof of Theorem \ref{thm7}}
Define for each ordinal $\xi$ the class $\mathbb{P}_{\xi}$,
setting  $\mathbb{P}_{0}:=\{0,1\}$,
$\mathbb{P}_{\xi}:=\bigcup\{\mathbb{P}_{\xi'} : \xi'<\xi\}$ if
$\xi$ is a limit ordinal, $\mathbb{P}_{\xi+1}$ equals to the set
of order types $\gamma$ which decompose into elements of
$\mathbb{P}_{\xi}$ under one of
 the forms stated in $(p_{2})$, $(p_{3})$ and $(p_{4})$ of the definition of $\mathbb{P}$ (cf. Section 1). For an example,
$\mathbb{P}_1=\mathbb{P}_{0}\cup \{2, \omega, \omega^*\}$,
$\mathbb{P}_2=\mathbb{P}_{1}\cup \{3, 4, 1+\omega^{*},
2+\omega^{*},  \omega^{*}2, \omega^{*^{2}}, \omega^{*}n+\omega,
n+\omega^{*}\omega, \omega\omega^{*}+n$ with $n<\omega\}\cup \{
\gamma: \gamma\equiv (\omega^{*}+\omega)\omega^{*} \}$. Since
$\mathbb{P}_{0}\subseteq \mathbb{P}_{1}$ the sequence of
$\mathbb{P}_{\xi}$ is non decreasing and moreover
$\mathbb{P}=\mathbb{P}_{\omega_{1}}$. Note that $\mathbb{P}$ is
not  preserved under equimorphy  (indeed
$\omega\omega^{*}\omega\not\in \mathbb{P}$ while
$(\omega\omega^{*}+1)\omega\in \mathbb{P}_3$). Let $\gamma\in
\mathbb{P}$; the {\it $\mathbb{P}$-rank} of $\gamma$, denoted
$rank_{\mathbb{P}}(\gamma)$, is the least $\xi$ such that
$\gamma\in \mathbb{P}_{\xi}$.  This notion of rank allows to prove
properties of the elements of $\mathbb{P}$ by induction on their
rank, especially those given just below.
\\
{\bf Claim 1}\\ Let $\gamma$ be a countable order type;\\ $1.1)$
If $\gamma\in \mathbb{P}$ then $\delta\in \mathbb{P}$ for every
initial segment $\delta$ of $\gamma$;\\ $1.2)$ $\gamma\in
\mathbb{P}$ if and only if $\gamma$ decomposes into  a sum
$\gamma=\gamma_{0}+\dots+\gamma_{n}$ of members of $\mathbb{P}$,
each $\gamma_{i}$ ($i<n$) is strictly left-indecomposable and has
a largest element, $\gamma_{n}$ is indecomposable;\\ $1.3)$
Suppose $\gamma\in \mathbb{P}_{\xi+1}\setminus \{0,1\}$
($\xi<\omega_{1}$).  If $\gamma$ is strictly right-indecomposable
then $\gamma=\Sigma_{n<\omega}(\alpha_{n}+1)$ with
$\alpha_{n}+1\in \mathbb{P}_{\xi}$;  if $\gamma$ is strictly
left-indecomposable then $\gamma=\Sigma^*_{n<\omega} \alpha_{n}$
with $\alpha_{n}\in \mathbb{P}_{\xi}$, the $\alpha_{n}$'s forming
a quasi-monotonic sequence.\\

\noindent {\bf Claim 2}\\ The class of indecomposable members  of
$\mathbb{P}$  is the smallest class $Ind_{\mathbb{P}}$ of order
types such that:\\
        $2.1)$ $0,1\in Ind_{\mathbb{P}}$;\\
        $2.2)$ If $(\alpha_{n})_{n<\omega}$ is a quasi-monotonic sequence of elements of
$Ind_{\mathbb{P}}$ then $\Sigma^*_{n<\omega} \alpha_{n} \in
Ind_{\mathbb{P}}$; if, moreover, each $\alpha_{n}$ has a largest
element then $\Sigma_{n<\omega} \alpha_{n} \in Ind_{\mathbb{P}}$.

        From Claim 2 and the fact that every countable indivisible scattered order-type is an $\omega$-sum or an
$\omega^*$-sum of smaller indivisible order-types, we deduce:\\
 {\bf Claim 3}\\
If $\gamma\in \mathbb{P}$ then each  maximal indivisible
order-type which  embeds into $\gamma$ is equimorphic to some
element of $\mathbb{P}$.\\

  \begin{lemma}\label{lem III.8.1}
        Let $\alpha$ be a countable, scattered and indecomposable, order type.\\
$(i)$ The set of $\beta\in \mathbb{P}$ which embed into $\alpha$
has a largest element (up to equimorphy) which we denote
$\mathbb{P}(\alpha)$;\\ $(ii)$ If
$\alpha\equiv\Sigma^{*}_{n<\omega} \alpha_{n}$ with each
$\alpha_{n}$ indecomposable and verifying $0<\alpha_{n}<\alpha$
then
$\mathbb{P}(\alpha)\equiv\Sigma^{*}_{n<\omega}\mathbb{P}(\alpha_{n})$.\\
$(iii)$ If $\alpha\equiv \Sigma_{\lambda<\mu} \alpha_{\lambda}$
where $\mu$ is an indecomposable ordinal, each $\alpha_{\lambda}$
is strictly left-indecomposable and verifies
$0<\alpha_{\lambda}<\alpha$ then
$\mathbb{P}(\alpha_{\lambda})<\mathbb{P}(\alpha)$ for each
$\lambda<\mu$ . Moreover, there is a sequence
$\lambda_{0}<\lambda_{1}<\ldots<\lambda_{n}<\ldots$ such that
$Sup\{\lambda_{n} : n<\omega\}=\mu$ and
$\mathbb{P}(\alpha)\equiv\Sigma_{n<\omega}\mathbb{P}(\alpha_{\lambda_{n}})$.
\end{lemma}
\begin{proof}
$(i)$ We may suppose $\alpha\not=0, 1$.\\ {\bf Case1} $\alpha$ is
strictly right-indecomposable. Let $\mathcal{A}:=\{\gamma\in
\mathbb{P}: \gamma\leq \alpha, \gamma= \beta +1$ for some $\beta
\}$. This set contains a countable subset $X$ wich is cofinal (cf.
$(iv)$ of Theorem \ref {tmI.4.1} in  Section 2). Let $\beta_{0}+1,
\beta_{1}+1 \dots\beta_{n}+1,\dots $ an enumeration  of the
elements of $X$, with possible repetitions, wich is
quasi-monotonic and let $\beta:=\beta_{0}+1+
\beta_{1}+1+\dots+\beta_{n}+1+\dots$ . Since $\mathbb{P}$
satisfies property $(p_{3})$, we have $\beta\in \mathbb{P}$. Since
$\alpha$ is strictly right-indecomposable, we have
$\beta\leq\alpha$. Let $\gamma\in \mathbb{P}$ with
$\gamma\leq\alpha$. Then, according to Claim 1.2, $\gamma$ is a
sum $\gamma_{0}+\gamma_{1}+\dots+\gamma_{n}$ with  $\gamma_{i}$
strictly left-indecomposable with a largest element for $i<n$,
$\gamma_n$ indecomposable and all $\gamma_{i}\in \mathbb{P}$. If
$\gamma_n$ is strictly left-indecomposable then, since from Claim
1.1 every initial  segment belongs to $\mathbb{P}$,   $\gamma$
embeds into a finite sum of elements of $\mathcal{A}$, while if
$\gamma_n$ is strictly right-indecomposable then, according to
Claim 1.3, it is an $\omega$-sum of elements of $\mathcal{A}$,
hence $\gamma$ too. Since $X$ is cofinal and the enumeration is
quasi-monotonic then $\gamma\leq\beta$. Thus, up-to equimorphy,
$\beta$ is the largest element of
$\mathbb{P}\cap\downarrow\alpha$.\\ {\bf Case2} $\alpha$ is
strictly left-indecomposable. We consider  the set
$\mathcal{B}:=\{\gamma\in \mathbb{P}: \gamma \leq \alpha, \gamma=
1+\beta\}$ and we proceed as in Case 1.\\
 $(ii)$ Suppose  $\alpha \equiv\sum^{*}_{n<\omega}\alpha_{n}$ with each $\alpha_{n}$
indecomposable and $0<\alpha_{n}<\alpha$. Since $\alpha$ is
indecomposable and $0<\alpha_{0}<\alpha$, then $\alpha$ is
strictly left-indecomposable. Since the $\alpha_{n}$ are
indecomposable and are distinct from  $0$ they form a
quasi-monotonic sequence. But then, the sequence of the
$\mathbb{P}(\alpha_{n})$ is quasi-monotonic too. Hence $\gamma:=
\sum^{*}_{n<\omega}\mathbb{P}(\alpha_{n})$ is indecomposable and
belongs to $\mathbb{P}$, thus $\gamma\leq \mathbb{P}(\alpha)$.
According to the proof of $(i)$ above, $\mathbb{P}(\alpha)$ is of
the form $\sum^{*}_{n<\omega}(1+\beta_{n})$ with $1+\beta_n\in
\mathbb{P}\cap\downarrow \alpha$. Let $1+\beta \in
\mathbb{P}\cap\downarrow \alpha$; according to Claim 1.2,  we may
write $1+\beta = 1+\gamma_{1}+\dots+\gamma_{m }$ with each
$\gamma_{i}$ indecomposable belonging to $\mathbb{P}$. Each
$\gamma_{i}$ embeds into some $\alpha_{n}$, hence into
$\mathbb{P}(\alpha_{n})$.  Because of the quasi-monotony of the
sequence of the $\mathbb{P}(\alpha_{n})$ it follows
$1+\beta\leq\sum^{*}_{n<\omega}\mathbb{P}(\alpha_{n}):= \gamma$.
Since $\gamma$ is strictly left-indecomposable, this inequality
applied to each $1+\beta_{n}$ leads to
$\mathbb{P}(\alpha)\leq\gamma$. Finally,
$\mathbb{P}(\alpha)\equiv\sum^{*}_{n<\omega}\mathbb{P}(\alpha_{n})$.\\
$(iii)$ Suppose $\alpha=\sum_{\lambda<\mu}\alpha_{\lambda}$ with
each $\alpha_{\lambda}$ strictly left-indecomposable verifying
$0<\alpha_{\lambda}<\alpha$. From $(ii)$ above,  if $\beta$ is a
strictly left-indecomposable order type then $\mathbb{P}(\beta)$
too; moreover, from Claim 1.1, we may suppose that
$\mathbb{P}(\beta)$  has a largest element. Hence, for every
increasing sequence $\lambda_{0}<\lambda_{1}<\dots
<\lambda_{n}<\dots$ of members of $\mu$, we have
$\sum_{n<\omega}\mathbb{P}(\alpha_{\lambda_{n}})\in\mathbb{P}\cap\downarrow\alpha$
proving $\sum_{n<\omega}\mathbb{P}(\alpha_{\lambda_{n}})\leq
\mathbb{P}(\alpha)$. Since $\alpha$ is indecomposable and
$0<\alpha_{0}<\alpha$, then $\alpha$ is strictly
right-indecomposable. According to the construction given in $(i)$
above, $\mathbb{P}(\alpha)$ is also strictly right-indecomposable
and, in fact,  is a sum $\sum_{n<\omega}(\beta_{n}+1)$ with
$\beta_{n}+1\in \mathbb{P}$. From Claim 1.2 each of these
$\beta_{n}+1$ is a finite sum of strictly left-indecomposable
members of  $\mathbb{P}$, hence $\mathbb{P}(\alpha)$ is a sum
$\sum_{n<\omega}\gamma_{n}$ with $\gamma_{n}$ strictly
left-indecomposable and belonging to $\mathbb{P}$. Each
$\gamma_{n}$ embeds into some $\alpha_{\lambda}$; indeed, since
$\gamma_{n}\leq \alpha$, let $\mu'$, $\mu'\leq \mu$, be the least
ordinal such that  $\gamma_{n}\leq
\sum_{\lambda<\mu'}\alpha_{\lambda}$. Since $\gamma_{n}$ is
strictly left-indecomposable then $\mu':=\mu''+1$  and then for
the same reason $\gamma_{n}\leq \alpha_{\mu''}$.  In fact, since
$\alpha$ is strictly right-indecomposable, $\gamma_{n}$ embeds
into cofinaly many $\alpha_{\lambda}$'s, hence into cofinaly many
$\mathbb{P}(\alpha_{\lambda})$'s. We can then  find  a cofinal
sequence sequence
$\lambda_{0}<\lambda_{1}<\ldots<\lambda_{n}<\ldots$  such that
$\gamma_{n}\leq \mathbb{P}(\alpha_{\lambda_{n}})$, proving
$\mathbb{P}(\alpha)\leq
\Sigma_{n<\omega}\mathbb{P}(\alpha_{\lambda_{n}})$. Finally, we
have $\mathbb{P}(\alpha)\equiv
\Sigma_{n<\omega}\mathbb{P}(\alpha_{\lambda_{n}})$, as
claimed.\end{proof}
\begin{proposition}
        For each countable scattered indecomposable order type $\alpha$ there is a distributive lattice $T_{\alpha}$ with $0$
such that:\\ $1)$ $ T_{\alpha}$  contains  no infinite independent
set;\\ $2) $ Each indivisible chain that embeds into $ T_{\alpha}$
embeds into $\mathbb{P}(\alpha)$;\\ $3)$ $I(\alpha)$ embeds into
$J(T_{\alpha})$.
\end{proposition}
Thus if $\alpha$ is indivisible and if $\mathbb
{P}(\alpha)<\alpha$ then $T_{\alpha}$ reveals that $\alpha\not \in
\mathbb {K}$. This proves Theorem \ref{thm7}.

\begin{proof}
To each countable scattered indecomposable order-type $\alpha$ we
associate an ordered  set $P_{\alpha}$, with a largest element if
$\alpha$ is strictly left-indecomposable, such that:\\ $1')$
$P_{\alpha}$ has no infinite antichain;\\ $2')$ $
I_{<\omega}(P_{\alpha})$ is a distributive lattice and every
indivisible chain which  embeds into $ I_{<\omega}(P_{\alpha})$
embeds into $\mathbb{P}(\alpha)$; \\ $3')$ the order on
$P_{\alpha}$ extends to a linear order $\leq_{\alpha}$ such that
the chain $\underline {P}_{\alpha}:=(P_{\alpha}, \leq_{\alpha})$
embeds $\alpha$. \\ Let  $T_{\alpha}:=I_{<\omega}(P_{\alpha})$.
Then $T_{\alpha}$  satisfies conditions $1)$, $2)$ and $3)$.
Indeed, $J(T_{\alpha})=J(I_{<\omega}(P_{\alpha}))\cong
I(P_{\alpha})$. Since $P_{\alpha}$ has no infinite antichain,
$I(P_{\alpha})$ does not embed $\mathfrak{P}(\omega)$ (Proposition
\ref {corI.1.2}), hence, as it results from Theorem \ref{tm4.1},
$T_{\alpha}$ has no infinite independent set, proving that $1)$
holds. Condition $2)$ is just Condition $2')$ in this case. Now,
since $\alpha\leq \underline{ P}_{\alpha}$ it follows that
$I(\alpha)\leq I(\underline{P}_{\alpha})\leq I(P_{\alpha})\cong
J(T_{\alpha})$, hence $3)$ holds.\\ Since the class of countable
order-types  is well-quasi-ordered, we may construct the
$P_{\alpha}$'s  by induction. Starting with
 $\alpha= 1$ we put $P_{1}:=1$.
Next, we distinguish the following cases:\\ {\bf  Case 1}:
$\alpha$ strictly left-indecomposable. In this case
$\alpha\equiv\sum^{*}_{n<\omega}\alpha_{n}$ where the
$\alpha_{n}$'s form a quasi-monotonic sequence of indecomposable
order types ($(ii)$ of Corollary \ref {corI.4.2}). We put
$P_{\alpha}:=(\sum^{*}_{n<\omega}P_{\alpha_{n}})+1$. Conditions
$1')$ and $3')$ are trivially satisfied. For $2')$ observe that
from $(ii)$ of Lemma \ref{lem III.8.1} we have
$\mathbb{P}(\alpha)\equiv \sum^{*}_{n<\omega}\mathbb{P}
(\alpha_{n})$.\\ {\bf Case2}: $\alpha$ strictly
right-indecomposable. From Lemma \ref{lemI.4.5}, $\alpha$ writes
$\sum_{\lambda<\mu}\alpha_{\lambda}$ where $\mu$ is an
indecomposable ordinal, each $\alpha_{\lambda}$ is strictly
left-indecomposable and verifies $\alpha_{\lambda}<\alpha$. If
$\mu=\omega$ we  put $P_{\alpha}:=\sum_{n<\omega}P_{\alpha_{n}}$.
In full generality, let $Q$ be the set of integers ordered by the
intersection of the natural order $\leq_{\omega}$ on $\omega$, and
the order $\leq_{\mu}$ of type $\mu$, the order $\leq_{\mu}$  been
chosen such that $0$ is the least element of $Q$. We put
$P_{\alpha}:=\sum_{\lambda\in Q}P_{\alpha_{\lambda}}$. Since the
order of $Q$ is the intersection of two well-orders, $Q$ has no
infinite antichain. Since the $P_{\alpha_{\lambda}}$'s  have no
infinite antichain, the sum $\sum_{\lambda\in
Q}P_{\alpha_{\lambda}}$ also. Hence  $1')$ holds. By construction,
the natural order on $\mu$ is a linear extension of $Q$, thus the
chain $\underline{Q}:=(Q,\leq_{\mu})$ embeds $\mu$. From the
inductive hypothesis,
 each  $P_{\alpha_{\lambda}}$ extends to a chain  $\underline{P}_{\alpha_{\lambda}}$ which embeds
$\alpha_{\lambda}$. Hence
$\alpha:=\sum_{\lambda<\mu}\alpha_{\lambda}\leq
\underline{P}_{\alpha}:=\sum_{\lambda\in \underline
Q}\underline{P}_{\alpha_{\lambda}}$ and $3')$ holds . It remains
to prove that  $2')$ holds. First, $I_{<\omega}(P_{\alpha})$ is a
lattice. It suffices  to check that if $x,y\in P_{\alpha}$ then
the initial segment $Z:=\downarrow x\cap\downarrow y$ is
finitely generated. With no loss of  generality, we may suppose
that the $P_{\alpha_{\lambda}}$'s  are pairwise  disjoint and that
$P_{\alpha}$ is the union of them. If $x$ and $y$ are in the same
$P_{\alpha_{\lambda}}$ then $Z=\downarrow(Z\cap
P_{\alpha_{\lambda}})\cup R_{\lambda}$ where
$R_{\lambda}:=\cup\{P_{\alpha_{\lambda'}}: \lambda'<\lambda\}$.
From the inductive hypothesis, $Z\cap P_{\alpha_{\lambda}}$ is a
finitely generated initial segment of $P_{\alpha _{\lambda}}$,
moreover, since each $P_{\alpha_{\lambda'}}$ has a largest element
$1_{\alpha_{ \lambda'}}$ and $\{\lambda'\in Q: \lambda'<\lambda
\}$ is finite, then $R_{\lambda}$ is finitely generated, hence $Z$
too. If $x\in P_{\alpha_{\lambda'}}$, $y\in
P_{\alpha_{\lambda''}}$ with $\lambda'\not=\lambda''$, then
$Z=\cup \{P_{\alpha_{\nu}}: \nu\in
\downarrow\lambda'\cap\downarrow\lambda''\}$. Since each
$P_{\alpha_{\nu}}$ has a largest element $1_{\alpha_{ \nu}}$ and
$\downarrow\lambda'\cap\downarrow\lambda''$ is finite, then $Z$ is
finitely generated. Next, each indivisible chain $C$ contained
into $I_{<\omega}(P_{\alpha})$ embeds into $\mathbb{P}(\alpha)$.
To see this, associate to each $Z\in I_{<\omega}(P_{\alpha})$ the
set $p(Z):=\{\lambda\in Q : Z\cap
P_{\alpha_{\lambda}}\not=\emptyset\}$ and observe that this set is
an initial segment of $Q$, hence is finite. Let $C$ be a chain
contained into $I_{<\omega}(P_{\alpha})$. The set $L$ formed by
the $p(Z)$ for $Z\in C$ is a chain of $I_{<\omega}(Q)$. Hence,
either $L$ is finite or has type $\omega$. For each $F\in L$ put
$C_{F}=\{Z\in C : p(Z)=F\}$. Each $C_{F}$ is an interval of $C$
and $C$ is the sum $\sum_{F\in L}C_{F}$. As a chain, $C_{F}$
embeds into the direct product $\Pi_{\lambda\in
MaxF}I_{<\omega}(P_{\alpha_{\lambda}})=\Pi_{\lambda\in
MaxF}T_{\alpha_{\lambda}}$. Let $\gamma$ be an indivisible chain
embedding into $C$. If $\gamma$ embeds into one of the $C_{F}$
then it embeds into $\Pi_{\lambda\in MaxF}T_{\alpha_{\lambda}}$
hence into one of the $T_{\alpha_{\lambda}}$ (cf. Lemma
\ref{propI.3.7}). From the inductive hypothesis, it embeds into
$\mathbb{P}(\alpha_{\lambda})$ hence into $\mathbb{P}(\alpha)$. If
not, $L$ has type $\omega$ and we can write $\gamma=\sum_{n<\omega
}\gamma_{n}$ with $\gamma_{n}$ indivisible embedding into some
$C_{F_{n}}$. We obtain
$\gamma\leq\sum_{n<\omega}\mathbb{P}(\alpha_{\lambda_{n}})\leq
\mathbb{P}(\alpha)$.\\
\end{proof}

\chapter[Independent sets in distributive lattices]{Infinite independent sets in distributive lattices} \label{chap:TD}

 We show that a poset $P$ contains a subset
isomorphic to $[\kappa]^{<\omega}$ if and only if the poset $J(P)$
consisting of ideals of $P$ contains  a subset isomorphic to
${\mathcal P}(\kappa)$, the power set of $\kappa$. If $P$ is a
join-semilattice this amounts to the  fact that $P$
contains  an  independent set of size $\kappa$.  We show that if
$\kappa:= \omega$  and $P$ is a distributive lattice, then this
amounts to the fact that $P$
 contains either
$I_{<\omega}(\Gamma)$ or
$I_{<\omega}(\Delta)$ as sublattices, where $\Gamma$ and $\Delta$ are
two special meet-semilattices already considered by J.D.Lawson,
M.Mislove and H.A.Priestley.

\section{Presentation of the results}
Let $P$  be a poset \index{poset}. An {\it
ideal} \index{ideal} of $P$ is a non-empty up-directed
\index{up-directed} initial segment \index{initial segment} of
$P$. The set $J(P)$ of ideals of $P$, ordered by inclusion, is an
interesting poset associated with $P$. And there are several
results about their relationship, eg \cite {duff, pz}. If $P:=
[\kappa]^{<\omega}$, the set,  ordered by inclusion, consisting of
finite subsets of $\kappa$, then the poset $J([\kappa]^{<\omega})$
is isomorphic to $\mathcal {P}(\kappa)$, the set,  ordered by
inclusion, consisting of arbitrary subsets of $\kappa$. We prove:

\begin{theorem} \label{thm1}
A poset $P$ contains a subset isomorphic to
$[\kappa]^{<\omega}$ if and only if $J(P)$ contains a subset
isomorphic to
$\mathcal {P}(\kappa)$ .\\
\end{theorem}

If $P$ is a join-semilattice \index{join-semilattice}, then such
containments amount to the existence of an independent set
\index{independent set} of size $\kappa$; a subset $X$ of $P$
being {\it independent} if $x\not \leq \bigvee F$ for every $x\in
X$ and every non-empty finite subset $F$ of $X\setminus \{x\}$.
The notion of independence  is better understood in terms of
closure operators \index{closure operator}. Let us recall that if
$\varphi$  is a closure operator on a set $E$,   a subset $X$ is
{\it independent} if $x\notin\varphi(X\setminus\{x\})$ for every
$x\in X$. Also, if $\varphi$ is algebraic then: $a)$ $\mathcal
{F}_{\varphi}$, the set of closed sets \index{closed set}, is  an
algebraic lattice \index{algebraic lattice}, $b)$ $\mathcal
{F}_{\varphi}^{<\omega}$, the set of its compact elements
\index{compact element}, is a join-semilattice with $0$ and $c)$
$\mathcal {F}_{\varphi}$ is isomorphic to $ J(\mathcal
{F}_{\varphi}^{<\omega})$.

A basic relationship  between closure operators and independent sets
is this.
\begin{theorem} \label{closure}\label{tm1}
Let $\varphi$ be a closure operator on a set $E$. The following
properties are equivalent:\\
$(i)$ $E$ contains an independent set of size $\kappa$;\\
$(ii)$ $\mathcal {F}_{\varphi}$  contains a subset isomorphic to
$\mathcal {P}(\kappa)$;\\
If $\varphi$ is algebraic, these two properties are equivalent to:\\
$(iii)$ $\mathcal
{F}_{\varphi}^{<\omega}$ contains a subset isomorphic to
$[\kappa]^{<\omega}$.\\
\end{theorem}
The proof is almost immediate and we will not give it. Via the
existence of an independent set, it shows that  the order
containment and the semilattice containment of
$[\kappa]^{<\omega}$ are equivalent, and that the same holds for
$\mathcal {P}(\kappa)$. \\ If $P$ is a join-semilattice  with $0$,
then  $J(P)$  is the lattice of closed sets of an algebraic
closure operator on $P$ whose independent sets are those defined
above in semilattice terms. In this case, Theorem \ref{thm1}
follows immediately from Theorem \ref {tm1}.

In the case of
distributive lattices \index{distributive lattice}, the order (or
join-subsemilattice) containment of $[\omega]^{<\omega}$ can be
replaced by some lattice containment. Given a poset $P$, let
$I_{<\omega}(P)$ be the set of finitely generated initial segments
\index{initial segment} of $P$. Let $\Delta:= \{ (i,j):
i<j\leq\omega \}$ ordered so that $(i,j)\leq (i',j')$ if and only
if either $j\leq i'$ or $i=i'$ and $j\leq j'$. Let
$\Gamma:=\{(i,j)\in \Delta : j=i+1$ or $j=\omega \}$ equipped with
the induced ordering. The posets $\Delta$ and $\Gamma$ are two
well-founded \index{well-founded} meet-semilattices
\index{meet-semilattice} whose maximal elements form an infinite
antichain \index{antichain}.  The posets $I_{<\omega}(\Delta)$ and
$I_{<\omega}(\Gamma )$ are two well-founded distributive lattices
\index{distributive lattice} containing a subset isomorphic to
$[\omega]^{<\omega}$.

\begin{figure}
\begin{center}

\includegraphics[width=2.5in]{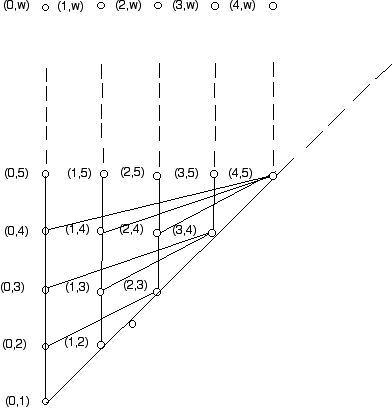}

\caption{$\Delta$}
\label{fig:delta}
\end{center}
\end{figure}

\begin{figure}
\begin{center}

 \includegraphics[width=2.5in]{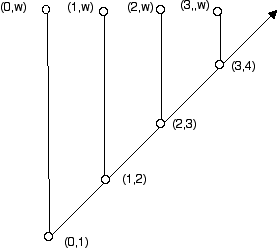}

\caption{$\Gamma$}
\label{fig:gamma}
\end{center}
\end{figure}

\begin{theorem} \label{thm8}
Let $T$ be a distributive lattice. The following properties are
equivalent:\\
$(i)$ $T$ contains a subset isomorphic to $[\omega]^{<\omega}$;\\
$(ii)$ The lattice $[\omega]^{<\omega}$ is a quotient of some
sublattice of $T$;\\
$(iii)$ $T$ contains a sublattice isomorphic to $I_{<\omega}(\Gamma)$
or to $I_{<\omega}(\Delta)$.\\
\end{theorem}
\begin{corollary}\label{thm9}
    Let $T$ be a distributive lattice. If $T$ contains a subset
isomorphic to $[\omega]^{<\omega}$ then it contains a
well-founded sublattice $T'$ with the same property.
    \end{corollary}

One of the main ingredient in our proof of Theorem \ref{thm8} is
Ramsey's theorem  \cite{rams} applied as in \cite{duff}. The posets
$\Delta$ (with a top element added) and $\Gamma$ have been considered
previously in \cite{lmp1} and \cite{lmp2}. The  poset
$\delta$ obtained from $\Delta$ by leaving out the maximal elements was
also considered by E. Corominas in 1970 as a
variant of an example built by R. Rado \cite{rado}.The results
presented here are contained in part in Chapter 1 of the doctoral
thesis of
the first author presented before the University Claude-Bernard
(Lyon1) december 18th,  1992 \cite{chak}, and announced in
\cite{cp}.
\section[Initial segments and ideals]{Initial segments, ideals and a proof of Theorem \ref{thm1}}
Our definitions and notations are standard and agree with \cite
{grat} except on minor points that we will mention. Let $P$ and
$Q$ be two posets. A map $f:P\rightarrow Q$ is {\it
order-preserving} \index{order-preserving} if $x\leq y$ in $P$
implies $f(x)\leq f(y)$ in $Q$; this is an {\it embedding}
\index{embedding} if the converse also holds; this is an {\it
order-isomorphism} \index{order-isomorphism} if in addition $f$ is
onto. We say that $P$ and $Q$ are {\it isomorphic}
\index{isomorphic}, or have the same  {\it order-type}
\index{order-type}, in notation $P\cong Q$, if there is an
order-isomorphism from $P$ onto $Q$. We also say that $P$ {\it
embeds} into $Q$ if there is an embedding from $P$ into $Q$, a
fact we denote $P\leq Q$. We denote $\omega$ the order-type of
$\N$, the set of natural integers, $\omega^*$ the order-type of
the set of negative integers. Let $P$ be  a join-semilattice with
a $0$, an element $x\in P$ is {\it join-irreducible}
\index{join-irreducible} if it is distinct from $0$, and if
$x=a\vee b$ implies $x=a$ or $x=b$ (this is slight difference with
\cite{grat}).\\ If  $P$ is a poset, a subset $I$ of $P$ is an {\it
initial segment} \index{initial segment} of $P$ if $x\in P$, $y\in
I$ and $x\leq y$ imply $x\in I$. If
 $A$ is a subset of $P$, then $\downarrow A:=\{x\in P: x\leq y$ for
some $y\in A\}$ denotes the least initial segment containing $A$.
If $I:=\downarrow A$ we say that $I$ is {\it generated} by $A$ or
$A$ is {\it cofinal in} \index{cofinal} $I$. If $A:=\{a\}$ then
$I$ is a {\it principal initial segment} \index{principal initial
segment} and we write $\downarrow a$ instead of $\downarrow
\{a\}$. The poset $P$ is {\it $\downarrow$-closed}
\index{$\downarrow$-closed}, if the intersection of two  principal
initial segments of $P$ is a finite union, possibly empty, of
principal initial segments. A {\it final segment} \index{final
segment} of $P$ is any initial segment of $P^*$,  the {\it dual}
\index{dual order} of $P$. We denote by $\uparrow A$ the final
segment generated by $A$. If $A:=\{a\}$ we write $\uparrow a$
instead of $\uparrow \{a\}$. The poset $P$ is $\uparrow$-{\it
closed} \index{$\uparrow$-closed} if its dual $P^*$ is
$\downarrow$-closed. A subset $I$ of $P$ is {\it up-directed}
\index{ up-directed} if every pair of elements of $I$ has a common
upper-bound \index{upper-bound} in $I$. An {\it ideal}
\index{ideal} is a non-empty up-directed initial segment of $P$
(in some other texts, the empty set is an ideal). We denote $
I(P)$ , resp. $I_ {<\omega }(P)$, resp. $J(P)$, the set of initial
segments, resp. finitely generated initial segments, resp. ideals
of $P$ ordered by inclusion and we set
 $ J_ *(P):=J(P)\cup\{\emptyset\}$, $I_ 0 (P):=I_ {<\omega}
(P)\setminus\{\emptyset\}$. We note that $I_{ <\omega}(P)$ is the set
of compact elements of $I(P)$, in particular $J(I_{ <\omega}(P))\cong
I(P)$. Moreover
$I_ {<\omega}(P)$ is  a lattice, and in fact a distributive lattice,
if and only if
$P$ is $\downarrow$-closed. We also note that $J(P)$  is
the set of join-irreducible elements of $I(P)$; moreover,
$I_{<\omega}(J(P))\cong I(P)$ whenever $P$ has no infinite antichain.
\\

\begin{lemma} \label{lemI.3.1}
Let $P, Q$ be two posets. $a)$ If $P\leq Q$ then $J(P)\leq J(Q)$;
the converse holds if $Q$ is a chain. $b)$ If $P$ and $Q$ are
join-semilattices \index{join-semilattice} (resp.
meet-semilattices \index{meet-semilattice}), and if $P$ embeds
into $Q$ by
 a join-preserving \index{join-preserving} (resp. meet-preserving \index{meet-preserving}) map then there is an
embedding from $J(P)$ into $J(Q)$ by a map preserving  arbitrary
joins (resp. arbitrary meets).
\end{lemma}
\begin{proof} Let $f$ be an order-preserving map from $P$ into $Q$.
The map $\psi_{f}:J(P)\rightarrow J(Q)$ defined by
$\psi_{f}(I):=\downarrow \!\!f[I]$  (where $f[I]:=\{f(x): x\in I\}$)
preserves suprema of up-directed subsets.
This is an embedding
provided that $f$ is an embedding.
Conversely, suppose  $J(P)\leq J(Q)$; if
$Q$ is a chain, then
$I(P)\cong 1+J(P)\leq 1+J(Q)\cong I(Q)$; let
$f$ be an embedding from $I(P)$ into $I(Q)$. For
$x\in P$, choose
$g(x)$ in $f(\downarrow \!\!x)\setminus
f(\{y: y<x\})$. The map
$g$ is an embedding from $P$ into $Q$. This proves a). It is easy to
check that the map $\psi_{f}$ satisfies the properties stated in b).
  \end{proof}
\begin{proposition} \label{lemI.3.3}
Let $K$ be a poset such that for every $x\in K$ the initial
segment $\downarrow \!\!x$ is finite and the initial segment
$K\setminus\uparrow x$ is a finite union of ideals. For every
poset $P$, the following properties are equivalent:\\ $(i)$ $K\leq
P$;\\ $(ii)$ $J(K)\leq J(P)$.\\
\end{proposition}
\begin{proof}
$(i)\Rightarrow(ii)$ Lemma \ref{lemI.3.1}.\\
$(ii)\Rightarrow(i)$ Let $f$ be a map from $J(K)$ into $J(P)$. For
$x\in K$, put $F(x):=\{f(J): x\not\in
J\in J(K)\}$ and $R(x):=\bigcup F(x)$ and $C(x):=f(\downarrow \!\!x)\setminus\
R(x)$. \\
{\bf Claim} $C(x)$ is non-empty.

 Indeed, if $x$ is the least element of $K$
then $R(x)$ is empty, hence $C(x)=f(\downarrow \!\!x)$ which is
non-empty. If $x$ is not the  least element of $K$  then
$K\setminus\uparrow x$ is non-empty, hence, from our hypothesis
on $K$, is the union $J_{1}\cup J_{2}\cup\dots\cup J_{n_{x}}$ of a
non-empty family of ideals. Since the ideals of $P$ are the
join-irreducible members of the distributive lattice $I(P)$, it
follows first
 that $R(x)= f(J_{1})\cup f(J_{2})\cup\dots\cup f(J_{n_{x}})$
(indeed, if $f(J)\subseteq R(x)$  then $J\subseteq
J_{n_{i}}$ for some $i$). With this, it follows next that
$C(x)$ is non-empty (otherwise, from $f(\downarrow \!\!x)\subseteq
R(x)$ we would have $f(\downarrow \!\!x)\subseteq
f(J_{n_{i}})$, hence $x\in
J_{n_{i}}$, for some $i$, which is impossible) and our claim is proved.
\\

Let
$g:K\rightarrow P$ be such that $g(x)\in C(x)$ for every $x\in K$.
Clearly $x'\not\leq x''$ implies $g(x')\not\leq g(x'')$.
Indeed, if
$x'\not\leq x''$ then $x'\not\in\downarrow \!\!x''$, hence $C(x')\cap
f(\downarrow \!\!x'')=\emptyset$. This, added to the fact that $g(x')\in
C(x')$, $g(x'')\in
f(\downarrow \!\!x'')$ and $f(\downarrow \!\!x'')$ is an initial segment,
gives $g(x')\not\leq g(x'')$. Hence, if   such a $g$ is
order-preserving then this is an embedding.  We define
such an order-preserving map
$g$  by induction on the size of
$\downarrow \!\!x$. If $\vert\downarrow \!\!x\vert=1$ then we
choose for
$g(x)$ any element of
$C(x)$. If
$\vert\downarrow \!\!x\vert>1$ and if $g(y)$ is defined for all $y$,
$y<x$, then we have $g(y)\in f(\downarrow \!\!y)\subseteq f(\downarrow \!\!x)$, hence $g(y)\in f(\downarrow \!\!x)$. The set $\{g(y): y<x\}$ is
finite, hence, it has an upper-bound $z$ in $f(\downarrow \!\!x)$. Select
$g(x)$ in
$C(x)\cap\uparrow \!\!z$.
\end{proof}\\

Hypotheses of this proposition are satisfied if
$K:=[\kappa]^{<\omega}$. Since in this case $J(K)\cong {\mathcal
P}(\kappa)$, we obtain Theorem \ref{thm1}. \\These hypotheses are
also satisfied by many other posets. This is the case of
join-semilattices which  embed  into some  $[\kappa]^{<\omega}$ as
join-subsemilattices. It is not difficult to see that  a
join-semilattice $K$ has this property  if and only if for every
$x\in K$, the set $J(P)^{\bigtriangleup}_{\neg x}$ of completely
meet-irreducible \index{completely meet-irreducible} members of
$J(P)$ which do not contain $x$ is finite (see \cite {chapou2}).
An interesting example is  $\delta$ ( note that the map $\varphi$
from $\delta$ into $[\omega]^{<\omega}$ defined by
$\varphi(i,j):=\{0,\dots,j-1\}\setminus \{i\}$ is
join-preserving). Posets of the form $K:=I_{<\omega}(Q)$, where
$Q$ embeds into $[\kappa]^{<\omega}$, also embed into
$[\kappa]^{<\omega}$ as join-semilattices. Evidently, it is only
needed to consider those not embedding  $[\kappa]^{<\omega}$ as a
join-semilattice, this amounting to the fact that  $Q$ contains no
infinite antichain (cf. Section 3, Corollary \ref {belordre}).  A
simple minded example is $K:=\omega\times\dots\times\omega$, the
direct product \index{direct product} of finitely many copies of
$\omega$ (obtained with $Q :=\omega\oplus\dots\oplus\omega$, the
direct sum \index{direct sum}of the same number of copies of
$\omega$). \\ The sierpinskization \index{sierpinskization}
technique leads to examples of join-semilattices satisfying the
hypotheses of Proposition \ref{lemI.3.3}, but which do not embed
into $[\omega]^{<\omega}$ as join-subsemilattices. A {\it
sierpinskization} of a countable order-type $\alpha$ with $\omega$
is a poset $S_{\alpha}$ made of a set $E$ and an order $\mathcal
E$ on $E$ which is  the intersection of two linear orders
\index{linear order} $\mathcal A$ and $\mathcal  N$ on $E$ such
that the chains $A:=(E, \mathcal A)$ and $N:=(E, \mathcal N)$ have
types $\alpha$ and $\omega$ respectively. Clearly, every principal
initial segment of $S_{\alpha}$ is finite. More importantly,
non-principal ideals of $S_{\alpha}$ form a chain: in fact, {\it a
subset  $I$ of $E$ is a non-principal   ideal of $S_{\alpha}$ if
and only if $I$ is a non-principal ideal of the chain $A$}. Also,
{\it if $K$ is a sierpinskization of $\alpha$ and $\omega$ then
for every $x\in K$, the initial segment $K\setminus\uparrow x$ is
a  finite union of ideals if and only if no interval of $\alpha$
has order-type $\omega^*$}  (cf \cite {chapou2}).

From this, Proposition \ref{lemI.3.3} applies to any
sierpinskisation of $\omega\alpha$ with $\omega$. As shown in
\cite {pz}, those given by a bijective map
$\psi:\omega\alpha\rightarrow \omega$ which is  order-preserving
on each component $\omega \cdot \{i\}$ of $\omega\alpha$ are all
embeddable in each other, and for this reason  denoted by the same
symbol $\Omega(\alpha)$; moreover, $\Omega(\alpha)$ embeds into
$\Omega(\beta)$ if and only if $\alpha$ embeds into $\beta$.
Consequently, the sierpinskization technique allows to construct
as  many  countable posets $K$ satisfying the hypotheses of
Proposition \ref{lemI.3.3} as there are countable order-types. A
bit more is true. {\it Among the representatives of $\Omega
(\alpha)$, some are semilattices }(and among them, subsemilattices
of the direct product $\omega\times \alpha$). {\it Except for
$\alpha=\omega$, the representatives of $\Omega (\alpha)$  which
are join-semilattices never embed into $[\omega]^{<\omega}$ as
join-semilattices }(whereas they embed  as  posets) \cite {chapou2}. \\

The posets $\Omega (\alpha)$ and $I_{<\omega}(\Omega (\alpha))$ do
not embed in each other as join-semilattices.  They provide two
examples of a  join-semilattice $P$ such that $P$ contains no
chain of type $\alpha$ and  $J(P)$ contains a chain of type
$J(\alpha)$. However, note that if $\alpha$ embeds $\omega^*$ then
$I_{<\omega}(\Omega (\alpha))$ reduces to $[\omega]^{<\omega}$
(they embed  in each other as join-semilattices). Are they
substantially different examples? (see \cite{chapou2} for more).

\section[Infinite independent set]{Distributive lattices containing an infinite independent set: proof of Theorem \ref{thm8}}
A poset $P$ is {\it well-founded}
\index{well-founded} if every non-empty subset has a minimal
element; it is {\it well-quasi-ordered} \index{well-quasi-ordered}
(w.q.o. in brief) if it is well-founded with no infinite
antichain. Let us recall the Higman's characterization of w.q.o.
sets \cite{higm}.
\begin{theorem}\label{Higman}
Let $P$ be a poset. The following properties are equivalent:\\
$(i)$ $P$ is well-quasi-ordered;\\
$(ii)$ For every infinite sequence $(x_{n})_{n<\omega}$ of elements
of $P$, some infinite subsequence is non-decreasing;\\
$(iii)$ For every infinite sequence $(x_{n})_{n<\omega}$ of elements
of $P$, there are $n<m$ such that $x_{n}\leq x_{m}$;\\
$(iv)$ Every final segment of $P$ is finitely generated;\\
$(v)$ The set $I(P)$ of initial segments of $P$,
ordered by inclusion, is well-founded.\\
\end{theorem}
The following lemma provides an alternative version of the {\it
minimal bad-sequence} \index{minimal bad-sequence} technique
invented by C. St J. A. Nash-Williams \cite{nash}.
\begin{lemma}\label{bad-sequence} Let $P$ be a well-founded poset. If
$P$ contains an infinite antichain then it contains some infinite
antichain $A$ such that:\\ $1)$ for every $x\in P$ either $x\geq
y$ for some $y\in A$ or $x<y$ for all $y\in A$ but finitely
many;\\ $2)$ $P\setminus \uparrow A$ is w.q.o.\\
\end{lemma}
\begin{proof}
Since $P$ is well-founded, every final segment of $P$ is generated by
its minimal elements. Let $\mathcal{F}$ be the set of
non-finitely generated final segments of $P$. If $X$ is an infinite
antichain of $P$ then $\uparrow \!\!X
\in \mathcal{F}$; moreover $\mathcal{F}$ is closed by union of
chains.  Hence $\mathcal{F}$, ordered by inclusion, is
inductive. According to Zorn's lemma, it has some
maximal element $F$. We claim that  $A:=Min(F)$ satisfies $1)$ and
$2)$. Indeed, let first $x\in P$. Set $F':=F\cup \uparrow \!\!x$.
Since $F'$ is a final segment containing $F$ then either $F'=F$ or
$F'$ is finitely
generated. In the first case $x\geq y$ for some $y\in A$ whereas in
the latter case $x<y$ for all $y\in A$ but finitely many, since $Min
(F')= \{x\}\cup (A\setminus\uparrow \!\!x)$. This  proves that $1)$ holds.
Next, let $G$ be a  final segment of $P\setminus \uparrow \!\!A$.
Then
$G\cup \uparrow \!\!A$  is a final segment of $P$. Since $\uparrow \!\!A=F$,
this final segment strictly contains $F$ if $G$ is non-empty,
hence it is finitely generated; this implies that
$G$ is finitely generated. According to $(iv)$ of Theorem
\ref{Higman},
$P\setminus \uparrow \!\!A$ is w.q.o.
\end{proof}
\begin{lemma}\label{badsequence2} If $P$ is $\uparrow$-closed and $A$
is an infinite antichain satisfying  condition $1)$  of Lemma
\ref{bad-sequence} and  distinct from  $Min(P)$ then $P\setminus
\uparrow \!\!A$ is an ideal; in particular, if
$P$ is a join-semilattice then all members of $A$ are
join-irreducible.
\end{lemma}
\begin{proof} Since $A\not=Min(P)$ then $P\setminus \uparrow \!\!A$ is
non-empty. Let $x,y\in P\setminus\uparrow \!\!A$. Condition
$1)$ insures that $A\setminus(\uparrow \!\!x\cap\uparrow \!\!y)$ is finite. Hence $\uparrow \!\!x\cap\uparrow \!\!y$ is
infinite and since $P$ is $\uparrow$-closed, $\uparrow \!\!x\cap\uparrow \!\!y=\uparrow \!\!F$ where $F$ is some finite set. Let $z\in F$ such that
$\uparrow \!\!z\cap A$ is infinite. Since $z\in F$, we have $x, y\leq
z$ and since $\uparrow \!\!z\cap A$ is infinite $z<a$ for some $a\in A$,
hence $z\in P\setminus \uparrow \!\!A$, proving that $P\setminus
\uparrow \!\!A$ is up-directed.\\
If $P$ is a join-semilattice, it is $\uparrow$-closed, hence
$P\setminus \uparrow \!\!A$ is an ideal. In this case $x\vee y\in
P\setminus
\uparrow \!\!A$ whenever $x, y\in P\setminus \uparrow \!\!A$. Hence if $a\in
A$ and $a=x\vee y$ then $x<a$ and $y<a$ are impossible.
\end{proof}\\
Let us recall a basic property of join-irreducibles in a distributive
lattices.
\begin{lemma}\label{lemII.1}
    Let $T$ be a distributive lattice and $x\in T$ with  $x$ distinct
    from the least element of $T$ if any. The following properties are
    equivalent:\\
$(i)$ $x$ is  join-irreducible;\\
$(ii)$ If  $a,b\in T$ and  $x\leq a\vee b$ then  $x\leq a$ or $x\leq
b$;\\
$(iii)$ For every integer $k$ and every $ a_{1},\dots,a_{k}\in T$
if $x\leq a_{1}\vee \dots\vee a_{k}$ then  $x\leq a_{i}$ for some
$i$,  $1\leq i\leq k$.
\end{lemma}

\begin{corollary}\label{belordre} A well-founded distributive lattice $T$ contains no
subset isomorphic to $[\omega]^{<\omega}$ if and only if it
contains no infinite antichain.
\end{corollary}
\begin{proof}
If $T$ contains an infinite antichain then, since $T$ is
well-founded, Lemmas \ref{bad-sequence} and \ref{badsequence2} apply,
hence the set $T^{\vee}$ of join-irreducible elements of $T$ contains
an infinite antichain. From Lemma \ref{lemII.1}, this antichain is in
fact an independent subset of $T$;
according to Theorem \ref {tm1}, $T$ contains a subset isomorphic to
$[\omega]^{<\omega}$. The converse is obvious.
\end{proof}\\

In the sequel we describe typical well-founded meet-semilattices
containing infinite antichains.

\begin{lemma}\label{lem II.3}
    Let $P$ be a meet-semilattice and $f: \Delta \rightarrow  P$  be a
map satisfying $f(i,j)=f(i,\omega )\wedge f(j,\omega)$ for all
$i<j<\omega$ .Then the following properties are equivalent:\\
(i) $f$ is meet-preserving;\\
(ii) $f$ is order-preserving;\\
(iii) $f(i,j)\leq f(k,\omega)$ for all $i<j<k<\omega$;\\
(iv) $f(i,j)\leq f(j,k)$ for all $i<j<k<\omega$;\\
(v) $f(i,j)\leq f(i,k)$ for all $i<j<k<\omega$;\\
(vi) $f(i,j)=f(i,k)\wedge f(j,k)$ for all $i<j<k<\omega$;\\
Moreover, if $f$ satisfies these conditions, then $f$ is one-to-one
if and only if it satisfies conditions $a)$ $f(i,j)< f(j,k)$  and
$b)$ $f(i,j)< f(i,k)$  for all $i<j<k<\omega$.
\end{lemma}
\begin{proof}\\
$(i)\Rightarrow (ii)$ Evident.\\
$(ii)\Rightarrow ( iii)$  In $\Delta$ we have $(i,j)\leq (k,\omega)$
for all $i<j<k<\omega$; if $f$ is order-preserving, then this
inequality is preserved.
\\$(iii)\Longleftrightarrow
(iv)\Longleftrightarrow (v)\Longleftrightarrow (vi)$
Since $f(i',j')=f(i',\omega)\wedge f(j',\omega)$ for all
$i'<j'<\omega$, we have
$f(i,j)\wedge f(i,k)=f(i,j) \wedge f(k,\omega)=f(i,j)\wedge
f(j,k)=f(i,k)\wedge f(j,k)$ for all $i<j<k<\omega$. The three
equivalences
follow.
$(iii)\Rightarrow (ii)$ Let $x:=(i,j)$ and $x':=(i',j')$ in $\Delta$
such that $x< x'$.\\ {\bf Case 1} $i=i'$  and $j< j'$.
 From the definition of $f$  we have $f(i,j)=f(i,\omega)\wedge
f(j,\omega)$, hence if $j'=\omega$ then $f(x)\leq f(x')$  . If
$j'<\omega$, then from
$(v)$, we have also $f(x)=f(i,j)\leq f(i,j')=f(x')$.\\
{\bf Case 2} $j \leq i'$. Suppose  $j'=\omega$. If $j<i'$ then from
$(iii)$ $f(x)=f(i,j)\leq f (i',\omega)=f(x')$; if $j=i'$ then from the
definition of $f(x)$ we have $f(x)=f(i,j)\leq f(j,\omega)=f(x')$.
Suppose $j'<\omega$. From $(iv)$ we have $f(x)=f(i,j)\leq f(j,i') \leq
f(i',j')=f(x')$ if $j<i'$ or $f(x)=f(i,j)\leq f(j,j')=f(i',j')=f(x')$
if $j=i'$.\\
$(ii)\Rightarrow (i)$ We have to check that $f(x\wedge x')= f(x)
\wedge f(x')$ for all pairs $x:=(i,j)$ , $x':=(i',j')$ in $\Delta$.
Since
$f$ is order-preserving, we only need to consider incomparable pairs.
Let $x:=(i,j)$ , $x':=(i',j')$ be such a pair. We have $i\not =i'$.
We may suppose $i<i' $; in this case, since $x$ and $x'$ are
incomparable we have $i'<j$,  hence
$   x\wedge x'=(i,  i')$.   \\
 If we suppose $j<\omega$ and $j'<\omega$ we have
$f(x)\wedge f(x')=f(i,j)\wedge
    f(i',j')=f(i,\omega )\wedge f(j,\omega )\wedge f(i',\omega )\wedge
    f(j',\omega )= f(i,i')\wedge f(j,\omega )\wedge f(j',\omega
)=f(i,i')=f(x\wedge x')$ since $(i,i')\leq
(j,\omega)\wedge(j',\omega)$ and
$f$ is order-preserving.\\
If $j<\omega$ and
$j'=\omega$ we have similarly $f(x)\wedge f(x')=f(i,j)\wedge
f(i',\omega)=f(i,\omega)\wedge f(j,\omega)\wedge
f(i',\omega)=f(i,i')\wedge f(j,\omega)=f(i,i')=f(x\wedge x')$. If
$j=\omega$ and
$j'<\omega$ we have $f(x)\wedge f(x')=f(i,\omega)\wedge
f(i',j')=f(i,\omega)\wedge f(i',\omega)\wedge
f(j',\omega)=f(i,i')\wedge f(j',\omega)=f(i,i')=f(x\wedge x')$. If
$j=j'=\omega$ we have $f(x)\wedge f(x')=f(i,\omega)\wedge
f(i',\omega)=f(i,i')=f(x\wedge x')$.

    Suppose that $f$ satisfies conditions $(i)-(vi)$. If $f$ is
one-to-one then from $(ii)$ applied to $(i,j)<(j,k)$ and $(i,
j)<(i,k)$,
we get $f(i,j)< f(j,k)$ and
$f(i,j)< f(i,k)$ as required. Suppose that $f$ is not one-to-one. Let
$x:=(i,j)$ and $x':=(i',j')$ be two distinct elements in $\Delta$ such
that $f(x)=f(x')$. \\
 {\bf Case 1} $x$ and $x'$ are comparable. We may suppose $x<x'$.
{\bf Subcase 1.1}  $i=i' $ and  $j<j'$. Since $f$ is
order-preserving $f(i,j)=f(i,j+1)$  and condition $b)$ is violated.
{\bf Subcase 1.2} $j\leq i'$. Since in this case $(i,j)\leq
(j,j+1)\leq  (i',j')$  we have $f(i,j)=f(j,j+1)$ and condition $a)$ is
violated.\\
{\bf Case 2} $x$ and $x'$ are incomparable. We may suppose $i<i'<j$.
Since $f$ is meet-preserving, we have
$f(i,j)=:f(x)=f(x')=f(x)\wedge f(x')=f(x\wedge x')=f(i,i')$ , hence
$f(i,i')=f(i,i'+1)$ and  condition
$b)$ is violated.
\end{proof}\\
Let $V$ be the meet-semilattice made of a countable antichain and a
least element added
(formally $V:=\{X\subseteq \omega: \vert X\vert  < 2\}$ ordered by
inclusion).
\begin{lemma}\label{lemII.7}
    Let $P$ be  a meet-semilattice. If $P$ contains an infinite
    antichain $A$ such that the set $P_{<}(A):=\{x\in P : x<a $ for some
$a\in A\}$ is
    well-quasi-ordered, then there is an infinite subset $A'$ of $A$
such that the
    meet-subsemilattice $P'$ of $P$ generated by $A'$ is either
isomorphic to $\Delta$,
    to $\Gamma$ or to $V$.
    \end{lemma}

\begin{proof}
    Let $\{x_{n} : n<\omega\}$ be a countable subset of $A$. Consider
the partition
    of $[\N]^{3}=\{(i,j,k) : i<j<k<\omega \}$ into the following parts:\\
$R_{1}:=\{(i,j,k)\in [\N]^{3} : x_{i}\wedge x_{j}$ incomparable to
$x_{i}\wedge
x_{k}\}$;\\
$R_{2}:=\{(i,j,k)\in [\N]^{3} : x_{i}\wedge x_{j}>x_{i}\wedge
x_{k}\}$;\\
$R_{3}:=\{(i,j,k)\in [\N]^{3} : x_{i}\wedge x_{j}<x_{i}\wedge
x_{k}\}$;\\
$R_{4}:=\{(i,j,k)\in [\N]^{3} : x_{i}\wedge x_{j}=x_{i}\wedge
x_{k}=x_{j}\wedge x_{k}\}$;\\
$R_{5}:=\{(i,j,k)\in [\N]^{3} : x_{i}\wedge x_{j}=x_{i}\wedge
x_{k}<x_{j}\wedge x_{k}\}$.\\
From Ramsey 's theorem\cite{rams} there is an infinite subset $H$ of $\N$
such that $[H]^{3} \subseteq  R_{i}$ for some $i\in\{1,2,3,4,5\}$.
Since $P_{<}(A)$ is well-quasi-ordered, $[H]^{3}$ is neither included
in $R_{1}$, nor
in $R_{2}$. In the remaining
cases, we select  an infinite subset $H'$  of $H$ and, setting
$A':=\{x_{n}: n\in H' \}$, we define a meet-preserving and
one-to-one map
$h_{H'}$ from $\Delta$,
    $\Gamma$ or  $V$ onto the meet-semilattice $P'$ generated by $A'$ .
In order to do so, we denote  by $\theta_{K}$  the
unique order-isomorphism from
$\N$ onto
 an infinite subset $K$ of  $\N$.
\\  {\bf Case 1} $[H]^{3}\subseteq R_{3}$.
Let $h_{H}: \Delta \rightarrow P$ be defined by
$h_{H}(x):=x_{\theta_{H}(i)}$ if $x:=(i,\omega)$ and $h_{H}(x):=
h_{H}(i,\omega)\wedge h_{H}(j,\omega)$ if $x:=(i,j)$ with
$i<j<\omega$. The  map $h_{H}$ satisfies condition $(v)$ of Lemma
\ref {lem II.3}, hence, it is meet-preserving.  It also satisfies
condition $b)$ of Lemma \ref {lem II.3} but it is not necessarily
one-to-one.  Let $H'$ be the image of $2\N$, the set of even
integers, by $\theta_{H}$, and let $\theta_{H'}$ be the unique
order-isomorphism \index{order-isomorphism} from $\N$ onto $H'$.
Like $h_{H}$, the map $h_{H'}: \Delta \rightarrow P$   is also
meet-preserving \index{meet-preserving map} and satisfies
condition $b)$. It also satisfies condition $a)$. Indeed, let
$i<j<k<\omega$. Since $\theta_{H'}(n)=\theta_{H}(2n)$ for every
$n<\omega$, we have $h_{H'}(i,j)=h_{H}(2i,2j)\leq
h_{H}(2j,2j+1)<h_{H}(2j,2k)=h_{H'}(j,k)$. According to Lemma \ref
{lem II.3}, $h_{H'}$ is one-to-one. \\ {\bf Case 2}
$[H]^{3}\subseteq R_{4}$.\\ In this case we have $x_{i}\wedge
x_{j}=x_{i'}\wedge x_{j'}$ for every $(i,j),(i',j')\in[H]^{2}$; we
denote $a$ this common value. Let $H':= H$, $h_{H'}: V \rightarrow
P$ defined by setting $h_{H'}(x):=x_{\theta_{H'}(i)}$ if
$x:=\{i\}$ and $h_{H'}(x):=a$ if $x:=\emptyset$.\\ {\bf Case 3}
$[H]^{3}\subseteq R_{5}$.\\ Set $H':=H$; let $h_{H'}:
\Gamma\rightarrow P $  be the map defined by setting
$h_{H'}(x):=x_{\theta_{H'}(i)}$ if $x:=(1,i)$ and
$h_{H'}(x):=x_{\theta_{H'}(i)}\wedge x_{\theta_{H'}(j)}$, where
$j$ is any element of $H'$ such that $i<j$,  if $x:=(0,i)$.
\end{proof}
\begin{lemma}\label{lemII.4}
Let $P$ be a well-founded meet-semilattice. The image of $P$ by a
meet-preserving map $f$  is well-founded.
\end{lemma}
\begin{proof}
Let $ Y$ be a non-empty subset of $f(P)$, then the image by $f$ of
every minimal
    element $a$ in $f^{-1}(Y)$ is minimal in $Y$.  Indeed, let $b\in Y$
    with $b\leq f(a)$. We have $b=f(a')$ for some $a'\in f^{-1}(Y)$.
Since $f$ is meet-preserving, we
have $b=f(a)\wedge f(a')=f(a\wedge a')$. As a first consequence, we
get  $a\wedge a'\in
    f^{-1}(Y)$  which, in turns, gives $a=a\wedge a'$ since $a$ is
minimal in $f^{-1}(Y)$.  Next, we get
 $b= f(a)$ proving the minimality of $f(a)$.
\end{proof}

\begin{theorem}\label{semiwqo}
Let $P$ be a meet-semilattice. The following properties are
equivalent:\\
$(i)$ The meet-subsemilattice generated by some infinite antichain of
$P$ is well-founded;\\
$(ii)$ The meet-subsemilattice $Q$ generated by some  infinite
antichain $A$ of $P$ is such that $\{x\in Q: x<y$ for some $ y \in
A\}$
is well-quasi-ordered;\\
$(iii)$ $P$ contains a meet-subsemilattice isomorphic to $\Gamma$ or
to $\Delta$ or to $V$;\\
$(iv)$ There is a meet-preserving map $f: \Delta \rightarrow P$
whose  image contains an infinite antichain.\\
\end{theorem}

\begin{proof}
$(i)\Rightarrow (ii)$ Lemma \ref{bad-sequence}.\\
$(ii)\Rightarrow (iii)$ Lemma \ref{lemII.7}.\\
$(iii)\Rightarrow (iv)$ Each of the  meet-semilattices $\Gamma$ and
$V$ is a  quotient of $\Delta$ by a meet-preserving map. Indeed,
define $g:\Delta\rightarrow  \Gamma$ by $g(i,j):= (1,i)$ if $j =
\omega$ and $g(i,j):= (0,i)$ otherwise. Clearly $g(i,\omega)\wedge
g(j,\omega)=(1,i)\wedge (1,j)=(0,i)=g(i,j)$ for all $i<j<\omega$;
hence, from Lemma \ref {lem II.3} , $g$ is meet-preserving.
Similarly, define
$h:\Delta\rightarrow V$  by $h(i,j):= \{i\}$ if $j=\omega$ and
$h(i,j):=\emptyset$ otherwise;  the map $h$ is meet-preserving too.\\
$(iv)\Rightarrow (i)$ Lemma \ref {lemII.4}.
\end{proof}

\begin{lemma}\label{lemII.2}
    Let $T$ be a distributive lattice containing an independent set $L$
with at least two elements. Let
    $<L>$ be the sublattice of $T$ generated by $L$. Then:\\
$(i)$ every $a\in L$ is join-irreducible in $<L>$;\\ $(ii)$ the
map $\varphi$ from $<L>$ into $\mathcal{P}(L)$  defined by
$\varphi (x):=\{a\in L : a\leq x\}$ is a  lattice-homomorphism
\index{lattice-homomorphism} whose image is $[L]^{<\omega}$.
\end{lemma}
\begin{proof}
    $(i)$ Let $a\in L$. Since $\vert L\vert  \geq 2$, $a$ is distinct
from the least element of $<L>$  (if any). Let $u,v\in <L>$ such that
$a\not\leq u$ and
$a\not\leq
    v$. Since $<L>$ is distributive, we can write $u=u_{1}\wedge
u_{2}\wedge
    \dots \wedge u_{p}$ and $v=v_{1}\wedge v_{2}\wedge
    \dots \wedge v_{q}$, where every $u_{i}$ is the supremum of a finite
subset
    $U_{i}$ of $L$  and every $v_{j}$ is the supremum of a finite subset
    $V_{j}$ of $L$. Since $a\not\leq u$, there is some $u_{i}$ such that
$a\not\leq
    u_{i}$ and similarly, since $a\not\leq v$, there is some $v_{j}$
such that $a\not\leq
    v_{j}$. Since $u_{i}=\bigvee U_{i}$ then $a\notin U_{i}$ and, since
$v_{j}=\bigvee V_{j}$, then $a\notin V_{j}$ . So
$a\notin
    U_{i}\cup V_{j}$. Since $L$ is independent, it follows that
$a\not\leq
    \bigvee (U_{i}\cup V_{j})=u_{i}\vee v_{j}$.  Since $u\vee v\leq
    u_{i}\vee v_{j}$ then  $a\not\leq u\vee v$ proving that $a$ is
join-irreducible.\\
$(ii)$ As a map from $<L>$ into $\mathcal{P}(L)$, $\varphi$ is a
lattice-homomorphism. Indeed, let $x, y\in<L>$. The
equality
$\varphi (x\wedge y)=\varphi (x)\cap
\varphi (y)$ is clear. Let us check that  the equality $\varphi (x\vee
y)=\varphi (x)\cup \varphi (y)$ holds. Obviously $\varphi (x)\cup
\varphi
(y)\subseteq  \varphi (x\vee
y)$. For the reverse inclusion, let  $a\in \varphi (x\vee
y)$.
From  $(i)$ $a$  is join-irreducible in $<L>$, hence $a\leq x$ or
$a\leq y$  that is $x\in \varphi (x)\cup \varphi (y)$.  Since
$\vert L\vert \geq 2$,
$[L]^{<\omega}$ is the sublattice of $\mathcal{P}(L)$  generated by
$L':= \{\{a\}: a\in L\}$. Since
$\varphi$ is a lattice-homomorphism and $L'$ is the image of $L$ it
follows that $[L]^{<\omega}$ is the image of
$<L>$.
\end{proof}
\begin{lemma}\label{lem II.4}
    Let $T$ be a lattice and $\varphi$  be a lattice-homomorphism from
$T$ onto
    $[\omega]^{<\omega}$. Then there is a meet-preserving map $f$ from
$\Delta$ into $T$ such that $\varphi of(i,\omega)=\{i\}$
for every $i<\omega$.
\end{lemma}
\begin{proof}
    First, we define $f(i,\omega)$ for $i<\omega$. Denote $B_{0}:=\{x
\in T:
    \varphi (x)=\emptyset  \}$ and  $A_{i}:=\{x\in T : \varphi
(x)=\{i\}\}$ for $ i < \omega$. Since $\varphi$ is onto, these sets
are non-empty.
    Let $b_{0}\in B_{0}, a_{0}\in A_{0}$ and $a_{1}\in A_{1}$.  Set
    $f(0,\omega):=a_{0}\vee b_{0}$ and $f(1,\omega):=a_{1}\vee b_{0}$.
Since
    $\varphi $ is a lattice-homomorphism we have $\varphi(
    f(i,\omega))=\varphi (a_{i})\vee \varphi (b_{0})=\{i\}\cup
\emptyset=\{i\}$ for
    $i=0,1$. Let $k\geq2$. Suppose $f(i,\omega)$ defined for $i<k$.
Choose $a_{k}$ in
    $A_{k}$ and set $b_{k}:=\bigvee \{f(i,\omega )\wedge f(j,\omega) :
i<j<k\}$.
    Put $f(k,\omega):=b_{k}\vee a_{k}$. We have $\varphi
(b_{k})=\emptyset $, and
    so $\varphi (f(k,\omega ))=\varphi (b_{k})\vee \varphi
(a_{k})=\{k\}$.
    Finally, put $f(i,j):=f(i,\omega )\wedge f(j,\omega)$ for all
    $i<j<\omega$. \\The map $f$ satisfies  condition $(iii)$ of Lemma
\ref{lem II.3}.
Indeed, let $i<j<k<\omega$, we have $f(i,j):=f(i,\omega )\wedge
f(j,\omega)\leq b_{k}\leq b_{k}\vee
a_{k}=:f(k,\omega)$, hence $f(i,j)\leq f(k, \omega)$, as required.
Hence, according to Lemma \ref{lem II.3}, $f$ is
meet-preserving. \\
\end{proof}

\begin{lemma}\label{lemII.5}
Let $f: P\rightarrow T$ be a meet-preserving map
from a meet-semilattice $P$ into a distributive lattice $T$.\\
$(i)$ The map $f^{\vee} $ from
$I_{0}(P):=I_{<\omega}(P)\setminus\{\emptyset\}$ into  $T$, defined
by $f^{\vee}(A):=\bigvee \{f(a) : a\in A\}$, is a
lattice-homomorphism.\\
$(ii)$$ f^{\vee} $ is  injective  if and only if $1)$ $f$ is
injective and
$2)$ for every $x\in P$, every finite non-empty subset $X$
of  $P$,  the equality $f(x)=\bigvee f(X)$ implies $x\in
X$.
\end{lemma}

\begin{proof}
    $(i)$ Let $A, B\in I_{0}(P)$. We have $f^{\vee}(A\cup
B)=\bigvee\{f(c) : c\in A\cup
    B\}=(\bigvee\{f(a) : a\in A\})\vee (\bigvee \{f(b) : b\in
    B\})=f^{\vee}(A)\vee f^{\vee}(B)$ , hence  $f ^{\vee}$ is
join-preserving.
    By definition $f^{\vee}(A)\wedge f^{\vee}(B)=(\bigvee\{f(a) : a\in
    A\})\wedge (\bigvee\{f(b) : b\in B\})$. Since $A$ and $B$ are
    finitely generated,  $T$ is distributive and $f$ is meet-preserving,
we have $f^{\vee}(A)\wedge
    f^{\vee}(B)=\bigvee \{f(a)\wedge f(b) : a\in A, b\in
    B\}=\bigvee \{f(a\wedge b) : a\in A, b\in B\}$.  Since  $A$ and $B$
are initial segments of $P$ then
    $A\cap B=\downarrow \!\!\{a\wedge b : a\in A, b\in B\}$. Hence,
    $f^{\vee}(A)\wedge f^{\vee}(B)=\bigvee\{f(c) : c\in A\cap
    B\}=f^{\vee}(A\cap B)$, proving that $f ^{\vee}$ is
meet-preserving.\\
$(ii)$ Let $T':=I_{0}(P)$ and let $i: P\rightarrow T'$ be the map
defined by
$i(x)=\downarrow \!\!x$. This map is injective and  maps $P':=P\setminus
Min(P)$ on $T^{'\vee}$, the set of
join-irreducible elements of $T'$. Clearly $f=f^{\vee}\circ i$.
Hence, if
$f ^{\vee}$ is injective, then $f$ too and Condition $1)$ is
satisfied. Moreover,
the image by $f ^{\vee}$ of $T^{'\vee}$ is the set of
join-irreducible elements of the sublattice
$f^{\vee}(T')$. Hence, the image of $P'$ by $f$ is the set of
join-irreducible members of the sublattice $f^{\vee}(T')$ . This
amounts to Condition $2)$. Conversely, if these  two conditions are
satisfied, then $f(P')$ is equal to the set of join-irreducible
elements of
$f^{\vee}(T')$ and this insures that $f^{\vee}$ is a
lattice-isomorphism from $T'$ onto $f^{\vee}(T')$.
\end{proof}
\begin{corollary}\label{translate}
If a meet-semilattice $P$ embeds into a meet-semilattice $Q$ as a
subsemilattice then $I_{0}
(P)$ embeds into $I_{0}(Q)$ as a sublattice.
\end{corollary}
\begin{proof}
Let $g: P \to Q$ be a one-to-one meet-preserving  map. Let
$T:=I_{0}(Q)$  and $f: P \to T$ defined by $f(x):=\downarrow \!\!g(x)$.
Conditions $1)$ and $2)$ in  $(ii)$  of Lemma \ref{lemII.5} are
satisfied, hence $f^{\vee} : I_{0}(P) \to T$ is a one-to-one
lattice-homomorphism.
\end{proof}
\begin{lemma}\label{lemII.11}
    Let $T$ be  a distributive lattice. If  the meet-subsemilattice of
$T$ generated by a subset $A$ of $T$ is isomorphic to $\Delta$,  to
    $\Gamma$ or to $V$, then
    the sublattice of $T$ generated by some infinite subset $A'$ of $A$
is either isomorphic to $I_{0}(\Delta)$,
    to $I_{0}(\Gamma)$ or to $[\omega]^{<\omega}$.
    \end{lemma}
\begin{proof}
Let $P$ one of the meet-semilattices $\Delta$,
    $\Gamma$ or $V$, and $f: P\rightarrow T$ be a one-to-one
meet-preserving map. According to $(i)$ of Lemma \ref
{lemII.5}, $f$ extends to a lattice-homomorphism $f^{\vee}:
I_{0}(P)\rightarrow T$. According to $(ii)$ of Lemma \ref
{lemII.5}, $f^{\vee}$ is one-to-one if for every $x\in P$, every
finite non-empty subset $X$ of  $P$, the equality
$f(x)=\bigvee f(X)$ implies $x\in X$. This condition is satisfied if
$P$ is $\Gamma$ or $V$  (indeed in this case, every
$x\in P\setminus Min (P)$ is {\it completely join-irreducible}, that
is $\{y\in P :y<x\}$ has a largest element). This is not
necessarily the case if $P:= \Delta$.  In this case, set
$g:\Delta\rightarrow \Delta$ defined by setting $g(i, j):=
(2i,\omega)$ if
$j=\omega$ and $g(i,j):=(2i,2j)$ if $j<\omega$ and $h:= f\circ g$.
Clearly $g$ is a one-to-one meet-preserving map, hence $h$
is a one-to-one meet-preserving map from $\Delta$ into $T$. Moreover
$h$ satisfies condition 2) of $(ii)$ of Lemma \ref
{lemII.5}. Indeed, let
$x:=(i,j)\in\Delta$ and $X$ be a finite subset of $\Delta$ such that
$h(x)=\bigvee
h(X)$. Since $h$ is an embedding, this entail that
$x=\bigvee X$ in $\Delta$. If $x\notin X$ then  $i=Max\{j' :
(i',j')\in X\}$ and $j=i+1$ . Since $f$ is an embedding, we have
$h(i',j')= f(2i',2j')\leq f(2i,2i +1)<f(2i,2j)=h(x)$ for all
$(i',j')\in X$, contradicting $h(x)=\bigvee
h(X)$. This insures that $h^{\vee}$ is  a one-to-one homomorphism
from
$I_{0}(\Delta)$  into $T$, the image being generated by the subset
$A':=\{h(i,\omega) : i<\omega\}$.
\end{proof}
\begin{lemma}\label{II.8}If $P$ is $\Gamma$ or $\Delta$, the lattices
$I_{0}(P)$ and $I_{<\omega}(P)$
embed in each other as sublattices. On an other hand,
    the lattices $I_{<\omega}(\Gamma)$ and $I_{<\omega}(\Delta)$
do not embed in each other as sublattices.
\end{lemma}
\begin{proof}
    If $P$ is $\Gamma$ or $\Delta$ then $1+P$ embeds into $P$ by a
meet-preserving map. From Corollary \ref{translate},
$I_{0}(1+P)$ embeds into $I_{0}(P)$ as a sublattice. Since
$I_{<\omega}(P)$  is isomorphic to $I_{0}(1+P)$ the desired
conclusion follows.\\
As a poset
$I_{<\omega}(\Delta)$ does not embed into
$I_{<\omega}(\Gamma)$. Indeed, the chain  $\omega +1$  embeds into
$\Delta$ hence into
$I_{<\omega}(\Delta)$ whereas it does not embed into
$I_{<\omega}(\Gamma)$ since
this poset embeds into
$[\omega]^{<\omega}$, a poset which does not embed  $\omega +1$.
On an other hand, $I_{<\omega}(\Gamma)$ embeds into
$I_{<\omega}(\Delta)$ as a poset, but not as a sublattice.
Indeed, in $\Gamma$ the antichain $\{(i, \omega) : i<\omega\}$ is such that
$( i, \omega)\wedge (j, \omega)=( i, \omega)\wedge (k, \omega)$
for $i<j<k<\omega$. Consequently, if  $I_{<\omega}(\Gamma)$ was embeddable
into  $I_{<\omega}(\Delta)$
 as a sublattice then $I_{<\omega}(\Delta)$ would contain an
antichain $\{U_{n} : n<\omega\}$
such that $U_{i}\cap U_{j}=U_{i}\cap U_{k}$ for all $i<j<k$. But this
equality does not hold  even for $i=0,j=1$ and all $k$.
Otherwise, if
$u\in U_{0}\setminus U_{1}$ , then $u\in U_{0}\setminus U_{k}$, for
all $k, k>1$. Hence,  $I_{<\omega}(\Delta
\setminus
\uparrow \!\!u)$  contains
the infinite antichain $\{U_{k} : 1<k\}$ from which it follows that
$u=(i,\omega)$ for some $i$ (indeed, if $u:=(i,j)$ with
$i<j<\omega$ then  , since
$\Delta
\setminus
\uparrow \!\!u$ is covered by finitely many chains, $I_{<\omega}(\Delta
\setminus
\uparrow \!\!u)$  contains no  infinite antichain). Let $v:= (i',j') <u$;
then $v\in U_{0}$ and, since $j'<\omega$,  the
previous argument gives $v\in U_{1}$. Since $U_{1}$ is finitely
generated, it contains some element $u_{1}$ above infinitely
many elements below $u$.   The structure of $\Delta$ imposes
$u_{1}=u$ which is impossible.
\end{proof}

\subsection{Proof of Theorem \ref{thm8}}

$(i) \Rightarrow (ii)$ Let $L$ be  an independent subset of $T$. From
Lemma \ref{lemII.2},
the sublattice $<L>$
generated by $L$ has $[L]^{<\omega}$ as a quotient. To conclude,
choose $L$ to be
countable.\\
$(ii)\Rightarrow (iii)$ If
$[\omega]^{<\omega}$ is a quotient of some sublattice of $T$ then,
according to Lemma \ref{lem II.4} there is a meet-preserving
map from
$\Delta$ into $T$ whose image contains an infinite antichain.
According to $(iii)$ of Theorem \ref {semiwqo},  $T$ contains either
$\Delta$, or
$\Gamma$ or  $V$ as a meet-subsemilattice, hence, from Lemma
\ref{lemII.11} and  Lemma \ref{II.8}, $T$ contains either
$I_{<\omega}(\Delta)$,
$I_{<\omega}(\Gamma)$ or  $[\omega]^{<\omega}$ as a sublattice.
Since
$[\omega]^{<\omega}$ has a sublattice isomorphic to
$I_{<\omega}(\Gamma)$ the conclusion
follows.\\ $(iii)\Rightarrow (i)$  $I_{<\omega}(\Gamma)$ and
$I_{<\omega}(\Delta)$
contain an infinite
independent set (namely $\{\downarrow(i, \omega) :
i<\omega\}$).\endproof

\chapter*{Notations, basic definitions and facts}
 \underline{Poset, qoset,
chain:}\\ If $(P,\leq)$ is a partially ordered set, shortly a  {\it poset},
we will often just write $P$ for $(P, \leq)$. We write $x\leq y$
for $(x,y)\in\leq$. A {\it qoset} is a quasi-ordered set and a
linearly ordered poset is a {\it
chain}.\\
\underline{Initial segment, principal, $I(P), I_{<\omega}(P), I_ 0
(P), \downarrow A$:}\\
A subset $I$ of $P$ is an {\it initial
segment} if $x\leq y$ and $y\in I$ imply $x\in I$. We denote by
$I(P)$ the set of initial segments of $P$ ordered by inclusion.
Let $A$ be a subset of $P$, then: $\downarrow A:=\{y\in P: y\leq
x$ for some $x\in A\}$. If $A$ contains only one element $a$, we
write $\downarrow a$ instead of $\downarrow \{a\}$. An initial
segment generated by a singleton is {\it principal} and it is {\it
finitely generated} if it is generated by a finite subset of $P$.
We denote by $I_{<\omega}(P)$ the set of finitely generated
initial segments and $I_ 0 (P):=I_ {<\omega}
(P)\setminus\{\emptyset\}$.\\
\underline{Up-directed, ideal,
$J(P)$, $\downarrow$-closed:}\\
A subset $I$ of $P$ is {\it
up-directed} if every pair of elements of $I$ has a common
upper-bound in $I$. An {\it ideal} is a non-empty up-directed
initial segment of $P$. We denote $J(P)$, the set of ideals of $P$
ordered by inclusion and we set $J_*(P):=J(P)\cup\{\emptyset\}$.
The poset $P$ is {\it $\downarrow$-closed} if the intersection of
two principal initial segments of $P$ is a finite union, possibly
empty, of principal initial segments.\\
 \underline{Dual, final
segment:}\\
The {\it dual} of $P$ is the poset obtained from $P$
by reversing the order; we denote it by $P^*$. A subset which is
an initial segment of $P^*$ will be called a final segment of
$P$.\\
\underline{Order-preserving, embedding,
order-isomorphism:}\\
 Let $P$ and $Q$ be two posets. A map
$f:P\rightarrow Q$ is {\it order-preserving} if $x\leq y$ in $P$
implies $f(x)\leq f(y)$ in $Q$; this is an {\it embedding} if
$x\leq y$ in $P$ is equivalent to $f(x)\leq f(y)$ in $Q$; if, in
addition, $f$ is onto, then this is an {\it order-isomorphism}.\\
\underline{Equimorphic posets, order type:}\\
 We say that $P$ {\it
embeds} into $Q$ if there is an embedding from $P$ into $Q$, a
fact we denote $P\leq Q$; if $P\leq Q$ and $Q\leq P$ then $P$ and
$Q$ are {\it equimorphic}, we denote $P\equiv Q$.  If there is an
order-isomorphism from $P$ onto $Q$ we say that $P$ and $Q$ are
{\it isomorphic} or have the same {\it order type},  a fact we denote
$P\cong Q$.\\
 \underline{$\omega, \omega^*, \eta$:}\\
  We denote
$\omega$ the order type of $\N$, the set of natural integers,
$\omega^*$ the order type of the set of negative integers and
$\eta$ the order type of $\Q$, the set of rational numbers.\\
\underline{Well-founded, well-quasi-ordered, well-ordered,
ordinal, scattered:}\\ A poset $P$ is {\it well-founded} if the
order type $\omega ^{*}$ does not embed into $P$. If furthermore,
$P$ has no infinite antichain then $P$ is {\it well-quasi-ordered}
 (wqo). A well-founded chain
is {\it well-ordered} ; its order type is an {\it ordinal}. A
poset $P$ is {\it scattered} if it does not contain a copy of
$\eta$, the chain of rational numbers.\\
 \underline{$\mathfrak
{P}(E)$, $[E]^{<\omega}$, $\mathcal{F}^{<\omega}$,
$\mathcal{F}^\cup$:}\\ Let $E$ be a set , we denote
$[E]^{<\omega}$ (resp. $\mathfrak {P}(E)$), the set, ordered by
inclusion, consisting of finite (resp. arbitrary) subsets of $E$.
If $\mathcal{F}$ is a subset of $\mathfrak P(E)$, we denote
$\mathcal{F}^{<\omega}$ (resp. $\mathcal{F}^\cup$) the collection
of finite (resp. arbitrary) unions of members of $\mathcal{F}$
ordered by inclusion.
\\ \underline{Lexicographic sum, ordinal sum,
lexicographic product:}\\ If $(P_{i})_{i\in I}$ is a family of
posets indexed by a poset $I$, the {\it lexicographic sum} of this
family is the poset, denoted $\Sigma_{i\in I}P_{i}$, defined on
the disjoint union of the $P_{i}$, that is formally the set of
$(i,x)$ such that $i\in I$ and $x\in P_{i}$, equipped with the
order $(i,x)\leq(j,y)$ if either $i<j$ in $I$ or $i=j$ and $x\leq
y$ in $P_{i}$. When $I$ is the finite chain
$n:=\{0,1,\ldots,n-1\}$ this sum is denoted
$P_{0}+P_{1}+\ldots+P_{n-1}$. When $I:=\omega$ this sum is denoted
$\Sigma_{i<\omega}P_{i}$ or $P_{0}+P_{1}+\ldots+P_{n}+\ldots$. We
denote $\Sigma^*_{i\in I}P_{i}$ for $\Sigma_{i\in I^*}P_{i}$; when
$I:=\omega$ we denote $\Sigma^*_{i<\omega}P_{i}$ or
$\ldots+P_{n}+\ldots+P_{1}+P_{0}$. When $I$ (resp. $I^*$) is well
ordered, or is an ordinal, we call  it {\it ordinal sum}
\index{ordinal sum} (resp. {\it antiordinal sum}) instead of
lexicographic sum. When all the $P_{i}$ are equal to the same
poset $P$, the lexicographic sum is denoted $P.I$, and called the
{\it lexicographic product} of $P$ and $I$.\\
 \underline{Direct
product,direct sum:}\\
 The {\it direct product} of $P$ and $Q$
denoted $P\times Q$ is the set of $(p, q)$ for $p\in P$ and $q\in
Q$, equipped with the product order; that is $(p,q)\leq (p',q')$
if $p\leq p'$ and $q\leq q'$. The {\it direct sum} of $P$ and $Q$
denoted $P\oplus Q$ is the disjoint union of $P$ and $Q$ with no
comparability between the elements of $P$ and the elements of $Q$
(formally $P\oplus Q$ is the set of couples $(x,0)$ with $x\in P$
and $(y,1)$ with $y\in Q$ equipped with the order $(p,q)\leq
(p',q')$ if $p\leq p'$ and $q=q'$).\\
\underline{Indecomposable,
right-indecomposable, left-indecomposable, indivisible:}\\ An
order type $\alpha$ is {\it indecomposable} if
$\alpha=\beta+\gamma$ implies $\alpha\leq\beta$ or
$\alpha\leq\gamma$; it is {\it right-indecomposable} if
$\alpha=\beta+\gamma$ with $\gamma\not=0$ implies
$\alpha\leq\gamma$; it is {\it strictly right-indecomposable} if
$\alpha=\beta+\gamma$ with $\gamma\not=0$ implies
$\beta<\alpha\leq\gamma$. The {\it left-indecomposability} and the
{\it strict
left-indecomposability} are defined in the same way. 
An order type $\alpha$ is {\it indivisible} if for every partition
of a chain $A$ of order type $\alpha$ into $B\cup C$, then either
$A\leq B$ or $A\leq C$.\\
\underline{Join and meet:}\\ The \emph{join} or  \emph{supremum} of a subset $X$ of a poset  $P$ is the least upper-bound of $X$ and is denoted $\bigvee X$. If $X$ is made of $x$ and $y$ it is denoted $x\vee y$. The \emph{meet} or \emph{infimum} of $X$ is denoted $\bigwedge X$. Similarly, $x\wedge y$ is the meet of $x$ and $y$.\\
\underline{Compact element:}\\ An element $x$ of a poset $P$ is \emph {compact} if $x\leq \bigvee X$ implies $x\leq \bigvee X'$ for some finite subset $X'$ of $X$.\\
\underline{Join-semilattice:}\\ A {\it join-semilattice} is a
poset $P$ such that  arbitrary elements $x, y$ have a join.\\
\underline{Join-preserving:}\\ Let $P$ and
$Q$ be two join-semilattices. A map $f:P\rightarrow Q$ is {\it
join-preserving} if $f(x\vee y)=f(x)\vee f(y)$ for every $x, y\in
P$.\\
\underline{Independent set:}\\
A subset $X$ of a join-semilattice $P$ is {\it
independent} if $x\not \leq \bigvee F$ for every $x\in X$ and
every non empty finite subset $F$ of $X\setminus \{x\}$.\\
\underline{Join-irreducible, join-prime, $\J_{irr}(P)$,
$\J_{pri}(P)$:}\\
An element $x$ of a join-semilattice $P$ is {\it join-irreducible} if
it is distinct from the least element (if any) and if $x=a\vee b$ implies $x=a$ or $x=b$.
We denote $\J_{irr}(P)$ the set of join-irreducible elements of
$P$.  An element $x\in P$ is  {\it join-prime}, if it is distinct
from the least element  (if any) and if
 $x\leq a\vee b$ implies
$x\leq a$ or $x\leq b$. This amounts to the fact that $P\setminus
\uparrow x$ is an ideal . We denote  $\J_{pri}(P)$, the set of
join-prime  members of $P$. We have $\J_{pri}(P)\subseteq
\J_{irr}(P)$.\\
\underline{Lattice, complete lattice:}\\ A lattice is a  poset $P$ in which every pair of elements has a join and a meet. If every subset has a join and a meet,  $P$ is a complete lattice.\\
\underline{Algebraic lattice:}\\ An \emph{algebraic lattice} is a complete lattice in which every element is a join of compact
 elements.\\
 \underline{Completely meet-irreducible, $\triangle
(L)$:}\\ Let $L$ be a complete lattice. For $x\in L$ , set
$x^{+}:=\bigwedge\{y\in L: x<y\}$. An element $x\in L$ is {\it
completely meet-irreducible} if $x=\bigwedge X$ implies $x\in X$,
or -equivalently- $x\neq x^{+}$. We denote $\triangle (L)$ the set
of completely meet-irreducible members of $L$.\\
\underline{Sierpinskization, $\Omega(\alpha)$:}\\ A {\it
sierpinskization} of a countable order-type $\alpha$ can be  obtained by
intersecting the natural order on the set $\N$
 of positive integers with  a linear order of $\alpha$.
Sierpinskisations given by a bijective map
$\psi: \omega \rightarrow \omega\alpha$ such that
$\varphi^{-1}$ is order-preserving
on each component $\omega \cdot \{i\}$ of $\omega\alpha$ are all
embeddable in each other, and for this reason  denoted by the same
symbol $\Omega(\alpha)$.\\ 
\underline{$L_{\alpha}$, $\mathbb J$, $\mathbb A$, $\mathbb L$,
$\J_{\neg\alpha}$, $\A_{\neg\alpha}$, $\mathbb L_{\neg\alpha}$:}\\
Let $\alpha$ be a chain, we denote by $L_{\alpha}:= 1+(1\oplus
\alpha)+1$ the lattice made of the direct sum of the one-element
chain $1$ and the chain $\alpha$, with top and bottom added. We
denote $\mathbb J$ the class of join-semilattices having a least
element, $\mathbb A$ the class of algebraic lattices, $\mathbb L$
the collection of $L\in\mathbb A$ such that $L$ contains no
join-semilattice isomorphic to $L_{\omega+1}$ or to $L_{\omega^*}$. We
denote $\J_{\neg\alpha}$ (resp. $\mathbb A_{\neg \alpha}$,
$\mathbb L_{\neg \alpha}$) the collection of $L\in\mathbb J$
(resp. $L\in\mathbb A$, $L\in\mathbb L$) such that $L$ contains no
chain of type $I(\alpha)$.\\
\underline{$K(L)$,
$\mathbb{K}$:}\\ Let $L$ be an algebraic lattice, we denote $K(L)$
the set of compact elements of $L$. We denote by $\mathbb{K}$ the
class of order types $\alpha$ such that $L \in \mathbb L_{\neg
\alpha}$ whenever $K(L)$ contains no chain of type $1+\alpha$ and
no subset isomorphic to $[\omega]^{<\omega}$.\\

\printindex

\end{document}